\documentclass[11pt]{amsart}
\usepackage{fullpage}   %comment this for transactions version

\usepackage{amsmath}
\usepackage{amsfonts}
\usepackage{amssymb}
\usepackage{amsthm}
\usepackage{subfigure} 
\usepackage{graphicx}
\usepackage{color}
\usepackage{lineno}
\usepackage{pinlabel}
\usepackage{tikz}
\usepackage{hyperref}
\usepackage{stmaryrd}
\newcommand*\patchAmsMathEnvironmentForLineno[1]{%
  \expandafter\let\csname old#1\expandafter\endcsname\csname #1\endcsname
  \expandafter\let\csname oldend#1\expandafter\endcsname\csname end#1\endcsname
  \renewenvironment{#1}%
     {\linenomath\csname old#1\endcsname}%
     {\csname oldend#1\endcsname\endlinenomath}}% 
\newcommand*\patchBothAmsMathEnvironmentsForLineno[1]{%
  \patchAmsMathEnvironmentForLineno{#1}%
  \patchAmsMathEnvironmentForLineno{#1*}}%
\AtBeginDocument{%
\patchBothAmsMathEnvironmentsForLineno{equation}%
\patchBothAmsMathEnvironmentsForLineno{align}%
\patchBothAmsMathEnvironmentsForLineno{flalign}%
\patchBothAmsMathEnvironmentsForLineno{alignat}%
\patchBothAmsMathEnvironmentsForLineno{gather}%
\patchBothAmsMathEnvironmentsForLineno{multline}%
}

\newtheorem{theorem}{Theorem}
\newtheorem{proposition}[theorem]{Proposition}
\newtheorem{lemma}[theorem]{Lemma}

\newtheorem{corollary}[theorem]{Corollary}

\theoremstyle{definition}
\newtheorem{definition}[theorem]{Definition}

\newtheorem{example}[theorem]{Example}
\newtheorem{remark}[theorem]{Remark}

\newcommand{\id}{\mathrm{id}}
\newcommand{\ba}{\setminus}

\newcommand{\del}{\bbslash}
\newcommand{\con}{\sslash}

\begin{document}

%\newcommand{\la}{\lambda}

%\linenumbers

\title[Hopf algebras and  Tutte polynomials]{Hopf algebras and  Tutte polynomials}

%\dedicatory{Working draft. Do not distribute!}

\author{Thomas Krajewski$^*$}
\author{Iain Moffatt$^*$}
\author{Adrian Tanasa$^\ddagger$}
\address{$^*$CPT - UMR 7332, CNRS, Aix-Marseille Universit\'{e} and Universit\'{e} de Toulon, Campus de Luminy, 13228 Marseille Cedex 9, France}
\email{thomas.krajewski@cpt.univ-mrs.fr}
\address{$^*$Department of Mathematics, Royal Holloway, University of London, Egham, Surrey, TW20 0EX, United Kingdom}
\email{iain.moffatt@rhul.ac.uk}
\address{$^\ddagger$Universit\'{e} de Bordeaux, LaBRI, UMR 5800, F-33400 Talence, France and Horia Hulubei National Institute for Physics and Nuclear Engineering P.O.B. MG-6 077125 Magurele, Romania and I. U. F. Paris, France}

\email{adrian.tanasa@ens-lyon.org}

\thanks{T.K. was partially supported by the ANR JCJC CombPhysMat2Tens grant.  A.T. was partially supported by the PN  09370102 and ANR JCJC CombPhysMat2Tens grants. This work was initiated while the authors were visiting the Erwin Schr\"odinger International Institute for Mathematical Physics (ESI) in Vienna during their 2014 programme on Combinatorics, Geometry and Physics. All three authors would like to thank the ESI for their support and for providing a productive working environment.}

\date{This version \today}

\begin{abstract}
By considering  Tutte polynomials of   Hopf algebras, we show how a Tutte polynomial can be canonically associated with  combinatorial objects that have some notions of deletion and contraction. We show that several graph polynomials from the literature arise  from this framework. These polynomials include the classical Tutte polynomial of graphs and matroids, Las~Vergnas' Tutte polynomial of the morphism of matroids and his Tutte polynomial for embedded graphs, Bollob\'as and Riordan's ribbon graph polynomial, the Krushkal polynomial, and the Penrose polynomial.

We show that our Tutte polynomials of Hopf algebras share common properties with the classical Tutte polynomial, including deletion-contraction definitions, universality properties, convolution formulas, and duality relations. New results for  graph polynomials from the literature are then obtained as examples of the general results.

Our results offer a framework for the study of the Tutte polynomial and its analogues in other settings,  offering the means to determine the properties and connections between a wide class of polynomial invariants. 
\end{abstract}

\maketitle

\section{Introduction and overview}\label{sec1}
The Tutte polynomial is arguably the most important  graph polynomial, and unquestionably the most studied. It encodes a substantial amount of the combinatorial information of a graph,  specialises to a myriad of other polynomials  (including the chromatic and  flow polynomials).  It appears in knot theory as the Jones and homfly-pt polynomials, and in statistical mechanics as the Ising and Potts model partition functions. 

Given the pervasiveness of the Tutte polynomial, it is unsurprising that attention has been given to finding analogues or extensions of the Tutte polynomial from graphs to other types of combinatorial object.
 These analogues can mostly be fit in to three broad types. Some analogues, such as W.T.~Tutte and H.~Crapo's extension to matroids \cite{Cr69,tuttethesis},  uncontroversially should be called a Tutte polynomial. Some analogues, such as M.~Las~Vergnas' Tutte polynomial for morphisms of matroids \cite{Las78a,Las80}, offer entirely satisfactory candidates for  a Tutte polynomial, but without an explanation of why we should use that particular polynomial and no other. Finally, some polynomials, such as the Bollob\'as-Riordan polynomial \cite{CMNR1} or G.~Farr's  polynomials of alternating dimaps \cite{Farr:2013aa}, offer polynomials that have some of the properties we would expect of a Tutte polynomial,  but do not have some of the other properties we would expect (for example, a ``full'' deletion-contraction definition in the case of the Bollob\'as-Riordan polynomial). 

Thus we arrive at the fundamental problem of what  we mean when we say that a polynomial invariant is a ``Tutte polynomial'' of some class of objects?  It is exactly this problem that we are interested in here. 

As an answer to this  problem, we propose a  Hopf algebraic framework for Tutte-like graph polynomials. This framework 
offers a canonical construction  of a ``Tutte polynomial'' of a   (suitable) set of combinatorial objects that is equipped with \emph{some} notions of ``deletion'' and ``contraction''. (We emphasise that these need not be the usual notions of deletion and contraction for the given objects. In fact, different  ``Tutte polynomials''   arise when using different notions of deletion and contraction for the same type of object.)  The resulting polynomials satisfy what we can reasonably expect an analogue of the Tutte polynomial to:
\begin{itemize}
\item Have a natural, canonical definition that arises from the class. 
\item Have a ``full'' recursive deletion-contraction-type definition terminating in trivial objects.
\item Have a  state-sum (rank-nullity-type) formulation.
\item Have a  universality property.
\item Share standard properties of the Tutte polynomial such as duality relations, and convolution formula when possible. 
\end{itemize}

The overall aim of this theory is to study graph polynomials en masse, rather than individually, and  we must question whether the framework offered here is useful for doing this.   Any general theory should: (i) absorb examples from the literature, and (ii) resolve problems.  The bulk of this paper is taken up verifying that the theory does indeed do this.

We show that various graph polynomials, including the classical Tutte polynomials of graphs, matroids, and morphisms of matroids, arise as canonical Tutte polynomials. We also show that other graph polynomials, such as the Bollob\'as-Riordan polynomial \cite{BR02}, the Krushkal~\cite{Kr}  and the surprisingly  Penrose polynomial~\cite{Aig97,Pe71}, arise as restricted versions of canonical Tutte polynomials.

As for resolving problems, we illustrate that the theory does this by considering topological Tutte polynomials.
Over the last few years there has been considerable interest in extensions of the Tutte polynomial to graphs embedded in surfaces.  The study of topological Tutte polynomials began, as far as the authors are aware, with M.~Las Vergnas' Tutte polynomial of the morphism of a matroid. By considering matroid perspectives associated with embedded graphs, in  \cite{Las78a,Las80}  (see also \cite{Las75,Las78}), he introduced a polynomial $L_G(x,y,z)$, since named the Las~Vergnas polynomial, that extends the classical Tutte polynomial to cellularly embedded graphs.  Unfortunately Las~Vergnas' polynomial did not gain much attention and it took several years for topological Tutte polynomials to attract the serious attention of the community. This attention was instigated by B.~Bollob\'as and O.~Riordan's papers \cite{BR02} and \cite{BR01} where they introduced a topological Tutte polynomial $R_G(x,y,z)$. This polynomial, which is usually described in the language of ribbon graphs, has attracted much attention and has found applications in knot theory and quantum field theory (see, for example, \cite{Da,EMMbook,KRT10} and the references therein). Most recently, motivated by the algebra and combinatorics of statistical mechanics, S.~Krushkal  introduced in \cite{Kr} a polynomial $K_G(x,y,a,b)$ that extends the Tutte polynomial (and the Bollob\'as-Riordan polynomial) to graphs that are (not necessarily cellularly) embedded in a surface. 

There are three problematic aspects to the theory of topological Tutte polynomials as it stands. This first is simply why are there three different ``Tutte polynomials'' for graphs in surfaces? Which can claim to be \emph{the} Tutte polynomial? Secondly, why do the polynomials not have full recursive deletion-contraction relations that terminate in trivial graphs in surfaces? Thirdly, why are almost all  results about topological graph polynomials in the literature restricted to the  2-variable specialisation of the Bollob\'as-Riordan polynomial $x^{\gamma(G)/2}   R_G(x+1,y,1/\sqrt{xy})$? Answering these three questions was the motivation behind this work. An answer given by  the Hopf algebraic framework is offered in Remark~\ref{adhj}.

This paper is structured as follows. Section~\ref{hjdas} introduces the Tutte polynomial of a Hopf algebra, canonical Tutte polynomials, and lists some examples of these. 
Section~\ref{s.app} proves that these polynomials have desirable properties, including full deletion contraction-definitions, state sum formulations, universality properties, convolution formulae, and that specialisation and duality results arise from maps at the Hopf algebra level.
Section~\ref{s.examp} provides full details for the canonical Tutte polynomials that was summarised Section~\ref{hjdas}.   

\section{The definition of a Tutte polynomial of a Hopf algebra}\label{hjdas}

\subsection{The general case of Hopf algebras}
Let $ \mathcal{H} =  \bigoplus_{i\geq 0}\mathcal{H}_i $ be a graded connected commutative Hopf algebra (i.e., $\mathcal{H}$ is graded Hopf algebra with $\mathcal{H}_0$ of dimension 1), where each $\mathcal{H}_i$  is a vector space over $\mathbb{Q}$. (We work here over $\mathbb{Q}$ for simplicity, although it is possible to work in a more general setting.) All Hopf algebras here are commutative. 
An element  $S\in \mathcal{H} $ is said to be of  \emph{graded dimension} $i$ if it is in $\mathcal{H}_i$.
 If $f$ and $g$ are mappings from $\mathcal{H}$ into some commutative algebra with product $m$, then their \emph{convolution product}, $f\ast g$, is the mapping from $\mathcal{H}$ defined by $f\ast g:=m\circ (f\otimes g) \circ \Delta$.

Let  $\{S_i\}_{i\in I}$ be a basis for $\mathcal{H}_1$.  For each $i\in I$ we  define the mapping $\delta_i:\mathcal{H} \rightarrow \mathbb{Q}$ to be  the linear extension of 
\begin{equation}\label{e.dd}
   \delta_i(S) := \begin{cases} 1 &\mbox{if } S=S_i ,\\   0 & \mbox{otherwise} .\end{cases} 
   \end{equation}
Let $\{ x_j  \}_{j\in J}$ be a set of indeterminates,  $a_i\in\mathbb{Q}[\{ x_j  \}_{j\in J}]$ for each  $i\in I$, and $\mathbf{a}=\{a_i\}_{i\in I}$. We define the \emph{selector} $\delta_{\mathbf{a}} :  \mathcal{H} \rightarrow \mathbb{Q}[\{ x_j  \}_{j\in J}]$ by 
\begin{equation}\label{e.dd2}
 \delta_{\mathbf{a}}:=\sum_{i\in I} a_i\delta_i    .
 \end{equation}
 Similarly, for a set of indeterminates $\{ y_j  \}_{j\in J}$, set  $\mathbf{b}=\{b_i\}_{i\in I}$  with each   $b_i\in\mathbb{Q}[\{ y_j  \}_{j\in J}]$,  we define $\delta_{\mathbf{b}}:=\sum_{i\in I} b_i\delta_i $ .

With this choice of $\delta_{\mathbf{a}}$ (or $\delta_{\mathbf{b}}$) we can consider its $\ast$-exponential: 
\begin{equation}\label{e.dd3}
\exp_*(\delta_{\mathbf{a}}) =   \sum\limits_{m \geq 0}   \frac{\delta_{\mathbf{a}}{}^{\ast m}}{m!}  =  \epsilon + \delta_{\mathbf{a}} + \frac{1}{2}( \delta_{\mathbf{a}}\ast \delta_{\mathbf{a}} )+ \cdots    ,
 \end{equation}
where $\epsilon$ is the Hopf algebra counit.

We now introduce the Tutte polynomial of a Hopf algebra.
\begin{definition}\label{d.TH}
Let $\mathcal{H}$,  $\delta_{\mathbf{a}}$  and $\delta_{\mathbf{b}}$ be as above. Then we define the  \emph{Tutte polynomial} of $\mathcal{H}$,  $\alpha( \mathbf{a}, \mathbf{b}): \mathcal{H} \rightarrow  \mathbb{Q}[\{ x_j ,y_j \}_{j\in J}] $ by 
\[ \alpha( \mathbf{a}, \mathbf{b})  := \exp_* (\delta_{\mathbf{a}}) \ast \exp_* (\delta_{\mathbf{b}}) . \]
\end{definition}

At this point the reader may find it instructive to look forward to Example~\ref{ex.tutte}, which shows a  computation of Tutte polynomial of a Hopf algebra of graphs.

Before we continue (and, in particular, justify why we name $\alpha$ the Tutte polynomial)  we say a few words about  notation. In our examples,  $\mathcal{H}_1$ will  have a small dimension (of 2 to 5 elements) so we will usually fix an order of the basis and specify $\mathbf{a}$ and $\mathbf{b}$ as vectors, and  the pair $\mathbf{a}, \mathbf{b}$ as a list in $ \alpha( \mathbf{a}, \mathbf{b})$. We will also use a similar notation to specify $\delta_{\mathbf{a}}$. This will both reduce clutter and make better contact with standard graph polynomial notation.  Furthermore, we will often only define $\delta_{\mathbf{a}}$ with the understanding that $\delta_{\mathbf{b}}$ is defined similarly. Its exact definition will be clear from context.

\subsection{Deletion, contraction, and canonical Tutte polynomials}

Our aim here is to show that the  general definition of the Tutte polynomial of a Hopf algebra,   $\alpha( \mathbf{a}, \mathbf{b})$, provides a framework for studying a large class of graph polynomials. To this end we now work towards identifying how a variety of graph polynomials arise canonically as the Tutte polynomial of a Hopf algebra. More strongly, we  show that given a set of combinatorial objects with some notion of deletion and contraction, we can use Definition~\ref{d.TH} to obtain a natural and canonical Tutte polynomial for that class of objects. We will go on to identify numerous graph polynomials as the   canonical Tutte polynomial of an appropriate class.

\begin{definition}\label{d.1}
A \emph{minor system} consists of the following.
\begin{enumerate}
\item A graded set $\mathcal{S}=\bigcup_{n\geq 0}  \mathcal{S}_n$ of finite combinatorial objects such that each $S\in \mathcal{S}_n$ has a finite set $E(S)$ of exactly $n$ sub-objects associated with it, and such that there is a unique element $1\in \mathcal{S}_0$. 

\item Two \emph{minor operations}, $\del$  called \emph{deletion} and $\sslash$ called \emph{contraction}, that associate  elements $S\del  e $ and $S\con  e$, respectively, to each pair $(S  \in \mathcal{S}_n, e\in E(S))$, where  $E(S\del  e)=E(S\con e)=E(S)\setminus e$, and such that  for $e\neq f$ 
\[   (S \del e) \del f = (S \del f) \del e , \quad   (S \con e) \del f = (S \del f) \con e , \quad  (S \con e) \con f = (S \con f) \con e.\]
\end{enumerate}
\end{definition}

An example of a minor system is the set of matroids, with $E(S)$ the cardinality of the ground set of a matroid $S$, and with the usual  deletion and contraction  of matroids. Other examples can be found in Section~\ref{s.examp}.

We say that $S'$ is a \emph{minor} of $S$, if $S'$ can be obtained from $S$ by a sequence of applications of the minor operations.  By definition, if $S$ is in a minor system, then so are all of its minors.
Since the order of the application of the minor operations to distinct elements of $E(S)$ does not matter, we can use the notation $S\del A$ and $S\con A$ to mean we apply the appropriate minor operation to all of the elements in $A\subseteq E(S)$ in some order.

It is  a fairly routine exercise to verify that minor systems have a natural Hopf algebra structure:
\begin{proposition}\label{p.1}
The vector space $\mathcal{H}$ of formal $\mathbb{Q}$-linear combinations of elements  of a minor system $\mathcal{S}$ forms a coalgebra with counit  under
\[  \Delta(S)=\sum_{A\subseteq E(S)}(S\del A^c  ) \otimes  (S\con A),   \quad \quad    \varepsilon(S) = \begin{cases} 1 &\mbox{if } S\in \mathcal{S}_0, \\ 
0 & \mbox{otherwise} .\end{cases}  \]  
If, in addition, the vector space forms a commutative algebra with multiplication $m$ and unit $\eta$ such that  $\eta(1)\in \mathcal{S}_0$; 
  for all $S_1,S_2\in \mathcal{S}$
  \[ E( m( S_1\otimes S_2 ))  =  E(S_1) \sqcup E(S_2),\]
 and for each $A_i\subseteq E(S_i)$ 
\[m( S_1\del A_1 \otimes S_2 \del A_2 ) = m( S_1 \otimes S_2)  \del (A_1 \sqcup A_2 ), \quad
m( S_1\con A_1 \otimes S_2 \con A_2 ) = m( S_1 \otimes S_2)  \con (A_1 \sqcup A_2 ),
    \]
 then it is a graded connected Hopf algebra. 
\end{proposition}
Hopf algebras have a long history in combinatorics, starting with G.-C. Rota \cite{MR540015}, and S. Joni and G.-C. Rota \cite{MR544721}. We do not attempt to give a comprehensive survey of their use here, but do make a few comments. Various instances of the Hopf algebras defined in Proposition~\ref{p.1} are very well-known and well-studied. Of particular relevance here is that the deletion-contraction Hopf algebras of Proposition~\ref{p.1} (which includes deletion-restriction Hopf algebras upon choosing contraction to be deletion) are much studied in matroid theory. They appear in W.~Schmitt's article \cite{Sch94}, and have been used to study graph polynomials (see, for example,  \cite{MR1654184, MR2037633,MR2718681,MR1778203} for a selection of applications). 
It is also worth noting that the Hopf algebras for ribbon graphs are closely related to those arising in the theory of Vassiliev invariants  \cite{MR1318886}.

We call a Hopf algebra of the type described in Proposition~\ref{p.1} the \emph{Hopf algebras of the minor system} $\mathcal{S}$.

\bigskip

We say that a selector $\delta_{\mathbf{a}}$ is \emph{uniform} if, for each $S\in \mathcal{H}$,  the evaluations of $\delta^{\otimes m}$ for each summand of $\Delta^{(m-1)}(S)$ are equal. (Equivalently,  $\delta_{\mathbf{a}}$ is uniform if  $\delta^{\otimes m}$ is a well-defined map on the symmetric algebra $\mathfrak{S}^m(\mathcal{H}_1)$ for each $m$.)

\begin{definition}
Let  $\mathcal{H}$  be a Hopf algebra of a minor system  $\mathcal{S}$, and $\delta_{\mathbf{a}}$ and $\delta_{\mathbf{b}}$ be uniform selectors, where the $\delta_i$ are determined by the elements of $\mathcal{S}_1$. Then we say that  $\alpha( \mathbf{a}, \mathbf{b})$, as given in Definition~\ref{d.TH}, is a
\emph{canonical Tutte polynomial} of the minor systems $S$.
\end{definition}

\subsection{A summary of examples of canonical Tutte polynomials}\label{cxty}
In this section we give an overview of  how various polynomials arise as  canonical Tutte polynomials. The level of detail we give in this summary is just enough to understand the applications and properties given in the next section. Full details are given in  Section~\ref{s.examp}. Of course some of the polynomials that arise are generalisations of others, however recall that the aim here is to find the correct notion of a Tutte polynomial for a given setting rather than to construct the most general polynomial possible.

\subsubsection*{Matroids, $\mathcal{H}^m$}
(See Section~\ref{s.mat} for details.)

\begin{itemize}

\item \textit{Objects:} Matroids with their usual deletion and contraction. (See Definition~\ref{d.hm}.)

\item \textit{Selector:} $\delta_{\mathbf{a}}= x_1\delta_c+x_2\delta_l$, $\delta_{\mathbf{b}}= y_1\delta_c+y_2\delta_l$ where $\delta_c$ detects $U_{1,1}$ and $\delta_l$ detects $U_{0,1}$.  (See \eqref{t.mat.delta}.)

\item \textit{Canonical Tutte polynomial:} (See Theorem~\ref{t.mat})
\[ \alpha(  \mathbf{a}, \mathbf{b}) ( M )  =    x_1^{r(M)} y_2^{|E(M)|-r(M)}    T_{M}\left( \frac{y_1}{x_1}+1, \frac{x_2}{y_2}+1\right),\]
where $T_M(x,y)$ is  the classical Tutte polynomial of a matroid, \eqref{e.tutteM}.
\end{itemize}

\subsubsection*{Matroid Perspectives, $\mathcal{H}^{mp}$}
(See Section~\ref{ssec.mp} for details.)
\begin{itemize}

\item \textit{Objects:} Matroid Perspectives with their usual deletion and contraction.  (See Definition~\ref{d.hm}.)

\item \textit{Selector:} $\delta_{\mathbf{a}}= x_1\delta_{cc} + x_2\delta_{ll} +x_3\delta_{cl}$, $\delta_{\mathbf{b}}= y_1\delta_{cc} + y_2\delta_{ll} +y_3\delta_{cl}$ where  
$\delta_{cc}$, $\delta_{ll}$, $\delta_{cl}$ detect, respectively, 
  $U_{0,1} \rightarrow U_{0,1}$, $U_{1,1} \rightarrow U_{1,1}$, and $U_{1,1} \rightarrow U_{0,1}$ (See \eqref{e.demp}.)

\item \textit{Canonical Tutte polynomial:} (See Theorem~\ref{s2.t1}.)
\[ \alpha( \mathbf{a}, \mathbf{b}) ( \mathbf{M} )  =     x_1^{r'(M')} y_2^{|E|-r(M)} x_3^{r(M)-r'(M')}     T_{\mathbf{M} }\left( \frac{y_1}{x_1}+1, \frac{x_2}{y_2}+1, \frac{y_3}{x_3}  \right), \]
where $T_{\mathbf{M}} (x,y,z)$ is  Las~Vergnas' Tutte polynomial of  matroid perspectives, \eqref{d.tpmp}.
 \end{itemize}

\subsubsection*{Delta-matroids, $\mathcal{H}^{dm}$}
(See Section~\ref{s.BRdm} for details.)
\begin{itemize}

\item \textit{Objects:} Delta-matroids with their usual deletion and contraction. (See Definition~\ref{d.hopfdm}.)

\item \textit{Selector:} 
$\delta_{\mathbf{a}}= x_1\delta_{c} + x_2\delta_{o} +\sqrt{x_1x_2}\delta_{n}$, 
$\delta_{\mathbf{b}}= y_1\delta_{c} + y_2\delta_{o} +\sqrt{y_1y_2}\delta_{n}$, where $\delta_{c} $, $\delta_{o} $ and $\delta_{n}$ detect, respectively, 
$D_c:=( \{e\}, \{  \{e\} \})$, $D_o:=( \{e\}, \{  \emptyset \})$, and $D_n:=( \{e\}, \{  \emptyset , \{e\} \})$
  (See \eqref{e.BRdelta1}.)

\item \textit{Canonical  Tutte polynomial:} (See Theorem~\ref{s2.t3}.)
\[ \alpha( \mathbf{a}, \mathbf{b}) ( D)  =  x_1^{ \rho(D)}   y_2^{|E|- \rho(D)} \tilde{R}_D\left(\frac{y_1}{x_1}+1, \frac{x_2}{y_2} +1\right),\]
where $\tilde{R}_D(x,y)$ is the (2-variable) Bollob\'as-Riordan polynomial, \eqref{s2.e3}.
\end{itemize}

\subsubsection*{Delta-matroids, $\mathcal{H}^{pe}$}
(See Section~\ref{ss.penrose} for details.)
\begin{itemize}

\item \textit{Objects:} Delta-matroids with  the operations $D/e$ and $(D+e)/e$ as deletion and contraction. See Definition~\ref{pe.d1}. (Note that using operations $D\ba e$ and $(D+e)/e $ results in an equivalent polynomial, as in Theorem~\ref{dual6}.)

\item  \textit{Selector:} Same as that above for  $\mathcal{H}^{dm}$. 

\item \textit{Canonical  Tutte polynomial:} (See Theorem~\ref{s3.t1}.)
 \[\alpha( \mathbf{a}, \mathbf{b}) ( D)  =  x_1^{ \xi(D)}   y_2^{|E|- \xi(D)} \tilde{P}_D\left(\frac{y_1}{x_1}+1, \frac{x_2}{y_2} +1\right).\]
where  $\tilde{P}_D(x,y)$ is the 2-variable Penrose polynomial,  \eqref{s3.e10}.
\end{itemize}

\subsubsection*{Graphs, $\mathcal{H}^{g}$}
(See Section~\ref{s.gr}.)
\begin{itemize}

\item \textit{Objects:} A quotient space of graphs with their usual deletion and contraction. (See Definition~\ref{d.hg}.)

\item \textit{Selector:}    
$ \delta_{\mathbf{a}} = x_1\delta_b+x_2\delta_l$,
$ \delta_{\mathbf{b}} = y_1\delta_b+y_2\delta_l$,
where $\delta_b$ and $\delta_l$ detect a one-edge bridge and loop, respectively.
(See \eqref{g.selector}.)

\item \textit{Canonical  Tutte polynomial:} (See Theorem~\ref{t.gr}.)
\[ \alpha(  \mathbf{a}, \mathbf{b}) ( G )  =    x_1^{r(G)} y_2^{|E(G)|-r(G)}    T_{G}( \tfrac{y_1}{x_1}+1, \tfrac{x_2}{y_2}+1),
\]
where $T_G(x,y)$ is the classical Tutte polynomial of a graph,  \eqref{e.tutte}.
\end{itemize}

\subsubsection*{Graphs in pseudo-surfaces, $\mathcal{H}^{ps}$}
(See Section~\ref{ssec.lv} for details.)
\begin{itemize}

\item \textit{Objects:} A quotient space of graphs in pseudo-surfaces (see Definition~\ref{s2.d2}), deletion deletes edges, contraction forms topological quotient space.  (See Definition~\ref{s2.d3}.)

\item \textit{Selector:}   
$\delta_{\mathbf{a}}= x_1\delta_{cc} + x_2\delta_{ll} +x_3\delta_{cl} $,
$\delta_{\mathbf{b}}= y_1\delta_{cc} + y_2\delta_{ll} +y_3\delta_{cl} $,
where $\delta_{cc}$, $\delta_{ll}$ and $\delta_{cl}$, detect 
a 1-path in the sphere, a loop in the sphere, and a loop that is a meridian of a torus, respectively. 
(See \eqref{e.deps}.)

\item \textit{Canonical  Tutte polynomial:} (See Theorem~\ref{s2.c1}.)
 \[\alpha( \mathbf{a}, \mathbf{b}) ( G )  =   x_1^{r(G)} y_2^{\kappa(G)} x_3^{n(G)-\kappa(G)}     L_{G\subset \Sigma}( \tfrac{y_1}{x_1}+1, \tfrac{x_2}{y_2}+1, \tfrac{y_3}{x_3}  ) ,\]
where $L_{G\subset \Sigma}(x,y,z) $ is the Las~Vergnas polynomial, \eqref{LVeq4}. 
\end{itemize}

\subsubsection*{Ribbon graphs, $\mathcal{H}^{rg}$}
(See Section~\ref{s.BRrg} for details.)
\begin{itemize}
\item \textit{Objects:} A quotient space of ribbon graphs, with the usual deletion and contraction of ribbon graphs. (See Definition~\ref{d.hopfrg}.)

\item \textit{Selector:}  
$ \delta_{\mathbf{a}}= x_1\delta_{b} + x_2\delta_{o} +\sqrt{x_1x_2}\delta_{n}$, 
$\delta_{\mathbf{b}}= y_1\delta_{b} + y_2\delta_{o} +\sqrt{y_1y_2}\delta_{n}$,
where $\delta_{b}$,  $\delta_{o}$, and $\delta_{n}$, detect a one-edge bridge, orientable loop and non-orientable loop, respectively.
 (See \eqref{e.rgbr1}.)

\item \textit{Canonical  Tutte polynomial:} (See Theorem~\ref{t.rgbr1}.)
 \[
 \alpha( \mathbf{a}, \mathbf{b}) ( G)  = x_1^{ \rho(G)}   y_2^{|E|- \rho(G)} \tilde{R}_G\left(\frac{y_1}{x_1}+1, \frac{x_2}{y_2} +1\right),
 \]
where $\widetilde{R}_G( x,y)$ is the 2-variable Bollob\'as-Riordan polynomial,  \eqref{gjs}. Note that 
this equals $R_G(x,y-1,1/\sqrt{x(y-1)})$, where $R_G(x,y,z)$ is the  Bollob\'as-Riordan polynomial, as in \eqref{dhdf}.
\end{itemize}

\subsubsection*{Ribbon graphs, $\mathcal{H}^{per}$}
(See Section~\ref{s.penrg} for details.)
\begin{itemize}

\item \textit{Objects:} A quotient space of ribbon graphs, with ribbon graph contraction as deletion, and twist-contract as contraction. (See Definition~\ref{d.hopperg}.)

\item \textit{Selector:} Same as for  $\mathcal{H}^{rg}$ above. (See \eqref{e.rgbr1}.)

\item \textit{Canonical  Tutte polynomial:} (See Theorem~\ref{t.rgpe1}.)
\[ \alpha( \mathbf{a}, \mathbf{b}) ( G)  =  x_1^{ \xi(D)}   y_2^{|E|- \xi(D)} \tilde{P}_D\left(\frac{y_1}{x_1}+1, \frac{x_2}{y_2} +1\right), \]
where  $\tilde{P}_G(x,y)$  is the 2-variable Penrose polynomial of \eqref{yisnf}.
\end{itemize}

\subsubsection*{Vertex partitioned ribbon graphs $\mathcal{H}^{vrg}$}
(See Section~\ref{dasg} for details.)
\begin{itemize}

\item \textit{Objects:} A quotient space of ribbon graphs equipped with a partition of their vertex set. (See Definition~\ref{d.hopfvrg}.)

\item \textit{Selector:}   
$\delta_{\mathbf{a}}= x_1\delta_{b} + x_2\delta_{o} +\sqrt{x_2x_3}_3\delta_{n}+x_3\delta_{l}$,
$\delta_{\mathbf{b}}= y_1\delta_{b} + y_2\delta_{o} +\sqrt{y_2y_3}_3\delta_{n}+y_3\delta_{l}$
where the $\delta_i$ are as in \eqref{e.prgic}).

\item \textit{Canonical  Tutte polynomial:} (See Theorem~\ref{t.vbr1}.)
\[ \alpha( \mathbf{a}, \mathbf{b}) ( G,\mathcal{P})  =      y_1^{r(G_{/\mathcal{P}})}    y_2^{|E|-\rho(G)}    y_3^{\rho(G)-r(G_{/\mathcal{P}})}   
  \sum_{A\subseteq E}   \left(\frac{x_1}{y_1}\right)^{r(A_{/\mathcal{P}})}  \left(\frac{x_2}{y_2}\right)^{|A|-\rho(A)} \left(\frac{x_3}{y_3}\right)^{\rho(A)-r(A_{/\mathcal{P}})}.\]
This is  an extension of the Bollob\'as-Riordan polynomial   to vertex partitioned graphs. It coincides with the Bollob\'as-Riordan polynomial $R_G(x,y,z)$ when the blocks of the partition are of size one. 
\end{itemize}

\subsubsection*{Vertex partitioned graphs in surfaces, $\mathcal{H}^{vgs}$}
(See Section~\ref{xpdg} for details.)
\begin{itemize}

\item \textit{Objects:} A quotient space of graphs in surfaces equipped with a partition of their vertex sets (Definition~\ref{d.vgs2}). (See Definition~\ref{d.vgs1}.)

\item \textit{Selector:}   
$\delta_{\mathbf{a}}= x_1\delta_{1} + x_2\delta_{2} +x_3\delta_{3}+\sqrt{x_3x_4}\delta_{4}+x_4\delta_{5}$, 
$\delta_{\mathbf{b}}= y_1\delta_{1} + y_2\delta_{2} +y_3\delta_{3}+\sqrt{y_3y_4}\delta_{4}+y_4\delta_{5}$, 
where the $\delta_i$ are as in \eqref{dhgads}).

\item \textit{Canonical  Tutte polynomial:} (See Theorem~\ref{t.vkr1}.)
\begin{multline*}
 \alpha( \mathbf{a},\mathbf{b}) ( G\subset\Sigma,\mathcal{P})  =    
   y_1^{r(G_{/\mathcal{P}})} 
    y_2^{\kappa(G\subset\Sigma,\mathcal{P})}  
         y_3^{\rho(G\subset\Sigma,\mathcal{P})-r(G_{/\mathcal{P}})}   
    y_4^{  |E| - \rho(G\subset\Sigma,\mathcal{P})-\kappa(G\subset\Sigma,\mathcal{P})}   
\\
\\\sum_{A\subseteq E}  
 \left(\frac{x_1}{y_1}\right)^{r(A_{/\mathcal{P}})} 
  \left(\frac{x_2}{y_2}\right)^{\kappa(A)} 
        \left(\frac{x_3}{y_3}\right)^{\rho(A)-r(A_{/\mathcal{P}})}     
   \left(\frac{x_4}{y_4}\right)^{  |A| - \rho(A)-\kappa(A)}   
      .\end{multline*}
This is  an extension of the Krushkal polynomial  to vertex partitioned graphs in surfaces. It coincides with the Krushkal polynomial $K_{G\subset \Sigma}(a,b,x,y)$ when the blocks of the partition are of size one. 

\end{itemize}

\section{Properties of canonical Tutte polynomials}\label{s.app}

Throughout this section we work primarily with Hopf algebras of a minor system. We will show that canonical Tutte polynomials of minor systems, in general, have many of the desirable properties of the classical Tutte polynomial of a graph or matroid. 
The summary given in Section~\ref{cxty} provides enough information to interpret the applications of the general results to those examples. However, in order to bring the general results `up front',  at times we refer forward to later sections for full details.

\subsection{State sum formulations}
The following theorem provides state sum formulations for canonical Tutte polynomials. 

\begin{theorem}\label{tm} 
Let  $\mathcal{H}$  be a  Hopf algebra of a minor system $\mathcal{S}$ with coproduct $ \Delta(S)=\sum_{A\subseteq E(S)}(S\del A^c  ) \otimes  (S\con A) $. 
Suppose that a set $I$ indexes the elements of $\mathcal{H}_1$, and that the functions $\delta_i$ are defined by Equation~\eqref{e.dd}.
Suppose also that for each $j$ in some indexing set $J$ there is a function $r_j:\mathcal{H}\rightarrow \mathbb{Q}$ 
such that 
\begin{equation}\label{tm.e1}
r_j(S) =  r_j(S\con e) +m_{ij} \quad \text{ when } \delta_i(S \del e^c)=1,
\end{equation}
where $S\in \mathcal{H}$, $e\in S$ and $m_{ij}\in \mathbb{Q}$; and such that 
$r_j(S) = 0$  when $S\in\mathcal{H}_0$.

For a set of indeterminates $\{x_j\}_{j\in J}$ define  
\begin{equation}\label{tm.e2}
\delta_{\mathbf{a}}:= \sum_{i\in I} a_i\delta_i  \quad \text{where} \quad a_i:=\prod_{j\in J} x_j^{m_{ij}}. 
\end{equation}
Then $\delta_{\mathbf{a}}$ is uniform. Moreover, if  $\delta_{\mathbf{b}}:= \sum_{i\in I} b_i\delta_i $ with $b_i:=\prod_{j\in J} y_j^{m_{ij}}$, the Tutte polynomial of $\mathcal{H}$ satisfies
\begin{equation}\label{tm.e3}
 \alpha( \mathbf{a}, \mathbf{b})(S)   =     \prod_{j\in J}      y_j^{r_j(S)} \sum_{A\subseteq E(S)}   \prod_{j\in J}  \left(\frac{x_j}{y_j}\right)^{r_j(A)},
\end{equation}
 where   $r_i(A):=  (S\del A^c)$.
\end{theorem}
%%%
\begin{proof}
The proof of the theorem has four main steps: (1) showing $\delta_{\mathbf{a}}$ is uniform, (2) finding a closed form for $\exp_* (\delta_{\mathbf{a}})(S)$, (3) showing that for each $j$,  $r_j(S)= r_j(S\del A^c) + r_j(D\con A)$, and (4) proving the  given form of $\alpha( \mathbf{a}, \mathbf{b})(S)$. 

We start by showing $\delta_{\mathbf{a}}$ is uniform. For this we set up some notation. For $S\in \mathcal{S} \subset \mathcal{H}$ we use $P_l(S)$,  where $l=1,\ldots |E(S)|!$, to denote  the summands of $\Delta^{(|E(S)|-1)}(S)$ that consist of the tensor product of  $|E(S)|$ objects each of which is of  graded dimension 1. In addition let $\#_i(P_l(S))$ denote the number of tensor factors in $P_l(S)$  which $\delta_i$ maps to 1. 

To prove the uniformity of $\delta_{\mathbf{a}}$  we need to show that for each $l$ and $S$, $\delta_{\mathbf{a}}^{\otimes l}$ takes the same value on each summand of $\Delta^{(l-1)}(S)$. If any tensor factor in the  summand is not of graded dimension 1 then  $\delta_{\mathbf{a}}$ will evaluate to zero on that summand, thus we need only consider summands in which each tensor factor is of graded dimension 1. That is, we need to show that for each $S\in\mathcal{S}$,  $\delta^{\ast |E(S)|}$  takes the same value on $P_l(S)$, for each $l$. To do this we show 
\begin{equation}\label{tm.e4}
\delta^{\ast |E(S)|}   (P_l(S)) =   \prod_{j\in J} x_j^{r_j(S)}. 
\end{equation}
By definition we have 
 \begin{equation}\label{tm.e5}
\delta^{\ast |E(S)|}   (P_l(S)) =  \prod_{i\in I} a_i^{\#_i(P_l(S))}   =  \prod_{j\in J}    x_j^{   \sum_{i\in I }  m_{ij}  \cdot \#_i(P_l(S))  }.
\end{equation}
So we need to show for each $j\in J$ that
 \begin{equation}\label{tm.e6}
r_j(S) =      \sum_{i\in I }  m_{ij} \cdot  \#_i(P_l(S)) .
\end{equation}
We will do this by induction on $|E(S)|$.   If  $|E(S)|=0$ the result holds since both sides of \eqref{tm.e6} are trivial. 
$|E(S)|=1$ then for exactly one $k\in I$, $\delta_k(S)=1$ and so  
$ \sum_{i\in I }  m_{ij} \cdot  \#_i(P_l(S))  = m_{kj} = r_j(S)$.  

Now suppose that  $S\in\mathcal{S}$ with $|E(S)|\geq 2$, and that \eqref{tm.e6} holds for all $S'\in \mathcal{S}$ with $|E(S')|<|E(S)|$.   
We can write 
 \begin{equation}\label{tm.e7} 
 P_l(S) ={Q}_1 \otimes {Q}_2 \otimes {Q}_3 \otimes \cdots \otimes {Q}_{|E(S)|},
  \end{equation}
where each ${Q}_k$ is of graded dimension 1.
Observe that we can write 
 \begin{equation}\label{tm.e8} 
 P_n(S\con e_1) = {Q}_2 \otimes {Q}_3 \otimes \cdots \otimes {Q}_{|E(S)|},
  \end{equation}
for some $n$ and that ${Q}_1= S \del{e_1}^c $.
By the inductive hypothesis
 \begin{equation}\label{tm.e16} 
 r_j(S\con e_1) =      \sum_{i\in I }  m_{ij} \cdot  \#_i(P_n(S\con e_1)),
 \end{equation}
for each $j\in J$.

We know that $\delta_{p}(S \del{e_1}^c ) =1$ for some $p$ and is zero otherwise. Using Equation~\eqref{tm.e8} for the first equality, the inductive hypothesis for the second, and  Equation~\eqref{tm.e1}  for the third, we have
 %\begin{align}\label{tm.e9} 
\[ \sum_{i\in I }  m_{ij} \cdot  \#_i(P_l(S))   
     m_{pj} + \left( \sum_{i\in I }  m_{ij} \cdot  \#_i(P_n(S\con e_1))    \right)
  =  m_{pj} +    r_j(S\con e_1) 
=    r_j(S) .\]
% \end{align}
Thus we have shown $\delta_{\mathbf{a}}$ is uniform.

\medskip

Next we  find a closed form for $\exp_* (\delta_{\mathbf{a}})(S)$.
Using the definition of the $\ast$-exponential, 
\begin{equation}\label{tm.e10}
\exp_*(\delta_{\mathbf{a}}(S)) =   \sum\limits_{p=0}^\infty  \left( \frac{\left(\sum_{i\in I} a_i\delta_i \right)^{\ast p}}{p!}\right) (S)  .
  \end{equation}
All the terms in this sum vanish except the ones for which $p=|E(S)|$ (as otherwise some $\delta_i$ will evaluate to zero). 
Furthermore, the non-vanishing terms arise exactly from the $|E(S)|!$ terms of $\Delta^{(|E(S)|-1)}(S)$ that consist of the tensor product of  $|E(S)|$ elements of graded dimension 1. Thus
\begin{equation}\label{tm.e11}
\exp_*(\delta_{\mathbf{a}}(S)) =  \frac{1}{|E(S)|!}  \sum_{l=1}^{|E(S)|!}   \delta^{\ast |E(S)|}   (P_l(S)).
  \end{equation}
Equation~\eqref{tm.e4} then gives  
\begin{equation}\label{tm.e15}
\exp_*(\delta_{\mathbf{a}}(S)) =  \prod_{j\in J} x_j^{r_j(S)}.
  \end{equation}

\medskip

Next, to show that  $\alpha( \mathbf{a}, \mathbf{b}) ( S)$ can be written on the form of  Equation~\eqref{tm.e3} we prove the following identity.
For each $A\subseteq E(S)$, and for each $j\in J$,
\begin{equation}\label{tm.e12}
r_j(S)= r_j(S\del A^c) + r_j(S\con A). 
\end{equation}
To prove \eqref{tm.e12} we
 start with the observation that since  $\Delta$ is a cocommutative and   $(S \con e) \del f = (S \del f) \con e$ for $e\neq f$, 
\begin{equation}\label{tm.e13}
 \Delta(S) =  \Delta(S\del A^c)\otimes  \Delta(S\con A)  . 
 \end{equation}
 Using the notation from Equation~\eqref{tm.e7}, let  
   \begin{equation}\label{tm.e14} 
 P_l(S) ={Q}_1 \otimes {Q}_2 \otimes {Q}_3 \otimes \cdots \otimes {Q}_{|E(S)|}
  \end{equation}
 be one of the $|E(S)|!$ summands of $\Delta^{(|E(S)|-1)}(S)$ in which each $Q_k$ is of graded dimension 1.
 Then since $\delta_{\mathbf{a}}$ is uniform and since $\delta_i$ is zero on all elements except for those of graded dimension 1,
 \begin{align*}
\exp_*(\delta_{\mathbf{a}}) (S) 
&=   \sum\limits_{p=0}^\infty   \left( \frac{\left(\sum_{i\in I} a_i\cdot \delta_i \right)^{\ast p}(S)}{p!}\right) 
=  \prod_{k=1}^n \left(  \sum_{i\in I}  a_i\cdot \delta_{i}(Q_k) \right)
\\&=  \left(  \prod_{k=1}^{|A|} \left(  \sum_{i\in I}  a_i \cdot \delta_{i}(Q_k) \right)   \right)  \cdot \left(  \prod_{k=|A|+1}^{|E(S)|-|A|} \left(  \sum_{i\in I}  a_i\cdot \delta_{i}(Q_k) \right)   \right) 
  \\ &=\exp_*(\delta_{\mathbf{a}}) (S\del A^c) \cdot \exp_*(\delta_{\mathbf{a}}) (S\con A) ,
\end{align*}
where the last equality follows since  ${Q}_1 \otimes \cdots \otimes {Q}_{|E(S)|}$ is a summand of $\Delta^{(|E(S)|-1)}(S)$ so by Equation~\eqref{tm.e13} 
 ${Q}_1 \otimes \cdots \otimes {Q}_{|A|}$ is a  summand of $\Delta^{(|S\del A^c|-1)}(S\del A^c)$, and 
 $Q_{|A|+1}\otimes \cdots \otimes Q_{|E(S)|}$ is a  summand of $\Delta^{|S\con A|-1}(S\con A)$.
But then, by \eqref{tm.e15}, we have 
\[ \prod_{j\in J}  x_j^{r_j(S)} =  \left( \prod_{j\in J}  x_j^{r_j(S\del A^c)} \right)  \left( \prod_{j\in J}  x_j^{r_j(S\con A)}\right)\]
from which \eqref{tm.e12} immediately follows.

\medskip

Finally we prove Equation~\eqref{tm.e3}.
\begin{align*}
 \alpha( \mathbf{a}, \mathbf{b}) ( S)  &=  \sum_{A\subseteq E(S)}\exp_* (\delta_{\mathbf{a}}) ( S\del A^c )   \cdot \exp_* (\delta_{\mathbf{b}})  ( S\con A) 
=  \sum_{A\subseteq E(S)}      \left( \prod_{j\in J}  x_j^{r_j(S\del A^c)} \right)      \left( \prod_{j\in J}  y_j^{r_j(S\con A)}\right)
 \\ &=     \sum_{A\subseteq E}    \left( \prod_{j\in J}  x_j^{r_j(S\del A^c)} \right) \left( \prod_{j\in J}  y_j^{r_j(S)- r_j(S\del A^c)}\right)
=     \left( \prod_{j\in J}  y_j^{r_j(S)} \right)   \sum_{A\subseteq E}    \left( \prod_{j\in J} \left( \frac{x_j}{y_j}\right)^{r_j(A)} \right),
\end{align*}
completing the proof of the theorem.
\end{proof}

The reader will undoubtedly recognise Equation~\eqref{tm.e3} as being of a similar form to the spanning subgraph expansion of the classical Tutte polynomial of a graph
$T_G(x,y) = \sum_{A\subseteq E(G)} (x-1)^{r(G)-r(A)} (y-1)^{n(A)}$. (In fact it is its universal form.)

\subsection{Deletion-contraction definitions}
The motivation behind our consideration of  Hopf algebras generated by concepts of deletion and contraction was to construct graph polynomials that satisfy a  recursive deletion-contraction definition that is independent of order of edges to which it is applied, and that reduces the computation of a polynomial to that of a unique trivial object (equivalently, it generates a 1-dimensional skein module). That is we want our polynomials to satisfy a recursive definition analogous to that for the classical Tutte polynomial of a matroid. The following theorem tells us that they do.

\begin{theorem}\label{t.1}
Let $\mathcal{H}$ be a Hopf algebra of a minor system, and $\delta_{\mathbf{a}}$ and  $\delta_{\mathbf{b}}$  be a uniform selectors. Then  the canonical Tutte polynomial $\alpha(\mathbf{a}, \mathbf{b})$ is a recursively defined by 
\[ \alpha(S)   
=    \begin{cases}\delta_{\mathbf{b}}(S\con e^c )  \cdot \alpha(S\del e)  +   \delta_{\mathbf{a}} (S\del e^c ) \cdot \alpha(S\con e)   & \text{if }  S\notin \mathcal{S}_0 ,
\\ 1 & \text{if }  S\in \mathcal{S}_0 .\end{cases}
 \] 

\end{theorem}
\begin{proof}
Suppose that $\delta_{\mathbf{a}}$ and  $\delta_{\mathbf{b}}$  are uniform. 
For each $S\in \mathcal{S}$,  we can write $f_S( \mathbf{a})$ for  $\exp_* (\delta_{\mathbf{a}}) (S)$.
Then \[ \alpha(\mathbf{a}, \mathbf{b})(S) =   \sum_{A\subseteq E(S)}  f_{S\del A^c}( \mathbf{a}) \cdot f_{S\con  A}( \mathbf{b}) . \]
Suppose that $|E(S)|\geq 2$, and  $e\in  E(S)$. Then  
\begin{align*}
\alpha( \mathbf{a}, \mathbf{b})(S) 
&=   \sum_{A\subseteq E(S)}  f_{S\del A^c}( \mathbf{a}) \cdot f_{S\con  A}( \mathbf{b}) 
=  \sum_{\substack{A\subseteq E(S) \\ e\notin A} }  f_{S\del A^c}( \mathbf{a}) \cdot f_{S\con  A}( \mathbf{b})
     + \sum_{\substack{A\subseteq E(S) \\ e\in A} }  f_{S\del A^c}( \mathbf{a}) \cdot f_{S\con  A}( \mathbf{b}) \\
&= \left( \sum_{A\subseteq E(S \del e) }  f_{(S\del e)  \del A^c}( \mathbf{a}) \cdot f_{S\con  A}( \mathbf{b}) \right)
     +\left(  \sum_{A\subseteq E(S\con A) }  f_{S\del A^c}( \mathbf{a}) \cdot f_{(S\con e)\con  A}( \mathbf{b}) \right).
   \end{align*} 
Consider the computation of one of the terms of the form   $ f_{S\con  A}( \mathbf{b}) $ in the above. 
Recall $ f_{S\con  A}( \mathbf{b})  =  \exp_* (\delta_{\mathbf{b}}) (S\con  A) $.    If $m=|E(S\con  A)|$, then  we can write 
\[  \exp_* (\delta_{\mathbf{b}}) (S\con  A)   =   \frac{1}{m!}  \sum_{j=1}^m   \delta_{\mathbf{b}}(U_{j,1}) \cdot   \delta_{\mathbf{b}}(U_{j,2})\cdot \cdots  \cdot\delta_{\mathbf{b}}(U_{j,m})   \]
where each  $U_{i,j}\in \mathcal{S}_1$. (The  $U_{j,1} \otimes \cdots \otimes U_{j,m}$ are exactly the terms  of $\Delta^{(m)} (S\con  A)$ in which each tensor factor is in $\mathcal{S}_1$.)
Since $\delta_{\mathbf{b}}$ is uniform, 
\[ f_{S\con  A}( \mathbf{b})= \exp_* (\delta_{\mathbf{b}}) (S\con  A) = \delta_{\mathbf{b}}(U_{j,1}) \cdot   \delta_{\mathbf{b}}(U_{j,2})\cdot \cdots  \cdot\delta_{\mathbf{b}}(U_{j,m}) ,\]
for each $j$, and  we can choose the summand $\exp_* (\delta_{\mathbf{b}}) (S\con  A)$  is calculated from.    
Writing $\Delta^{(m)}$ as $ (\Delta^{(m-1)} \otimes \id )  \circ \Delta$, we can choose a summand that arises as a term of 
$   \Delta^{(m-1)}( (S\con  A)\del e  )    \otimes ((S\con  A)\con e^c )  $, 
and so 
$ f_{S\con  A}( \mathbf{b})   = f_{(S \con A)  \del e}( \mathbf{b})   \cdot \delta_{\mathbf{b}}((S \con A)\con e^c ) =f_{(S  \del e)\con A}( \mathbf{b})   \cdot \delta_{\mathbf{b}}(S\con e^c )  $.

For    $ f_{S\del A^c}( \mathbf{a}) $ we proceed similarly.
We have $  f_{S\del A^c}( \mathbf{a}) =  \exp_* (\delta_{\mathbf{a}}) (S\del A^c) $.  
If $m=|E(S\del  A^c)|$, then by the definition of $\delta_{\mathbf{a}}$, we can write   
\[  \exp_* (\delta_{\mathbf{a}}) (S\del  A^c)   =   \frac{1}{m!}  \sum_{j=1}^{m!}   \delta_{\mathbf{a}}(U_{j,1}) \cdot   \delta_{\mathbf{a}}(U_{j,2})\cdot \cdots  \cdot\delta_{\mathbf{a}}(U_{j,m})   \]
where, again, each  $U_{i,j}\in \mathcal{S}_1$. 
%(The  $U_{j,1} \otimes \cdots \otimes U_{j,m}$ are exactly the terms  of $\Delta^{(m)} (S\del  A^c)$ in which each tensor factor is in $\mathcal{S}_1$.)
Since $\delta_{\mathbf{a}}$ is uniform, 
\[ f_{S\del  A^c}( \mathbf{a})= \exp_* (\delta_{\mathbf{a}}) (S\del  A^c) = \delta_{\mathbf{a}}(U_{j,1}) \cdot   \delta_{\mathbf{a}}(U_{j,2})\cdot \cdots  \cdot\delta_{\mathbf{a}}(U_{j,m}) ,\]
for each $j$, and so we can choose the summand $\exp_* (\delta_{\mathbf{a}}) (S\del  A^c)$  is calculated from.    
So writing $\Delta^{(m)}$ as $ ( \id  \otimes \Delta^{(m-1)} )  \circ \Delta$, we can choose a summand that arises as a term of 
$   ((S\del  A)\del e^c )  \otimes  \Delta^{(m-1)}( (S\del  A)/ e  )     $ ,  
and so 
$  f_{S\del  A}( \mathbf{a})   =  \delta_{\mathbf{a}} ((S\del  A)\del e^c )  \cdot f_{(S\del  A)\con e }( \mathbf{a} )
 =  \delta_{\mathbf{a}} (S\del e^c )  \cdot f_{(S\con e)\del  A }(   \mathbf{a} )
  =  \delta_{\mathbf{a}} (S\del e^c )  \cdot f_{(S\con e)\del  A }(    \mathbf{a} )
 $.
Thus we have that 
\[ \alpha( \mathbf{a}, \mathbf{b})(S)   =    \delta_{\mathbf{b}}(S\con e^c )  \cdot \alpha( \mathbf{a}, \mathbf{b})(S\del e)  +   \delta_{\mathbf{a}} (S\del e^c ) \cdot \alpha( \mathbf{a}, \mathbf{b})(S\con e)  .\]
It is easily checked that this identity also holds when $|E(S)|=1$, and that $\alpha( \mathbf{a}, \mathbf{b})(S)=1$ when  $|E(S)|=0$.
\end{proof}

Theorem~\ref{t.1}  gives recursive deletion-contraction definitions for each of the graph  polynomials in Section~\ref{cxty}. Some of these relations are known: 
\begin{enumerate}
\item  Theorems~\ref{t.1} and~\ref{t.mat} give the standard deletion-contraction relations for the Tutte polynomial of a matroid.%, shown in Equation~\eqref{e.dctm}. 
\item  Theorems~\ref{t.1} and~\ref{t.gr} give the standard deletion-contraction relations for the Tutte polynomial of a graph. 
\item  Theorems~\ref{t.1} and~\ref{s2.t1} give the deletion-contraction relations for $T_{M\rightarrow M'}$ from \cite{Las99}.   %Theorem 5.3 of \cite{Las99}. 
\item  Theorems~\ref{t.1} and~\ref{s2.c1} give the  deletion-contraction relations for $L_{G\subset \Sigma}$ from \cite{EMMlv}. % Theorem 4.9 of \cite{EMMlv}. 
\item  Theorems~\ref{s2.t3}, \ref{t.rgbr1}  and~\ref{t.1} give deletion-contraction relations for $\tilde{R}$ that are equivalent to those for the Bollob\'as-Riordan polynomial along $z=1/\sqrt{xy}$ of a ribbon graph  from  \cite{EMMbook}, and of a delta-matroid from \cite{CMNR2}.
\end{enumerate}
However, the deletion-contraction definition relations for the remaining polynomials are new to the literature. Specifically, our work here gives new deletion-contraction definitions (that  terminate in their evaluations on trivial objects) for the 
\begin{enumerate}
\item 3-variable Bollob\'as-Riordan polynomial,
\item Krushkal polynomial,
\item 2-variable Penrose polynomial.
\end{enumerate}
One thing we emphasise is that to obtain the full deletion-contraction relations for the 3-variable Bollob\'as-Riordan polynomial and  the Krushkal polynomial we had to extend the class of objects on which the polynomials had been defined. %Such an extension was done in \cite{EMMlv} to obtain a full deletion-contraction definition of the Las~Vergnas polynomial.

 In the interests of brevity we will not explicitly write down the deletion-contraction definitions for all of the above polynomials. Instead we will illustrate the application of Theorem~\ref{t.1} to $\tilde{R}_{D}(x,y)$. For the deletion-contraction definition, the definitions for ribbon loops can be found just above Lemma~\ref{s2.l9}, and for ribbon dual-loops, just above  Lemma~\ref{s2.l10}.
\begin{theorem}\label{s2.c11}
$\tilde{R}_D(x,y)$    is recursively defined by  $\tilde{R}_{ ( \emptyset , \emptyset )}(x,y)=1$ and 
\[  \tilde{R}_D(x,y) =   f(e)\cdot  \tilde{R}_{D\backslash e}(x,y)+ g(e)\cdot   \tilde{R}_{D/ e}(x,y), \]
where 
\begin{align*}   f(e)&=  \begin{cases}
y-1   &\text{if $e$ is not a ribbon dual-loop,}\\
1   &\text{if $e$ is an orientable ribbon dual-loop,}\\
 \sqrt{y-1}   &\text{if $e$ is a non-orientable ribbon dual-loop;}
  \end{cases}  
\\
    g(e)&=  \begin{cases}
x-1   &\text{if $e$ is not a ribbon loop,}\\
1   &\text{if $e$ is an orientable ribbon loop,}\\
 \sqrt{x-1}   &\text{if $e$ is a non-orientable ribbon loop.}
  \end{cases}  \end{align*}
\end{theorem}
\begin{proof}
Theorem~\ref{s2.t3} gives that $ \tilde{R}_D(x,y)= \alpha( 1,y-1,\sqrt{y-1}, x-1,1,\sqrt{x-1})$. The result then follows by an application of Theorem~\ref{t.1} using Lemmas~\ref{s2.l9} and \ref{s2.l10} to recognise the values of $\delta(D/e^c)$ and $\delta(D\backslash e^c)$.
\end{proof}

\subsection{Universal forms}
The well-known universality property of the Tutte polynomial of a matroid can be formulated as saying that there exists a unique, well-defined, matroid polynomial  $f_M(x,y, a, b)$ given by
\begin{equation}\label{univ1} f_M =   \begin{cases}
  1 & \text{if } M=U_{0,0},\\
y f_{M\backslash e} (x,y)  &\text{if $e$ is a loop,}\\
xf_{M\backslash e} (x,y)   &\text{if $e$ is a coloop,}\\
af_{M\backslash e} (x,y) +b f_{M/ e} (x,y)  &\text{otherwise,}
 \end{cases}  \end{equation}
and that \[  f_M =  b^{r(M)} a^{|E(M)|-r(M)}    T_{M}( \tfrac{x}{b}, \tfrac{y}{a}).\]

The two key features of  this universality property are (1) that the recursion relations in Equation~\eqref{univ1} give a well-defined polynomial, and (2) that this polynomial can be obtained from a particular distinguished  specialisation, namely  $T_{M}$.
In the present context of canonical Tutte polynomials $\alpha$, Theorem~\ref{t.1} provides a recursion relation for a polynomial, the question becomes one of determining what  particular distinguished  specialisation can play the role of $T_M$ in the general setting.  This is answered by the following  theorem.
% answers this question for  the polynomials in Section~\ref{s.examp}.
\begin{theorem}\label{univ2}
Let  $\mathcal{H}$, $J$, $\delta_{\mathbf{a}}$, and $\delta_{\mathbf{b}}$ be defined as in Theorem~\ref{tm}.  
Suppose that $J_X, J_Y, J_Z \subseteq J$ partition $J$, and that  $\delta_{\mathbf{a}'}$ and $\delta_{\mathbf{b'}}$ are obtained from $\delta_{\mathbf{a}}$ and $\delta_{\mathbf{b}}$, respectively, by setting $x_j=1$ when $j\in J_X$ and $y_j=1$ when $j\in J_Y$. 
Then there is a unique, well-defined polynomial invariant $\alpha$ of $\mathcal{H}$ given by
\[ \alpha(S)   
=    \begin{cases}\delta_{\mathbf{b}}(S\con e^c )  \cdot \alpha(S\del e)  +   \delta_{\mathbf{a}} (S\del e^c ) \cdot \alpha(S\con e)   & \text{if }  S\notin \mathcal{S}_0 ,
\\ 1 & \text{if }  S\in \mathcal{S}_0 .\end{cases}
 \] 
Moreover,
\[    \alpha(S)    =      \left(  \prod_{j\in J_X}      x_j^{r_j(S)}  \right)      \left(  \prod_{j\in J_Y}      y_j^{r_j(S)}  \right)   \left(    \left. \alpha'(\mathbf{a}',\mathbf{b}') (S)\right|_{ \substack{ y_j\mapsto y_j/x_j \text{ for } j\in J_X \\   x_j\mapsto x_j/y_j \text{ for } j\in J_Y  } }  \right)   , \]
where $\alpha'(\mathbf{a}',\mathbf{b}') (S)$  is the canonical Tutte polynomial of $\mathcal{H}$ defined by   $\delta_{\mathbf{a}'}$ and $\delta_{\mathbf{b'}}$ .
\end{theorem}
\begin{proof}
It follows from Theorems~\ref{tm} and~\ref{t.1} that the recursion relations define the polynomial 
\[\alpha(S) = \alpha( \mathbf{a}, \mathbf{b})(S)   =    \prod_{j\in J}      y_j^{r_j(S)} \sum_{A\subseteq E(S)}   \prod_{j\in J}  \left(\frac{x_j}{y_j}\right)^{r_j(A)}. \]
Also by Theorem~\ref{tm}, 
\begin{multline*}  \alpha'(\mathbf{a}',\mathbf{b}') (S)   =         \left(  \prod_{j\in J_X   }      y_j^{r_j(S)}  \right)     \left(  \prod_{j\in J_Z   }      y_j^{r_j(S)}  \right)
\\
   \sum_{A\subseteq E(S)}   \left(  \prod_{j\in J_X}  \left(   \frac{1}{y_j}\right)^{r_j(A)} \right)     \left(  \prod_{j\in J_Y}  \left(   x_j\right)^{r_j(A)} \right)    \left(  \prod_{j\in J_Z}  \left(   \frac{x_j}{y_j}\right)^{r_j(A)} \right) .  \end{multline*}
 The theorem is readily seen to hold upon comparing these state sums for $\alpha$ and $\alpha'$.
\end{proof}

Conceptually Theorem~\ref{univ2} says that in the definition  $\alpha( \mathbf{a}, \mathbf{b})$  via Theorem~\ref{tm}, half of the variables are redundant. Although $\alpha( \mathbf{a}, \mathbf{b})$ is in variables $\{x_j, y_j\}_{j\in J}$, for each $j\in J$ we can set either $x_j$  or $y_j$ to 1 without losing any information from the polynomial. 

 Theorem~\ref{univ2} applies to all of the invariants in Section~\ref{s.examp}. In particular, it shows that each of $T_M$, $T_G$, $T_{\mathbf{M}}$, $L_{G\subset \Sigma}$,   $\widetilde{R}_D$, $\widetilde{R}_G$, $R_{(G,\mathcal{P})}$, $\widetilde{K}_{G\subset \Sigma}$,    $\widetilde{P}_D$, and $\widetilde{P}_G$ is a universal object for the relevant class of polynomials. We will not write down explicit universality statements for each of these polynomials, but will note, as an example, that for matroids, the two universality statements for $T_M$ in this section coincide.

\subsection{Specialisation as a Hopf algebra morphism}

Although fairly straightforward,  the following theorem provides a formal definition of what it means for one graph polynomial to generalise or to contain another.  We will use it to show that many known relations between graph polynomials result from natural maps on the Hopf algebra level. 

\begin{theorem}\label{tgen}
Let $\mathcal{H}$ and $\mathcal{H}'$ be graded connected commutative Hopf algebras and $\phi:\mathcal{H}\rightarrow \mathcal{H}'$ be a Hopf algebra morphism. Suppose that $\delta_{\mathcal{H}, \mathbf{a}} $ and  $\delta_{\mathcal{H}',\mathbf{x}}$ are selectors for  $\mathcal{H}$ and $\mathcal{H}'$, respectively, such that $\delta_{\mathcal{H}, \mathbf{a}}   =\delta_{\mathcal{H}',\mathbf{x}} \circ \phi$. 
Then
\[    \alpha_{\mathcal{H}}( \mathbf{a}, \mathbf{b})  =   \alpha_{\mathcal{H}'}( \mathbf{x}, \mathbf{y}) \circ \phi ,\]
where the $\alpha$'s are defined using $\delta_{\mathcal{H}}$ and $\delta_{\mathcal{H}'}$, respectively.
Moreover, if $\delta_{\mathcal{H}'}$ is uniform, then so is $\delta_{\mathcal{H}}$.
\end{theorem}
\begin{proof}
 We write $\delta_{\mathcal{H}}$ for $\delta_{\mathcal{H},\mathbf{a}} $, and  $\delta_{\mathcal{H}'}$ for $\delta_{\mathcal{H}', \mathbf{x}} $. Let $S\in \mathcal{H}$.
 Using Sweedler notation we can write  
 \[\Delta_{\mathcal{H}}^{(k)}(S )= \sum_j S_j^{(1)}\otimes \cdots \otimes S_j^{(k)}.\]  
 Since $\phi$ is a Hopf algebra morphism
 \[\Delta_{\mathcal{H'}}^{(k)}(\phi(S) )=  \phi^{\otimes k}  \left(\Delta_{\mathcal{H}}^{(k)}(S) \right)= \sum_j \phi\left(S_j^{(1)}\right)\otimes \cdots \otimes \phi\left(S_j^{(k)}\right) .  \]
 These two expressions and that  $\delta_{\mathcal{H}} =\delta_{\mathcal{H'}}\circ \phi$ give
\begin{align*}
\delta_{\mathcal{H}}^{\ast k}(S  ) &=  \sum_j  \delta_{\mathcal{H}} (S_j^{(1)})\cdot \cdots \cdot \delta_{\mathcal{H}} (S_j^{(k)})
=  \sum_j  \delta_{\mathcal{H'}} ( \phi( S_j^{(1)}))\cdot \cdots \cdot \delta_{\mathcal{H}'} (\phi(S_j^{(k)}))
=\delta_{\mathcal{H'}}^{\ast k}(\phi(S)  ) .
\end{align*}
It follows that $\exp_*(\delta_{\mathcal{H}})( S ) = \exp_*(\delta_{\mathcal{H'}})( \phi(S)  ) $, and so
$\alpha_{\mathcal{H}}(S)  =   \alpha_{\mathcal{H}'}(  \phi(S))$, as required. 

To see that  $\delta_{\mathcal{H}}$ is uniform when $\delta_{\mathcal{H}'}$ is, let $\sigma$ be an element of the symmetric group on $k$ elements. 
Then $\delta_{\mathcal{H}}^{\otimes k}  ( S_1\otimes \cdots \otimes   S_k  ) =  \delta_{\mathcal{H'}}(\phi( S_1))  \otimes \cdots \otimes  \delta_{\mathcal{H'}}(\phi( S_k)) )$ 
and $\delta_{\mathcal{H}}^{\otimes k}  ( S_{\sigma(1)}\otimes \cdots \otimes   S_{\sigma(k)}  ) =  \delta_{\mathcal{H'}}(\phi( S_{\sigma(1)}))  \otimes \cdots \otimes   \delta_{\mathcal{H'}}(\phi( S_{\sigma(k)})) )$. Since $\delta_{\mathcal{H'}}$ is uniform these two expressions are equal, and so  $\delta_{\mathcal{H}}$ is uniform.
\end{proof}
 
The point of Theorem~\ref{tgen} is that it can be used to show that the fact that one graph polynomial can be obtained as a specialisation of another follows from the fact that there is a Hopf algebra morphism (usually projection) between the corresponding Hopf algebras. Figure~\ref{dghad} summarise some of the Hopf algebra morphisms given in the paper. These descend to relations between graph polynomials on the level of canonical Tutte polynomials (See Section~\ref{cxty} for the corresponding polynomials, and follow the references for the exact specialisation). 

\begin{figure}[ht]
\begin{center}
\begin{tikzpicture}
\tikzset{node distance=3cm, auto}
  \node (A) {$\mathcal{H}^{mp}$};
  \node (B) [right of=A] {$\mathcal{H}^{m}$};
  \node (C) [right of=B] {$\mathcal{H}^{dm}$};
   \node (D) [below of=A] {$\mathcal{H}^{ps}$};
    \node (E) [below of=B] { $\mathcal{H}^{g}$};
      \node (F) [below of=C] {$\mathcal{H}^{rg}$};
        \node (G) [right of=F] {$\mathcal{H}^{vrg}$};
  \node (H) [right of=G] {$\mathcal{H}^{vgs}$};
   \draw[<->] (A) to node {Cor.~\ref{s2.t1.c2}} (B);
 \draw[<-] (E) to node[above]{plane} node[below]{Cor.~\ref{t.rgbr1.c2}} (F);
     \draw[<-] (B) to node[left] {Cor.~\ref{t.gr.l1}} (E);
          \draw[->] (B) to node {Cor.~\ref{s2.t3.c2}} (C);
    \draw[->] (D) to node {Lem.~\ref{s2.l5}} (A);
   \draw[->] (D) to node {Cor.~\ref{s2.t1.c3}} (E);
  \draw[->] (F) to node {Lem.~\ref{t.gr.l1}} (C);
    \draw[->] (G) to node[above] {Cor.~\ref{c.vbr2}} (F);
      \draw[->] (H) to node[above] {Cor.~\ref{c.vkr2}} (G);
        \draw[->,bend left] (H) to node[above] {Cor.~\ref{c.vkr2}} (D);
\end{tikzpicture}
\end{center}
\caption{Hopf algebra morphisms inducing relations between graph polynomials.}
\label{dghad}
\end{figure}
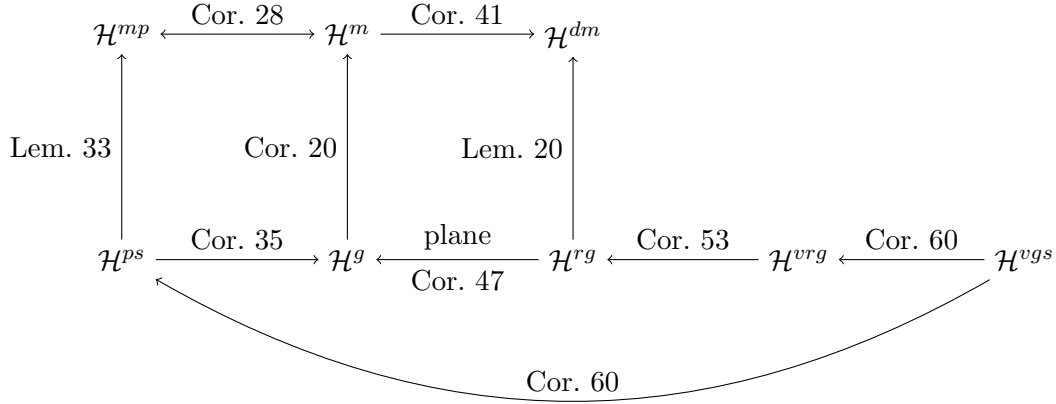

\subsection{Convolution formulas}
The convolution formula for the Tutte polynomial, which appears in W. Kook, V. Reiner and D. Stanton's paper \cite{KRS99}, and implicitly in G. Etienne and M. Las Vergnas' paper \cite{ELV}, expresses the Tutte polynomial of a graph or matroid $M$ in terms of 1-variable specialisations:
\begin{equation}\label{conv}
T_M(x,y)= \sum_{A\subseteq E(M)}  T_{M\ba A^c} (0,y)  \cdot T_{M/ A} (x,0).
\end{equation}
 It follows easily from the writing of  $T_M(x,y)$ in terms of exponentials in Corrolary~\ref{t.mat.c1}. To see this, start with the following rewriting of $\alpha$:  
 \begin{align}\label{conv2}
 \begin{split}
  \alpha( \mathbf{a}, \mathbf{b}) (S) 
&= [\exp_* (\delta_{\mathbf{a}}) \ast \exp_* (\delta_{\mathbf{b}})] (S)
\\ &= [\exp_* (\delta_{\mathbf{a}}) \ast \exp_* (\delta_{\mathbf{c}}) \ast \exp_* (\delta_{-\mathbf{c}}) \ast \exp_* (\delta_{\mathbf{b}})](G)
\\ &= [\exp_* (\delta_{\mathbf{a}}) \ast \exp_* (\delta_{\mathbf{c}})] \otimes [\exp_* (\delta_{-\mathbf{c}}) \ast \exp_* (\delta_{\mathbf{b}})](\Delta(S))
\\&= \sum_{A\subseteq E}\alpha( \mathbf{a}, \mathbf{c}) (S\del A^c)   \cdot \alpha(-\mathbf{c}, \mathbf{b})   (S\con A^c) .
\end{split}
\end{align}
An application of Theorem~\ref{t.mat.c1} with $  \mathbf{a} = (1,y-1)$,  $  \mathbf{b} = (x-1,1)$ and  $  \mathbf{c} = (-1,1)$ then gives \eqref{conv}. This derivation of the convolution formula was first observed in \cite{duchamp}.

Equation \eqref{conv2} is not specific to  $T_M$ and can be used to obtain new convolution formulae for other canonical Tutte polynomials. There is a slight subtlety in the derivation of such formula, however. For a given canonical Tutte polynomial $\alpha( \mathbf{a}, \mathbf{b})$,  the variables in each  of  $ \mathbf{a}$ or $ \mathbf{b}$ may depend upon each other (for example, this may be forced by  uniformity of $\delta$). We need to ensure that in  both $ -\mathbf{a}$ and $ -\mathbf{b}$ the variables satisfy the same dependence, and ensuring this may require some specialisation of the variables or classes of object considered.  For example, the 2-variable Bollob\'as-Riordan polynomial needs $ \mathbf{a}$ to be of the form $\mathbf{a} = (x,y,\sqrt{xy})$,  but $ -\mathbf{a} = (-x,-y,-\sqrt{xy})$ is not of this form ($(-x,-y,\sqrt{xy})$ is) and so we can not write the final equality in  \eqref{conv2} in this instance. As we will see below, we can work around this issue for the   the 2-variable Bollob\'as-Riordan polynomial   by restricting to orientable ribbon graphs or even delta-matroids.

\begin{theorem}\label{conv3}
The following identities hold.
\begin{enumerate}
\item \label{conv3.a}  If $\mathbf{M}$ is a matroid perspective and $ T_{\mathbf{M}}$ the Tutte polynomial of a morphism of a matroid, \[   T_{\mathbf{M}}(x,y)= \sum_{A\subseteq E(\mathbf{M})}  T_{\mathbf{M}\ba A^c} (0,y,-1) \cdot T_{\mathbf{M}/ A} (x,0,z).  \] 

\item   \label{conv3.b}  If $D$ is an even delta-matroid and $ \tilde{R}$ the 2-variable Bollob\'as-Riordan polynomial, \[\widetilde{R}_D(x,y) = 
\sum_{A\subseteq E(D)}  \widetilde{R}_{D\ba A^c} (0,y)  \cdot \widetilde{R}_{D/ A} (x,0).\]

\item \label{conv3.c}   If  $( G,\mathcal{P})$ is a vertex partitioned ribbon graph and $R_{ ( G,\mathcal{P})}(x,y,z)$ the Bollob\'as-Riordan polynomial, \[   R_{ ( G,\mathcal{P})}(x,y,yz^2)         =   \sum_{A\subseteq E(G)}    R_{ ( G,\mathcal{P}) \ba A^c }(0,y,yz^2)  \cdot R_{ ( G,\mathcal{P})/ A}(x,-1,1) .    \]

\item   \label{conv3.d}  If $( G\subset\Sigma,\mathcal{P})$ is  a vertex partitioned graph in a surface and  $\widetilde{K}_{(G\subset\Sigma,\mathcal{P})} (x,y,a,b)$ the Krushkal polynomial, 
\[   \widetilde{K}_{(G\subset\Sigma,\mathcal{P})} (x,y,a,ab^2)   =   \sum_{A\subseteq E(G)}  \widetilde{K}_{(G\subset\Sigma,\mathcal{P}) \ba A^c} (  -1, y,a,ab^2)  \cdot \widetilde{K}_{(G\subset\Sigma,\mathcal{P})/ A} ( x,-1,-1,-1) .  \]

\end{enumerate}

\end{theorem}
\begin{proof}
Item~\ref{conv3.a} follows by Corollary~\ref{s2.t1.c1} and Equation~\eqref{conv2} with $  \mathbf{a} = (1,y-1,1)$,  $  \mathbf{b} = (x-1,1,z)$ and  $  \mathbf{c} = (-1,1,-1)$.

For Item~\ref{conv3.b} start by considering Equation~\eqref{e.BRdelta1}.  Since $D$ is even and contraction and deletion preserve  the parity of  delta-matroids, $D_n=(\{e\},\{ \emptyset, \{e\} \})$ will never appear in any $\Delta^{n}(D)$, where $n\in \mathbb{N}$. Thus we can ignore the $\delta_n$ term in    Equation~\eqref{e.BRdelta1},  giving that  $\alpha(\mathbf{a},\mathbf{b})$ of Equation~\eqref{e.BRdelta2} equals 
 $ \exp_* (  x_1\delta_b+x_2\delta_o ) \ast  \exp_* (  y_1\delta_b+y_2\delta_o ) $. We can now apply Equation~\eqref{conv2} to this expression for $\alpha$ with  $  \mathbf{a} = (1,y-1)$, $  \mathbf{b} = (x-1,1)$ and $ \mathbf{c} = (-1,1)$. Applying Theorem~\ref{s2.t3} then gives the convolution formula.
 
 Item~\ref{conv3.c} follows from the expression of $R$ in terms of $\alpha$ in Definition~\ref{def.vbr} and Equation~\eqref{conv2} with $  \mathbf{a} = (1, y,yz,yz^2)$,  $  \mathbf{b} = (x-1,1,1,1)$ and  $  \mathbf{c} = (-1,1,1,1)$.

  Item~\ref{conv3.d} follows from the expression of $\widetilde{K}$ in terms of $\alpha$ in Definition~\ref{def.vkr} and Equation~\eqref{conv2} with $  \mathbf{a} = (1, y,a,ab,ab^2)$,  $  \mathbf{b} = (x,1,1,1,1)$ and  $  \mathbf{c} = (-1,1,1,1,1)$.
\end{proof}

By the discussions in Sections~\ref{ssec.lv} and~\ref{s.BRrg},  Item~\ref{conv3.a}  in Theorem~\ref{conv3} can be expressed in terms of the Las~Vergnas polynomial of graphs in pseudo-surfaces, and Item~\ref{conv3.b} in terms of orientable ribbon graphs (since $D(G)$ is even when $G$ is orientable by \cite{ab2,CMNR1}). 
 Note that while Corollary~\ref{c.vbr2} gives $R_{G}(x,y)$ as an evaluation of    $R_{(G,\mathcal{P})} (x,y,z)$, this result with  Item~\ref{conv3.c} of Theorem~\ref{conv3} does not give a convolution formula for $R_{G}(x,y)$ since the $R_{ ( G,\mathcal{P})/ A}(x,-1,1)$ is not   of the form required by Corollary~\ref{c.vbr2} to specialise to $R_{G/A}(x,y)$.

\subsection{Duality}\label{s.dual}
It is well-known that the Tutte polynomial of a matroid satisfies the duality relation $T_M(x,y)=T_{M^*}(y,x)$. We will now show how such duality  relations fit in our Hopf algebra framework, and, moreover, result from Hopf algebra morphisms.

\begin{definition}\label{dual1}
Let $\mathcal{H}$ be a Hopf algebra  of a minor system, as described in Proposition~\ref{p.1}. 
By a \emph{combinatorial duality} for  $\mathcal{H}$ we mean an involutionary grading preserving algebra morphism $\ast: \mathcal{H}\rightarrow \mathcal{H}$, where we denote $\ast(S)$ by $S^*$ and call it the \emph{dual} of $S$,  such that for each $S\in \mathcal{H}$ and each $e\in E(S)$, we   have $(S\del e)^*=S^* \con e$ and  $(S\con e)^*=S^* \del e$.
\end{definition}

We can now state a general duality theorem, which is a variation of  Theorem~\ref{tgen}.
\begin{theorem}\label{dual3}
Let $\mathcal{H}$ be a Hopf algebra  of a minor system with a combinatorial duality $\ast$. Let $\delta_{\mathbf{a}}=\sum_{i\in I} a_i\delta_i$  and $\delta_{\mathbf{b}}=\sum_{i\in I} b_i\delta_i$ be  selectors for  $\mathcal{H}$. Then for all $S\in \mathcal{H}$, 
\[   \alpha (\mathbf{a}, \mathbf{b}) (S) =    \alpha ( \mathbf{b}^*,\mathbf{a}^*) (S^*)  , \]
 where  $\alpha ( \mathbf{b}^*,\mathbf{a}^*) $ is defined by the selectors 
 $\delta_{\mathbf{a}^*} := \delta_{\mathbf{a}}  \circ \ast$ and  $\delta_{\mathbf{b}^*} := \delta_{\mathbf{b}}  \circ \ast$.
\end{theorem}
\begin{proof}
The proof is very similar to the proof of Theorem~\ref{tgen}, and so we only provide a sketch. 
First observe that    $\Delta(S^*)=(\tau\circ \Delta (S))^*$ where $\tau:\mathcal{H}\rightarrow \mathcal{H}:S\otimes S' \mapsto  S'\otimes S$ is the \emph{flip}.  This is since 
$ \Delta(S^*)=\sum\limits_{A\subseteq E(S^*)}(S^*\del A^c  ) \otimes  (S^*\con A)    
=   \sum\limits_{A\subseteq E(S)}  (S\con A^c)^*  \otimes (S\del A  )^*  
%=   \sum\limits_{A\subseteq E(S)}  (S\con A)^*  \otimes (S\del A^c  )^* 
= (\tau\circ \Delta (S))^*$, 
where the last equality follows since the sum is over all subsets of $E(S)$.

The result can then be obtained by  following the proof of Theorem~\ref{tgen}, but replacing $\tau\circ \ast$ for $\phi$ and noting that the presence of the flip $\tau$ reverses the order of the tensor factors.
\end{proof}

\begin{corollary}\label{dual4}
The following duality identities hold.
\begin{enumerate}
\item \label{dual4.a} [Crapo \cite{Cr69},Tutte~\cite{Tu47}] For a matroid  $M$, $T_M(x,y)=T_{M^*}(y,x)$.
\item \label{dual4.b} [Las~Vergnas \cite{Las80}] For a matroid perspective $\mathbf{M}=M\rightarrow M'$,  \[T_{\mathbf{M}}(x,y,z)=  z^{r(E)-r'(E)} T_{\mathbf{M}^*}(y,x,1/z).\]
\item \label{dual4.c} [Las~Vergnas \cite{Las78}] For a graph in a pseudo-surface  $G\subset \Sigma$, \[L_{G\subset \Sigma}(x,y,z)=  z^{n(G)-\kappa(G)} L_{G^*\subset \Sigma}(y,x,1/z).\]
\item \label{dual4.d} [Chun et al \cite{CMNR1}] For a delta-matroid  $D$, $\widetilde{R}_D(x,y)=\widetilde{R}_{D^*}(y,x)$.
\item \label{dual4.e} [Ellis-Monaghan and Sarmiento \cite{ES}, Moffatt \cite{Mo08}] For a ribbon graph  $G$, $\widetilde{R}_G(x,y)=\widetilde{R}_{G^*}(y,x)$.
\end{enumerate}
\end{corollary}
\begin{proof}
Each identity follows by an application of Theorem~\ref{dual3}. 

For  Item~\ref{dual4.a}, $U_{1,1}^*=U_{0,1}$ and so for  $\delta_{\mathbf{a}}=x_1\delta_c+x_2\delta_l$ from \eqref{t.mat.delta.1}, $\delta_{\mathbf{a}^*} = \delta_{\mathbf{a}}  \circ \ast =  x_2\delta_c+x_1\delta_l$. Thus by Corollary~\ref{t.mat.c1} and Theorem~\ref{dual3}
\[  T_{M}(x,y) = \alpha( 1, y-1, x-1,1) ( M ) =   \alpha( 1, x-1,1, y-1 ) ( M^*) =  T_{M^*}(y,x).  \]

For   Item~\ref{dual4.b},   $(U_{0,1} \rightarrow U_{0,1})^* = (U_{1,1} \rightarrow U_{1,1})$, 
$(U_{1,1} \rightarrow U_{1,1})^*=(U_{0,1} \rightarrow U_{0,1})$, 
and $(U_{1,1} \rightarrow U_{0,1})^* =  (U_{1,1} \rightarrow U_{0,1})$. 
 For  $\delta_{\mathbf{a}}=x_1\delta_{cc} + x_2\delta_{ll} +x_3\delta_{cl}$, from \eqref{e.demp}, $\delta_{\mathbf{a}^*} = \delta_{\mathbf{a}}  \circ \ast =x_2\delta_{cc} + x_1\delta_{ll} +x_3\delta_{cl}$, and so by Theorems~\ref{s2.t1.e} and~\ref{dual3},
 \[     T_{\mathbf{M} }( x,y,z  )    =  \alpha(  1,y-1,1 , x-1,1,z) (\mathbf{M})=   \alpha(  1 , x-1,z  ,  y-1,1 1) (\mathbf{M}^*)  =   z^{r(E)-r'(E)} T_{\mathbf{M}^*}(y,x,1/z).  \]

Item~\ref{dual4.c} can be obtained by proceeding as above but using the constructions in Section~\ref{ssec.lv}, or by using Item~\ref{dual4.b} and Equations~\eqref{LVeq} and~\eqref{LVeq2}.

 Item~\ref{dual4.d} follows similarly. With   $\delta_{\mathbf{a}}= a_1\delta_{b} + a_2\delta_{o} +a_3\delta_{n}$ from \eqref{e.BRdelta1}, $\delta_{\mathbf{a}^*}= a_2\delta_{b} + a_1\delta_{o} +a_3\delta_{n}$. Then by Theorems~\ref{s2.t3} and~\ref{dual3},
\[ \widetilde{R}_D(x,y) =   \alpha( 1, y-1, x-1,1) ( D ) =   \alpha( 1, x-1,1, y-1 ) ( D^*)  = \widetilde{R}_{D^*}(x,y).  \]

Item~\ref{dual4.e} can be obtained by proceeding as above by using the constructions in Section~\ref{s.BRrg}, or by using Item~\ref{dual4.d} and that $\widetilde{R}_G=\widetilde{R}_{D(G)}$ from Equation~\eqref{t.rgbr.e}.
\end{proof}

\section{Examples in detail}\label{s.examp}
In this section we give a large number of examples of minor systems and their canonical Tutte polynomials. In particular we identify the classical Tutte polynomial of a graph or matroid, and a number of its extensions to graphs in surfaces as canonical Tutte polynomials. Because of the wide variety of examples, this section is fairly long. However, each subsection deals with a different polynomial and the subsections are largely independent of each other.% No subsection is a prerequisite for the later content of the paper. Thus a reader may safely pick and choose the examples he or she is interested in, ignoring the details of the rest.

As mentioned previously, we will often specify $\mathbf{a}$, $\mathbf{b}$ $\delta_{\mathbf{a}}$, $\delta_{\mathbf{b}}$, and $\alpha(\mathbf{a},\mathbf{b})$ as follows. We fix some basis of $\mathcal{H}_1$ and some order of it. Then we specify $\delta_{\mathbf{a}} = \sum_{i=1}^n  a_i \delta_i$  by writing  $ \delta(a_1,a_2, \ldots, a_n)$. We do similarly for $\delta_{\mathbf{b}}$. Finally, we specify $\alpha(\mathbf{a},\mathbf{b})$ by writing $\alpha(a_1,\ldots, a_n,b_1,\ldots, b_n)$.

\subsection{The classical Tutte polynomial of a matroid}\label{s.mat}

The \emph{Tutte polynomial} of a matroid $M$ on a ground set  $E$ with rank function $r$ is  
\begin{equation}\label{e.tutteM}
T_M(x,y) = \sum_{A\subseteq E} (x-1)^{r(G)-r(A)} (y-1)^{|A|-r(A)}.
\end{equation}

The following result is readily seen to hold.
\begin{lemma}\label{l.hm}
The set of isomorphism classes of matroids forms a minor system where the grading is given by the cardinality of the ground set, deletion and contraction are given by the usual matroid deletion and contraction, and multiplication is given by direct sum.
\end{lemma}
 For convenience we will henceforth identify a  matroid with its isomorphism class. The minor system gives rise to a well-known deletion-contraction Hopf algebra of matroids.
\begin{definition}\label{d.hm}
We let  $\mathcal{H}^{m}$  denote the  Hopf algebra associated with matroids via  Lemma~\ref{l.hm} and Proposition~\ref{p.1}. 
Its coproduct is given by 
 $
 \Delta_{m}(M)=\sum_{A\subseteq E(M)}M\backslash A^c  \otimes  M/ A.
$
\end{definition}
There are exactly two elements in $\mathcal{H}^{m}_1$, namely the uniform matroids $U_{1,1}$ and $U_{0,1}$. The selector associated with $\mathcal{H}^{m}$ is 
\begin{equation}\label{t.mat.delta.1}
\delta_{\mathbf{a}}=\delta(x_1,x_2) = x_1\delta_c+x_2\delta_l,
\end{equation}
 where
\begin{equation}\label{t.mat.delta}
  \delta_{c}(M  ) := \begin{cases} 1 &\mbox{if } M = U_{1,1}, \\   0 & \mbox{otherwise}; \end{cases} 
\quad\text{and} \quad
   \delta_{l}(M) := \begin{cases} 1 &\mbox{if } M= U_{0,1},  \\   0 & \mbox{otherwise}. \end{cases} \end{equation}
  (In the notation of Theorem~\ref{tm}, $\mathbf{a}=(x_1,x_2)$ with ordering $U_{1,1}, U_{0,1}$ of a basis of $\mathcal{H}^{m}_1$.)
\begin{theorem}\label{t.mat}
The Tutte polynomial of  a matroid  arises as the canonical Tutte polynomial of the  Hopf algebra $\mathcal{H}^{m}$: 
 \begin{equation}
 \alpha( x_1, x_2, y_1,y_2) ( M )  =    x_1^{r(M)} y_2^{|E(M)|-r(M)}    T_{M}\left( \frac{y_1}{x_1}+1, \frac{x_2}{y_2}+1\right). 
 \end{equation}
\end{theorem}
\begin{proof} For a matroid $M$ on $E$,
upon taking $r_1(M):=r(M)$ to be the rank function of $M$, and $r_2:=|E| - r(M)$ to be its nullity, 
and recalling that $r(M) =r(M/e)$ if $e$ is a loop, and is equal to $r(M/e)+1$ otherwise, 
 Theorem~\ref{tm} gives 
 \begin{align*}
 \alpha( x_1, x_2, y_1,y_2) ( M )  &= 
   y_1^{r(M)}    y_2^{|E|-r(M)} \sum_{A\subseteq E(S)}   \left(\frac{x_1}{y_1}\right)^{r(A)}  \left(\frac{x_2}{y_2}\right)^{|A|-r(A)}
\\&=     x_1^{r(M)}    y_2^{|E|-r(M)} \sum_{A\subseteq E(S)}   \left(\frac{y_1}{x_1}\right)^{r(M)-r(A)}  \left(\frac{x_2}{y_2}\right)^{|A|-r(A)}.
\end{align*}
The result follows by comparing this with Equation~\eqref{e.tutteM}.
\end{proof}

\begin{corollary}\label{t.mat.c1}
With $\alpha( x_1, x_2, y_1,y_2)$ defined as in Theorem~\ref{t.mat},
 \begin{equation}
     T_{M}(x,y) = \alpha( 1, y-1, x-1,1) ( M ).
 \end{equation}
\end{corollary}

In Theorem~\ref{t.mat}  we have recovered a result of Duchamp et. al from \cite{duchamp}, which was the inspiration for this work.

\subsection{The classical Tutte polynomial of a graph}\label{s.gr}
The classical Tutte polynomial of a graph $G$ is 
\begin{equation}\label{e.tutte}
T_G(x,y) = \sum_{A\subseteq E(G)} (x-1)^{r(G)-r(A)} (y-1)^{|A|-r(A)},
\end{equation}
where $r(A):=v(A)-c(A)$, and $v(A)$ and $c(A)$ denote the numbers of vertices and components, respectively, of the spanning subgraph of $G$ on $A$.

For a graded connected Hopf algebra we require a single element of graded dimension zero. For this we consider graphs up to 1-sums. Recall $G$ and $H$ are graphs, and $v$ is a vertex of $G$ and  $u$ a vertex of $H$, then a \emph{1-sum}, $G\oplus_1 H$ is the graph obtained by identifying the vertices $u$ and $v$.  

The following is easily seen.
\begin{lemma}\label{l.hg}
The set of equivalence classes of graphs considered up to 1-sums and isomorphism  forms a minor system where the grading is given by the cardinality of the edge set, deletion and contraction are given by the usual graph deletion and contraction, and multiplication is given by disjoint union. 
\end{lemma}

We now identify a graph with its equivalence class. The above minor system gives rise to a Hopf algebra:
\begin{definition}\label{d.hg}
We let  $\mathcal{H}^{g}$  denote the  Hopf algebra associated with graphs via  Lemma~\ref{l.hg} and Proposition~\ref{p.1}. 
Its coproduct is given by 
$
 \Delta_{g}(G)=\sum\limits_{A\subseteq E(G)}G\backslash A^c  \otimes  G/ A.
 $
\end{definition}

%Recall for a graph $G$ is a graph,  its {\em cycle matroid} (or \emph{graphic matroid}) is $C(G):=(E(G), r_{C(G)})  $, where  $r_{C(G)}(A):= v(A)-c(A)$.
\begin{lemma}\label{t.gr.l1}
There is a natural Hopf algebra morphism $\phi: \mathcal{H}^{g}\rightarrow \mathcal{H}^{m}$ given by $\phi:G\rightarrow C(G)$, where $C(G)$ is the cycle matroid of $G$. 
\end{lemma}
\begin{proof}
Since $C(G\oplus_1 H)=C(G)\oplus C(H)$, $\phi$ is well-defined. 
It is easily seen that $\phi$ is multiplicative, and sends the (co)unit to the (co)unit.
A standard result in matroid theory is that $C(G)/A=C(G/A)$ and  $C(G)\backslash A=C(G\backslash A)$, giving
$  C(\Delta_{g}(G))  =    \sum_{A\subseteq E} C(G\backslash  A^c)  \otimes  C(G/A)  =  \sum_{A\subseteq E} C(G)\backslash  A^c \otimes  C(G)/A  = \Delta_{m}(C(G)) $.
\end{proof}
We will use $\phi$ to identify the Tutte polynomial of $\mathcal{H}^{g}$.

 $\mathcal{H}^{g}_1$  has two elements, a bridge and a loop which giving rise to a selector
\[  \delta(x_1,x_2) = x_1\delta_b+x_2\delta_l,\]
 where
\begin{equation}\label{g.selector}
   \delta_{b}(G) := \begin{cases} 1 &\mbox{if } G=(\{u,v\},\{(u,v)\}),\\   0 & \mbox{otherwise}; \end{cases} 
\quad\text{and} \quad
   \delta_{l}(G) := \begin{cases} 1 &\mbox{if } G=(\{v\},\{(v,v)\}) ,\\   0 & \mbox{otherwise}. \end{cases} \end{equation}
\begin{theorem}\label{t.gr}
The Tutte polynomial of  a graph  arises as the Tutte polynomial of the  Hopf algebra $\mathcal{H}^{g}$: 
 \begin{equation}
 \alpha( x_1, x_2, y_1,y_2) ( G )  =    x_1^{r(G)} y_2^{|E(G)|-r(G)}    T_{G}( \tfrac{y_1}{x_1}+1, \tfrac{x_2}{y_2}+1). 
 \end{equation}
\end{theorem}
\begin{proof}
Upon verifying that $ \delta_{b}(G)=  \delta'_{c}(C(G))$ and $ \delta_{l}(G)=  \delta'_{l}(C(G))$, where the primed $\delta$'s are those of Equation~\eqref{t.mat.delta}, the result follows immediately from Theorems~\ref{tgen} and~\ref{t.mat}.
\end{proof}
Note that  Theorem~\ref{t.gr}  can also be proven via Theorem~\ref{tm} giving a proof almost identical to that of Theorem~\ref{t.mat}.

\begin{corollary}\label{t.gr.c1}
 \begin{equation}
     T_{G}(x,y) = \alpha( 1, y-1, x-1,1) ( G ).
 \end{equation}
\end{corollary}

Graphs provide a convenient setting to illustrate a direct computation of a canonical Tutte polynomial $\alpha( \mathbf{a}, \mathbf{b})$.  
\begin{example}\label{ex.tutte} 
Let $\mathcal{H}^g$ be the Hopf algebra of formal $\mathbb{Q}$-linear combinations of graphs considered up to the one point join operation and isomorphism, with multiplication given by disjoint union, and coproduct  given  $\Delta(G)=\sum_{A\subseteq E(G)}G\backslash A^c  \otimes  G/ A$. Let $\delta_{\mathbf{a}}=x_1\delta_b+x_2\delta_l$ and $\delta_{\mathbf{b}}=y_1\delta_b+y_2\delta_l$, where 
$\delta_b$ and $\delta_l$ are given by \eqref{g.selector}. Then 
\[ \Delta\left( \raisebox{-1mm}{\includegraphics[scale=0.15]{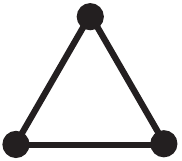}} \right) =  
  \raisebox{-1mm}{\includegraphics[scale=0.15]{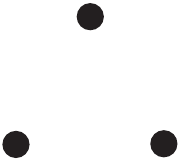}} \otimes   \raisebox{-1mm}{\includegraphics[scale=0.15]{a2}}  +   3 \raisebox{-1mm}{\includegraphics[scale=0.15]{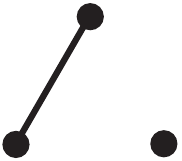}} \otimes   \raisebox{-1mm}{\includegraphics[scale=0.15]{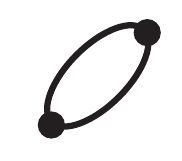}} +3 \raisebox{-1mm}{\includegraphics[scale=0.15]{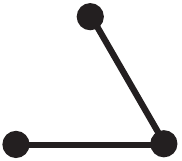}} \otimes   \raisebox{-1mm}{\includegraphics[scale=0.15]{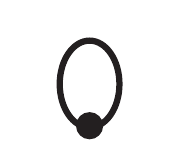}} +  \raisebox{-1mm}{\includegraphics[scale=0.15]{a2}} \otimes   \raisebox{-1mm}{\includegraphics[scale=0.15]{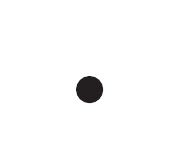}}   , \]
so
\begin{multline*}  \alpha( \mathbf{a}, \mathbf{b})   \left( \raisebox{-1mm}{\includegraphics[scale=0.15]{a2}} \right)  
=\exp_* (\delta_{\mathbf{a}}) \ast \exp_* (\delta_{\mathbf{b}})  \left( \raisebox{-1mm}{\includegraphics[scale=0.15]{a2}} \right) 
= \exp_* (\delta_{\mathbf{a}})  \left( \raisebox{-1mm}{\includegraphics[scale=0.15]{a4}} \right)  \cdot  \exp_* (\delta_{\mathbf{b}})  \left( \raisebox{-1mm}{\includegraphics[scale=0.15]{a2}} \right) 
 +3  \exp_* (\delta_{\mathbf{a}})  \left( \raisebox{-1mm}{\includegraphics[scale=0.15]{a3}} \right)  \cdot  \exp_* (\delta_{\mathbf{b}})  \left( \raisebox{-1mm}{\includegraphics[scale=0.15]{a7}} \right) 
\\
+3  \exp_* (\delta_{\mathbf{a}})  \left( \raisebox{-1mm}{\includegraphics[scale=0.15]{a1}} \right)  \cdot  \exp_* (\delta_{\mathbf{b}})  \left( \raisebox{-1mm}{\includegraphics[scale=0.15]{a6}} \right) 
+ \exp_* (\delta_{\mathbf{a}})  \left( \raisebox{-1mm}{\includegraphics[scale=0.15]{a2}} \right)  \cdot  \exp_* (\delta_{\mathbf{b}})  \left( \raisebox{-1mm}{\includegraphics[scale=0.15]{a5}} \right) .
 \end{multline*}
Now $\exp_* (\delta_{\mathbf{a}})  \left( \raisebox{-1mm}{\includegraphics[scale=0.15]{a4}} \right)=1$. 
The only non-zero terms of $ \exp_* (\delta_{\mathbf{b}})  \left( \raisebox{-1mm}{\includegraphics[scale=0.15]{a2}} \right) $ come from the terms of $  \Delta^{(2)}$ in which all tensor factors are in $\mathcal{H}^g_1$. Direct computation gives 
$ \Delta^{(2)}\left( \raisebox{-1mm}{\includegraphics[scale=0.15]{a2}} \right) =    (id\otimes \Delta)\circ \Delta  = \cdots +6    \raisebox{-1mm}{\includegraphics[scale=0.15]{a3}} \otimes \raisebox{-1mm}{\includegraphics[scale=0.15]{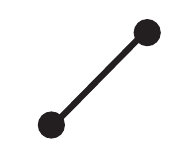}} \otimes   \raisebox{-1mm}{\includegraphics[scale=0.15]{a6}}   +\cdots$ where no other summands are in $(\mathcal{H}^g_1)^{\otimes 3}$.  Thus  
$ \exp_* (\delta_{\mathbf{b}})  \left( \raisebox{-1mm}{\includegraphics[scale=0.15]{a2}} \right)  =  y_1^2y_2$. By computing the other exponentials similarly we see that $ \alpha( \mathbf{a}, \mathbf{b})   \left( \raisebox{-1mm}{\includegraphics[scale=0.15]{a2}} \right)= y_1^2y_2+3x_1y_1y_2 + 3x_2^2y_2+x_1^2x_2=  x_1^{r(G)} y_2^{|E(G)|-r(G)}    T_{\includegraphics[scale=0.1]{a2}}( \tfrac{y_1}{x_1}+1, \tfrac{x_2}{y_2}+1)$.
\end{example}

\subsection{The Tutte polynomial of a morphism of a matroid}\label{ssec.mp}

As defined by Las~Vergnas in \cite{Las78a,Las80}, the {\em Tutte polynomial} of the matroid perspective  $\mathbf{M}=M\rightarrow M'$, where $M$ has rank function $r$, $M'$ has rank function $r'$ and both matroids have ground set $E$ is 
  \begin{equation}\label{d.tpmp}
   T_{\mathbf{M}}(x,y,z) = \sum_{A \subseteq E}
                 (x - 1)^{r'( E) - r'( A)}
                 (y-1)^{|A|-r(A)}
                  z^{(r(E)-r(A))-(r'(E)-r'(A))}.
  \end{equation}

The following lemma is easily seen to hold.
\begin{lemma}\label{s2.l1}
The set of isomorphism classes of matroid perspectives forms a minor system where the grading is given by the cardinality of the ground set, deletion and contraction are given by matroid perspective deletion and contraction, and multiplication is given by direct sum. 
\end{lemma}

\begin{definition}\label{d.hmp}
We let  $\mathcal{H}^{mp}$  denote the  Hopf algebra associated with matroid perspectives via   Proposition~\ref{p.1}. 
Its coproduct is given by 
$
 \Delta_{mp}(\mathbf{M})=\sum_{A\subseteq E}\mathbf{M}\backslash A^c  \otimes  \mathbf{M}/ A$,
 where  $\mathbf{M}=M\rightarrow M'$,  $E$ is its ground set, and  the deletion and contraction are the usual matroid perspective deletion and contraction.
\end{definition}

We will  show that Las Vergnas' Tutte polynomial of a matroid perspective is the canonical  Tutte polynomial of $\mathcal{H}^{mp}$. 

Up to isomorphism, there are exactly three matroid perspective over one element:   $U_{0,1} \rightarrow U_{0,1}$, $U_{1,1} \rightarrow U_{1,1}$, and $U_{1,1} \rightarrow U_{0,1}$. Set 
\[   \delta_{cc}(\mathbf{M}  ) := \begin{cases} 1 &\mbox{if } \mathbf{M}  =U_{1,1} \rightarrow U_{1,1}, \\   0 & \mbox{otherwise} ;\end{cases} 
 \quad  \quad
   \delta_{ll}(\mathbf{M}   ) := \begin{cases} 1 &\mbox{if } \mathbf{M}   = U_{0,1} \rightarrow U_{0,1}, \\   0 & \mbox{otherwise} ;\end{cases} \] 
\[   \delta_{cl}(\mathbf{M}   ) := \begin{cases} 1 &\mbox{if } \mathbf{M}  =U_{1,1} \rightarrow U_{0,1}, \\   0 & \mbox{otherwise} .\end{cases} \] 
(The subscripts of the $\delta$'s record, in order, if each matroid in $\mathbf{M}$ contains a loop or a coloop.)
Let
\begin{equation}\label{e.demp} 
\delta_{\mathbf{a}}= \delta(x_1,x_2,x_3) :=x_1\delta_{cc} + x_2\delta_{ll} +x_3\delta_{cl}. 
  \end{equation}

\begin{theorem}\label{s2.t1}
Las Vergnas' Tutte polynomial of  a matroid perspective arises as the canonical Tutte polynomial of the  Hopf algebra $\mathcal{H}^{mp}$: 
 \begin{equation}\label{s2.t1.e}
 \alpha( \mathbf{a}, \mathbf{b}) ( \mathbf{M} )  =     x_1^{r'(M')} y_2^{|E|-r(M)} x_3^{r(M)-r'(M')}     T_{\mathbf{M} }\left( \frac{y_1}{x_1}+1, \frac{x_2}{y_2}+1, \frac{y_3}{x_3}  \right) ,
 \end{equation}
where $\mathbf{a} = ( x_1,x_2,x_3)$, $\mathbf{b}  =(y_1,y_2,y_3)$,  $E$ is the ground set of the matroid perspective $\mathbf{M} =M\rightarrow M' $, $r$ is the rank function of $M$, and $r'$ the rank function  of $M'$. 
\end{theorem}
\begin{proof}
Set $r_1(\mathbf{M}):=r'(M')$, $r_2(\mathbf{M}):=|E|-r(M)$, and $r_3(\mathbf{M}):=r(M)-r'(M')$. Then
$r_1(\mathbf{M})=r_1(\mathbf{M}/e)+1$, if $\delta_{cc}(\mathbf{M}|_e)=1$;  
$r_2(\mathbf{M})=r_2(\mathbf{M}/e)+1$, if $\delta_{ll}(\mathbf{M}|_e)=1$;  
$r_3(\mathbf{M})=r_3(\mathbf{M}/e)+1$, if $\delta_{cl}(\mathbf{M}|_e)=1$;  
and otherwise $r_j(\mathbf{M})=r_j(\mathbf{M}/e)$.   
An application of  Theorem~\ref{tm} then gives 
 \[%\begin{multline*}
 \alpha( \mathbf{a},\mathbf{b}) ( \mathbf{M} )  =       y_1^{r'(M')}    y_2^{|E|-r(M)}    y_3^{r(M)-r'(M')}   
  \sum_{A\subseteq E}   \left(\frac{x_1}{y_1}\right)^{r'(A)}  \left(\frac{x_2}{y_2}\right)^{|A|-r(A)} \left(\frac{x_3}{y_3}\right)^{r(A)-r'(A)},
\]%\end{multline*}
which, remembering the definition of $T_{\mathbf{M}}$ from \eqref{d.tpmp}, is readily  written  as~\eqref{s2.t1.e}.
\end{proof}

\begin{corollary}\label{s2.t1.c1}
With $ \alpha( \mathbf{a}, \mathbf{b})$ defined as in Theorem~\ref{s2.t1},
 \begin{equation}
   T_{M\rightarrow M' }( x,y,z  ) =  \alpha( \mathbf{a}, \mathbf{b}) ( M\rightarrow M' ) ,
 \end{equation}
where $\mathbf{a} = ( 1,y-1,1)$, $\mathbf{b}  =(x-1,1,z)$. 
\end{corollary}

The following corollary provides a good illustration of how Hopf algebra maps give rise to relationships between polynomials. The following identities for the  Tutte polynomial of matroid perspectives first appeared in \cite{Las80}.
\begin{corollary}\label{s2.t1.c2}
Let $\mathcal{H}^{mp}$ be the Hopf algebra of matroid perspectives from Definition~\ref{d.hmp}, and $\mathcal{H}^{m}$ be the Hopf algebra of matroids  from Definition~\ref{d.hm}. Then the following hold.
\begin{enumerate}
\item The  inclusion $\phi_1: \mathcal{H}^{m} \rightarrow \mathcal{H}^{mp}$ defined by $ \phi_1 (M)= (M\rightarrow M)$ is a Hopf algebra morphism.    Furthermore it naturally induces the  identity $T_M(x,y) = T_{M\rightarrow M}(x,y,z)$.
\item The  projection $\phi_2: \mathcal{H}^{mp} \rightarrow \mathcal{H}^{m}$ defined by $ \phi_2 (M\rightarrow M')=M$ is a Hopf algebra morphism.    Furthermore it naturally induces the  identity $T_M(x,y) = T_{M\rightarrow M'}(x,y,x-1)$.
\item The  projection $\phi_3: \mathcal{H}^{mp} \rightarrow \mathcal{H}^{m}$ defined by $ \phi_3 (M\rightarrow M')=M'$ is a Hopf algebra morphism.    Furthermore it naturally induces the  identity $T_{M'}(x,y) =  (y-1)^{r(M)-r(M')} T_{M\rightarrow M'}(x,y,1/(y-1))$.
\end{enumerate}
\end{corollary}
\begin{proof}
It is readily verified that each of the three maps is a Hopf algebra morphism. To obtain the polynomial identities we apply Theorem~\ref{tgen} to Theorems~\ref{t.mat} and~\ref{s2.t1}.

For $\phi_1$,  in Theorem~\ref{tgen} let $\mathcal{H}=\mathcal{H}^{m}$,  $\mathcal{H'}=\mathcal{H}^{mp}$,
 $\delta_{\mathcal{H}}$ be the selector used in Theorem~\ref{t.mat}, and 
 $\delta_{\mathcal{H'}}$ be the selector used in Theorem~\ref{s2.t1}. 
 We then have $  \delta_{\mathcal{H}}( x_1,x_2) ( M )  =  \delta_{\mathcal{H}}( x_1,x_2,x_3) ( M\rightarrow M) $. By Theorem~\ref{tgen}, it follows that 
 \[ \alpha_{\mathcal{H}}(  x_1,x_2, y_1,y_2) ( M)  =   \alpha_{\mathcal{H}'}( x_1,x_2, x_3,y_1,y_2,y_3)   ( M\rightarrow M) .\]
 Theorems~\ref{t.mat} and~\ref{s2.t1} give 
 \[  
  x_1^{r(M)} y_2^{|E|-r(M)}    T_{M}\left(  \frac{y_1}{x_1}+1, \frac{x_2}{y_2}+1  \right)
=
  x_1^{r(M)} y_2^{|E|-r(M)} x_3^{r(M)-r(M)}     T_{M\rightarrow M}\left( \frac{y_1}{x_1}+1, \frac{x_2}{y_2}+1, \frac{y_3}{x_3}  \right) 
      \] 
 from which the result follows.

For $\phi_2$, in Theorem~\ref{tgen} let $\mathcal{H}=\mathcal{H}^{mp}$,  $\mathcal{H'}=\mathcal{H}^{m}$,
 $\delta_{\mathcal{H}}$ be the selector used in Theorem~\ref{s2.t1}, and 
 $\delta_{\mathcal{H'}}$ be the selector used in Theorem~\ref{t.mat}. 
 We then have $  \delta_{\mathcal{H}}( x_1,x_2,x_1) ( M\rightarrow M' )  =  \delta_{\mathcal{H}}(x_1,x_2) ( M) $. By Theorem~\ref{tgen}, it follows that 
$\alpha_{\mathcal{H}}(  x_1,x_2,x_1, y_1,y_2,y_1) ( M\rightarrow M' )  =   \alpha_{\mathcal{H}'}( x_1,x_2, y_1,y_2)   ( M) $.
 
 Theorems~\ref{t.mat} and~\ref{s2.t1} give 
 \[  
 x_1^{r'(M')} y_2^{|E|-r(M)} x_1^{r(M)-r'(M')}     T_{M\rightarrow M' }\left( \frac{y_1}{x_1}+1, \frac{x_2}{y_2}+1, \frac{y_1}{x_1}  \right) 
 =
 x_1^{r(M)} y_2^{|E|-r(M)}    T_{M}\left(  \frac{y_1}{x_1}+1, \frac{x_2}{y_2}+1  \right), 
    \] 
 from which the result follows.

 The argument for $\phi_3$ is similar. In Theorem~\ref{tgen} let $\mathcal{H}=\mathcal{H}^{mp}$,  $\mathcal{H'}=\mathcal{H}^{m}$,
 $\delta_{\mathcal{H}}$ be the selector used in Theorem~\ref{s2.t1}, and 
 $\delta_{\mathcal{H'}}$ be the selector used in Theorem~\ref{t.mat}. 
 We then have $  \delta_{\mathcal{H}}( x_1,x_2,x_2) ( M\rightarrow M' )  =  \delta_{\mathcal{H}}( x_1,x_2) ( M') $. By Theorem~\ref{tgen}, it follows that 
$ \alpha_{\mathcal{H}}(  x_1,x_2,x_2, y_1,y_2,y_2) ( M\rightarrow M' )  =   \alpha_{\mathcal{H}'}( x_1,x_2, y_1,y_2)   ( M') $.
 Theorems~\ref{t.mat} and~\ref{s2.t1} give 
 \[  
 x_1^{r'(M')} y_2^{|E|-r(M)} x_2^{r(M)-r'(M')}     T_{M\rightarrow M' }\left( \frac{y_1}{x_1}+1, \frac{x_2}{y_2}+1, \frac{y_2}{x_2}  \right) 
 =
 x_1^{r(M')} y_2^{|E|-r(M')}    T_{M'}\left(  \frac{y_1}{x_1}+1, \frac{x_2}{y_2}+1  \right), 
    \] 
and the result follows.
\end{proof}

\subsection{Las Vergnas' topological Tutte polynomial}\label{ssec.lv}

 Here a \emph{graph in a pseudo-surface},  $G\subset \Sigma$, consists of a graph $G=(V,E)$ and a drawing of $G$ on a pseudo-surface $\Sigma$  (i.e., a surface with pinch points,  also known as a pinched surface) such that the edges only intersect at their ends and such that any pinch points are vertices of the graph. 
 % Graphs in  pseudo-surfaces are considered up to homeomorphism of the pseudo-surface that restricts to graph isomorphism (the homeomorphism should be orientation preserving when $\Sigma$ is orientable).
The components of $\Sigma \backslash G$ are called the {\em regions} of $G$, and  $G\subset \Sigma$ is a \emph{cellularly embedded graph} if $\Sigma$ is a surface (so there are no pinch points) and each of its regions is homeomorphic to a disc. 

We define  
\begin{equation}\label{LVeq3}
\kappa(A):= \#\mathrm{cpts}(\Sigma  \backslash (V \cup  A))  - \#\mathrm{cpts}(\Sigma  \backslash V).
\end{equation} 
Let $G\subset \Sigma$ be a graph in a pseudo-surface, and $e\in E(G)$. Then we say that $e$ is a {\em quasi-loop} if  $\kappa(e)=1$,  a {\em quasi-bridge} if it is adjacent to exactly one region of $G\subset \Sigma$, and a \emph{bridge} (respectively, \emph{loop}) if it is a bridge (respectively,  loop) of the underlying graph $G$.
Note that a quasi-loop is necessarily a loop; a bridge is necessarily a quasi-bridge; and a quasi-bridge could be a loop, a bridge, or neither. 
If $e\in E(G)$ then $G\backslash  e\subset \Sigma$ is the  graph in a pseudo-surface obtained by removing the edge $e$ from the drawing of $G\subset \Sigma$ (without removing the  points of $e$ from $\Sigma$, or its incident vertices).    Edge contraction is defined by forming a quotient space of the surface: $G/e \subset \Sigma/e$ is the graph in a pseudo-surface obtained by identifying the edge $e$ to a point. This point becomes a vertex of $G/e$.  If $e$ is a loop, then contraction can create pinch points with the new vertex lying on it (see Figure~\ref{cont.a}--\ref{cont.c}).

\begin{figure}
\centering
\subfigure[$G\subset \Sigma$. ]{
\includegraphics[height=16mm]{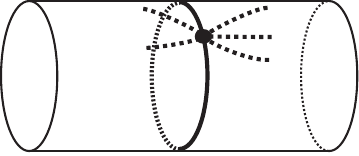}
\label{cont.a}
}
\hspace{1cm}
\subfigure[$G/e\subset \Sigma/e$.]{
\includegraphics[height=16mm]{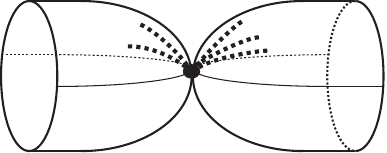}
\label{cont.c}
}
\hspace{1cm}
\subfigure[Resolving the pinch point.]{
\includegraphics[height=16mm]{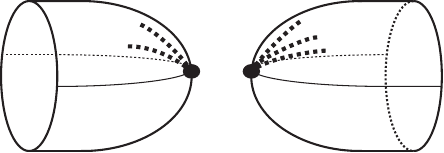}
\label{cont.b}
}
\caption{Actions on graphs in  pseudo-surfaces.}
\label{cont}
\end{figure}

The \emph{dual}, $G^*$, of a graph in a pseudo-surface $G \subset \Sigma$ is the abstract graph with vertex set corresponding to the regions of $\Sigma \backslash G$ and an edge between  (not necessarily distinct) vertices whenever the corresponding regions share an edge of $G$ on their boundaries. There is natural identification between the edges of $G^{*}$ and $G$.

The {\em cycle matroid}, $C(G)$, of $G \subset \Sigma$ is the cycle matroid of its underlying graph.
%If $G$ is a graph in a pseudo-surface,  its {\em cycle matroid} is  the 
% $C(G):=(E(G), r_{C(G)})  $, where  $r_{C(G)}(A):= v(A)-c(A)$; 
%and
 Its {\em bond matroid} is   $B(G):=(C(G))^*$. It is worth emphasising that although when $G$ is a plane graph  $ B(G^*) =C(G)$, this   {\em does not} hold, in general, for non-plane graphs. 

When $G\subset \Sigma$ is a graph  in a pseudo-surface
$ (B(G^{*}) \rightarrow C(G)) $ is a matroid perspective (see \cite{EMMlv,Las78}).
Its {\em Las~Vergnas polynomial}, $L_{G\subset \Sigma}$, is then defined by
\begin{equation} \label{LVeq}
L_{G\subset \Sigma}(x,y,z)   :=   T_ {(B(G^{*}) \rightarrow C(G))}  (x,y,z).
 \end{equation}
In \cite{EMMlv} it was shown that when $G\subset \Sigma$ is a graph in the  pseudo-surface and $A\subseteq E(G)$, then 
\begin{equation}\label{LVeq2}
r_{B(G^*)}(A)= |A|-\kappa(A).
\end{equation}
Then writing $\kappa(G)$ for $\kappa(E)$ and using \eqref{LVeq2}, 
\begin{equation} \label{LVeq4}
L_{G\subset \Sigma}(x,y,z)   =  \sum_{A \subseteq E} (x-1)^{r(G)-c(A)}(y-1)^{\kappa(A)}z^{[n(G)-\kappa(G)]-[n(A)-\kappa(A)]}.
 \end{equation}

We will  show that the Las~Vergnas polynomial is the canonical Tutte polynomial  associated with graphs in pseudo-surfaces and their minors. We will state the result before  describing the relevant Hopf algebra $\mathcal{H}_{ps}$, and selectors.

\begin{theorem}\label{s2.c1}
The Las Vergnas polynomial is the canonical Tutte polynomial of the  Hopf algebra $\mathcal{H}_{ps}$ associated with pseudo-surface minors: 
 \begin{equation}
 \alpha( \mathbf{a}, \mathbf{b}) ( G )  =   x_1^{r(G)} y_2^{\kappa(G)} x_3^{n(G)-\kappa(G)}     L_{G\subset \Sigma}( \tfrac{y_1}{x_1}+1, \tfrac{x_2}{y_2}+1, \tfrac{y_3}{x_3}  ) ,
 \end{equation}
where $\mathbf{a} = ( x_1,x_2,x_3)$, $\mathbf{b}  =(y_1,y_2,y_3)$. 
\end{theorem}
We will prove the theorem after describing the Hopf algebra $\mathcal{H}^{ps}$.

There are  infinitely many edgeless graphs in pseudo-surfaces so, as with graphs, we need to consider a  quotient space of graphs in pseudo-surfaces.

\begin{definition}\label{s2.d2}

We will say that two graphs in pseudo-surfaces  $G_1\subset \Sigma_1$ and  $G_k\subset \Sigma_k$ are  \emph{LV-equivalent} if there is a sequence of graphs in pseudo-surfaces  $G_1\subset \Sigma_1, G_2\subset \Sigma_2, \ldots, G_k\subset \Sigma_k$ such that $G_i\subset \Sigma_i$ is obtained from  $G_{i-1}\subset \Sigma_{i-1}$, or vice versa, by one of the following moves. 
\begin{enumerate}
\item Deleting an isolated vertex, that is not a pinch point, from a graph.
\item Deleting a component of the pseudo-surface that contains no edges of the graph. 
\item Connect summing two pseudo-surface components (away from any graph components), or identifying a vertex in each of them to form a pinch point. 
\item Replacing a region, with another pseudo-surface with boundary (so that it forms a region of a new graph in a pseudo-surface).
\end{enumerate}
\end{definition}
It is clear that LV-equivalence gives rise to an equivalence relation, and we let $\mathcal{G}^{ps}$ denote the set of all equivalence classes of graphs in pseudo-surfaces considered up to LV-equivalence.  
It is easily seen that $\mathcal{G}^{ps}$ forms a minor system: 
\begin{lemma}\label{s2.l4}
The set  $\mathcal{G}^{ps}$ forms a minor system where the grading is given by the cardinality of the edge set, deletion and contraction are given by pseudo-surface deletion and contraction, and multiplication is given by disjoint union. 
\end{lemma}

\begin{definition}\label{s2.d3}
We let  $\mathcal{H}^{ps}$  denote the  Hopf algebra associated with $\mathcal{G}^{ps}$ via  Lemma~\ref{s2.l4} and Proposition~\ref{p.1}. 
Its coproduct is given by 
 $
 \Delta_{ps}( G \subset \Sigma)=\sum_{A\subseteq E} (G\backslash  A^c\subset \Sigma)  \otimes  (G/A \subset \Sigma/A)
$,
 where  pseudo-surface deletion and contraction are used.
\end{definition}

\begin{lemma}\label{s2.l5}
There is a natural Hopf algebra morphism $\phi: \mathcal{H}^{ps}\rightarrow \mathcal{H}^{mp}$ given by $\phi:G\rightarrow \mathbf{M}(G)$, where $\mathbf{M}(G)$ is the matroid perspective $ (B(G^{*}) \rightarrow C(G)) $.
\end{lemma}
\begin{proof}
First observe that if  $G\subset \Sigma$ and $H\subset \Sigma'$ are related by LV-equivalence then, up to the numbers of isolated vertices, their underlying graphs $G$ and $H$, and also the dual graphs $G^*$ and $H^*$, have the same  maximal 2-connected components.  Thus $M(G)=M(H)$, and  $M(G^*)=M(H^*)$. It follows that  $\mathbf{M}(G)=\mathbf{M}(H)$, and so $\phi$ is well-defined. 
It is easily seen that $\phi$ is multiplicative, and sends the (co)unit to the (co)unit. For the coproduct, in Section~4.2 of \cite{EMMlv} it was shown that $\mathbf{M}(G\subset \Sigma)/A=\mathbf{M}(G/A \subset \Sigma/A )$ and  $\mathbf{M}(G\subset \Sigma)\backslash A=\mathbf{M}(G\backslash A \subset \Sigma )$. Using this  we have 
%\begin{align*}
$\mathbf{M}(\Delta_{ps}(G\subset \Sigma))  %&=   
 \sum_{A\subseteq E} \mathbf{M}(G\backslash  A^c\subset \Sigma)  \otimes  \mathbf{M}(G/A \subset \Sigma/A)  
%\\&
=  \sum_{A\subseteq E} \mathbf{M}(G\subset \Sigma)\backslash  A^c \otimes  \mathbf{M}(G \subset \Sigma)/A 
%\\&
=\Delta_{ps}( \mathbf{M}(G\subset \Sigma) )$.
%\end{align*}
\end{proof}

We now determine the Tutte polynomial of the Hopf algebra $\mathcal{H}^{ps}$ .

\begin{lemma}\label{s2.l3}
$\mathcal{G}^{ps}_1$  has a basis consisting of exactly three elements represented by 
\begin{enumerate}
\item a 1-path in the sphere,
\item a loop in the sphere,
\item a loop that forms the meridian of a torus.
\end{enumerate}
\end{lemma}
\begin{proof}
Let $G\subset \Sigma$ be a graph in a pseudo-surface with exactly one edge $e$. This edge is either a bridge, a loop that is a quasi-loop, or a loop that is a quasi-bridge. In all three cases,  resolve each pinch point as in Figures~\ref{cont.c}--\ref{cont.b}, delete any isolated vertices, then remove any empty surface components. If $e$ is a bridge then the resulting graph in a pseudo-surface has exactly one region which can be replaced with a disc to give a 1-path in the sphere. If $e$ is a loop that is a quasi-loop then there are two regions each of which can be replaced with a disc to give a loop in the sphere. Finally, if $e$ is a loop that is a quasi-bridge then there is one region with two boundary components. The region can be replaced with an annulus to give a loop that forms the meridian of a torus.
\end{proof}

We identify a graph  in a pseudo-surface with its LV-equivalence class.  We set 
\begin{align*}
   \delta_{br}(G\subset \Sigma  ) :=& \begin{cases} 1 &\mbox{if $G\subset \Sigma  $ is a 1-path in the sphere,} \\   0 & \mbox{otherwise;} \end{cases}  
\\
  \delta_{ql}(G\subset \Sigma ) :=& \begin{cases} 1 &\mbox{if $G\subset \Sigma  $ is a loop in the sphere,} \\   0 & \mbox{otherwise;} \end{cases}  
\\
   \delta_{qb}(G\subset \Sigma ) :=& \begin{cases} 1 &\mbox{if $G\subset \Sigma  $ is a loop that is a meridian of a torus,}  \\   0 & \mbox{otherwise.} \end{cases}  
\end{align*}
Let
\begin{equation}\label{e.deps}
\delta_{\mathbf{a}}= \delta(x_1,x_2,x_3) :=x_1\delta_{br} + x_2\delta_{ql} +x_3\delta_{qb}.   
\end{equation}

 \begin{proof}[Proof of Theorem~\ref{s2.c1}]
 Upon verifying that $ \delta(G\subset \Sigma  ) =  \delta'(\mathbf{M}(G\subset \Sigma  ) )$, where $\delta$ is from Equation~\eqref{e.deps} and $\delta'$ is from Equation~\eqref{e.demp}, the result follows by applying  Theorems~\ref{tgen} and~\ref{s2.t1} using the Hopf algebra morphism from Lemma~\ref{s2.l5} then reinterpreting the matroid parameters of Theorem~\ref{s2.t1} in terms of graphs in pseudo-surfaces. 
\end{proof}

Analogously to Corollary~\ref{s2.t1.c2}, identities between the Tutte polynomial of a graph and the Las~Vergnas polynomial can be seen to be consequences of Hopf algebra maps.
\begin{corollary}\label{s2.t1.c3}
Let $\mathcal{H}^{ps}$ be the Hopf algebra of graphs in pseudo-surfaces from Definition~\ref{s2.d3}, and $\mathcal{H}^{g}$ be the Hopf algebra of graphs from Definition~\ref{d.hg}. Furthermore let $\mathcal{H}^{pg}$ be the Hopf subalgebra of $\mathcal{H}^{ps}$ generated by plane graphs.
  Then the following hold.
\begin{enumerate}
\item The  projection $\phi_1: \mathcal{H}^{pg} \rightarrow \mathcal{H}^{g}$ that takes a graph in a pseudo-surface to its underlying graph  is a Hopf algebra morphism.   Furthermore it naturally induces the identity that for a plane graph $G$, $T_G(x,y) = L_{G\subset \mathbb{R}^2}(x,y,z)$.
\item The  projection $\phi_2: \mathcal{H}^{ps} \rightarrow \mathcal{H}^{g}$ defined by $ \phi_2 (G\subset \Sigma)=G$ is a Hopf algebra morphism.   Furthermore it naturally induces the  identity $T_{G}(x,y) =  (y-1)^{n(G)-\kappa(G)} L_{G\subset \Sigma}(x,y,1/(y-1))$.
\end{enumerate}
\end{corollary}
\begin{proof}
It is readily verified that each of the three maps is a Hopf algebra morphism. To obtain the polynomial identities we apply Theorem~\ref{tgen} to Theorems~\ref{t.gr} and~\ref{s2.c1}.

For $\phi_1$, in Theorem~\ref{tgen} let $\mathcal{H}=\mathcal{H}^{pg}$,  $\mathcal{H'}=\mathcal{H}^{g}$,
 $\delta_{\mathcal{H}}$ be the selector used in Theorem~\ref{s2.c1}, and 
 $\delta_{\mathcal{H'}}$ be the selector used in Theorem~\ref{t.gr}. 
 Since  $\delta_{qb}(G\subset \mathbb{R}^2)$ is always zero, we have $  \delta_{\mathcal{H}}( x_1,x_2,x_3) (G\subset \mathbb{R}^2) =  \delta_{\mathcal{H'}}( x_1,x_2) ( G) $. By Theorems~\ref{tgen}, \ref{t.gr} and~\ref{s2.c1}, 
 \[  
 x_1^{r(G)} y_2^{|E(G)|-r(G)}    T_{G}( \tfrac{y_1}{x_1}+1, \tfrac{x_2}{y_2}+1)
 =
 x_1^{r(G)} y_2^{\kappa(G)} x_3^{n(G)-\kappa(G)}     L_{G\subset \mathbb{R}^2}( \tfrac{y_1}{x_1}+1, \tfrac{x_2}{y_2}+1, \tfrac{y_3}{x_3}  ) 
    \] 
 from which the result follows.
 This argument holds if $G\subset \Sigma$ is, more generally, in  $\mathcal{H}^{pg}$.

For $\phi_2$, in Theorem~\ref{tgen} let $\mathcal{H}=\mathcal{H}^{ps}$,  $\mathcal{H'}=\mathcal{H}^{g}$,
 $\delta_{\mathcal{H}}$ be the selector used in Theorem~\ref{s2.c1}, and 
 $\delta_{\mathcal{H'}}$ be the selector used in Theorem~\ref{t.gr}. 
 We have $  \delta_{\mathcal{H}}( x_1,x_2,x_2) (G\subset \Sigma) =  \delta_{\mathcal{H'}}( x_1,x_2) ( G) $ (since   $\delta_{\mathcal{H}}( x_1,x_2,x_2)$ no longer distinguishes how a loop lies in the pseudo-surface). By Theorems~\ref{tgen}, \ref{t.gr} and~\ref{s2.c1}, 
 \[  
 x_1^{r(G)} y_2^{|E(G)|-r(G)}    T_{G}( \tfrac{y_1}{x_1}+1, \tfrac{x_2}{y_2}+1)
 =
 x_1^{r(G)} y_2^{\kappa(G)} x_2^{n(G)-\kappa(G)}     L_{G\subset \mathbb{R}^2}( \tfrac{y_1}{x_1}+1, \tfrac{x_2}{y_2}+1, \tfrac{y_2}{x_2}  ) 
    \] 
 from which the result follows.
\end{proof}

\subsection{Delta-matroids and the Bollob\'as-Riordan polynomial}\label{s.BRdm}
Our notation for delta-matroids follows~\cite{CMNR1,CMNR2} and we refer the reader to these references for background on them.

Let $D=(E,{\mathcal{F}})$ be a delta-matroid and  $A\subseteq E$. The  \emph{twist} of $D=(E,{\mathcal{F}})$ with respect to $A$  is $D* A:=(E,\{A\triangle  X \mid  X\in \mathcal{F}\})$.
The \emph{dual} of $D$, written $D^*$, is equal to $D*E$.
% The twist of a delta-matroid is a delta-matroid (see  \cite{ab1}).
 %
The feasible sets of  $D$ are graded by  cardinality. Let $\mathcal{F}_{\max}(D)$ and $\mathcal{F}_{\min}(D)$ be the set of feasible sets of maximum and minimum cardinality, respectively. We will usually omit $D$ when the context is clear.   If   the sets in $\mathcal{F}_{\min}$  (respectively, $\mathcal{F}_{\max}$) are of cardinality $m$ and     $k\in \mathbb{Z}$, then $\mathcal{F}_{\min+k}$ (respectively, $\mathcal{F}_{\max+k}$) denotes the set of feasible sets in $\mathcal{F}$ of cardinality $m+k$.

Let $D_{\max}:=(E,\mathcal{F}_{\max})$ and  $D_{\min}:=(E,\mathcal{F}_{\min})$.
Then $D_{\max}$ is the \emph{upper matroid} and $D_{\min}$ is the \emph{lower matroid} for $D$. %, defined by Bouchet in~\cite{ab2}.
%It is straight-forward to show that the upper matroid and the lower matroid are indeed matroids. 
 Let $r_{\max}$ and $r_{\min}$, respectively,  denote the rank functions of these two matroids. We define a function $\rho$ on delta-matroids by
\[ \rho(D) :=  \frac{1}{2}( r_{\max}(D)  +   r_{\min}(D) ) , \]
and for $A\subseteq E$,
\[ \rho(A) :=   \rho(D\backslash A^c)   .  \]
Observe that if $D$ is a matroid then $D_{\max}= D_{\min}$ and  $\rho$ is precisely its rank function. 
It is important to notice that in general   $\rho_D(A)\neq    \frac{1}{2}( r_{D_{\max}}(A)  +   r_{D_{\min}}(A) )$.

Defined in \cite{CMNR1}, the \emph{(2-variable) Bollob\'as-Riordan polynomial}, $\tilde{R}_D(x,y)$ is
\begin{equation}\label{s2.e3}
\tilde{R}_D(x,y) :=  \sum_{A\subseteq E}  (x-1)^{\rho(E)-\rho(A)}(y-1)^{|A|-\rho(A)}.
\end{equation} 
Note that \eqref{s2.e3} is  obtained by replacing $r$ for $\rho$ in the definition of the Tutte polynomial. It is the extension of the 2-variable version of Bollob\'as and Riordan's ribbon graph polynomial \cite{BR01,BR02} to delta-matroids.

We will  show that the 2-variable Bollob\'as-Riordan polynomial is the canonical  Tutte polynomial  of delta-matroids with their usual deletion and contraction. 

The following result is easily verified.
\begin{lemma}\label{l.hopfdm.1}
The set of isomorphism classes of delta-matroids  form a minor system where the grading is given by the cardinality of the ground set, deletion and contraction are given by delta-matroid deletion and contraction, and multiplication is given by direct sum. 
\end{lemma}
For convenience we will henceforth identify a delta-matroid with its isomorphism class. 

\begin{definition}\label{d.hopfdm}
We let  $\mathcal{H}^{dm}$  denote the  Hopf algebra associated with delta-matroids via  Lemma~\ref{l.hopfdm.1}  and Proposition~\ref{p.1}. 
Its coproduct is given by 
$
 \Delta_{dm}(D)=\sum_{A\subseteq E}D\backslash A^c  \otimes  D/ A
$,
 where  $D=(E,\mathcal{F})$ is a delta-matroid and the deletion and contraction are delta-matroid deletion and contraction.
\end{definition}

Up to isomorphism, there are exactly three delta-matroids over one element:   
$D_c:=( \{e\}, \{  \{e\} \})$, $D_o:=( \{e\}, \{  \emptyset \})$, and $D_n:=( \{e\}, \{  \emptyset , \{e\} \})$. 
 Accordingly, set 
\begin{equation}\label{e.dm.del}  
 \delta_{c}(D ) := \begin{cases} 1 &\mbox{if } D =D_c, \\   0 & \mbox{otherwise}; \end{cases} 
 \quad 
  \delta_{o}(D ) := \begin{cases} 1 &\mbox{if } D =D_o, \\   0 & \mbox{otherwise}; \end{cases} 
  \quad
  \delta_{n}(D ) := \begin{cases} 1 &\mbox{if } D =D_n ,\\   0 & \mbox{otherwise}. \end{cases}
 \end{equation}
(Using  notation defined after the statement of Theorem~\ref{s2.t3} below, the subscripts of the $\delta$'s record, in order, if each delta-matroids is a \underline{c}oloop, \underline{o}rientable ribbon-loop, or \underline{n}on-orientable ribbon-loop.)
We set
\begin{equation}\label{e.BRdelta1}
 \delta_{\mathbf{a}}= \delta(a_1,a_2,a_3) := a_1\delta_{c} + a_2\delta_{o} +a_3\delta_{n}.  
  \end{equation}
Then the Tutte polynomial of $\mathcal{H}^{dm}$  is defined by
\begin{equation}\label{e.BRdelta2}
\alpha( \mathbf{a}, \mathbf{b})  := \exp_* (\delta_{\mathbf{a}}) \ast \exp_* (\delta_{\mathbf{b}}) . 
\end{equation}

For a canonical Tutte polynomial we require that $\delta_{\mathbf{a}}$ is uniform. By applying $\delta\otimes \delta$ to  $\Delta(D)$, where $D$ is over $E=\{e,f\}$ and has feasible sets $\emptyset$, $\{e\}$, and   $\{e,f \}$,  it  is seen that $\delta_{\mathbf{a}}$ is uniform only if $a_3=\sqrt{a_1a_2}$. Thus the Tutte polynomial of a delta-matroid will be a 2-variable polynomial, rather than a 3-variable polynomial (which is perhaps  unexpected given that there are three delta-matroids of graded dimension 1).

\begin{theorem}\label{s2.t3}
The 2-variable Bollob\'as-Riordan polynomial arises as the canonical Tutte polynomial of the  Hopf algebra $\mathcal{H}^{dm}$: 
 \begin{equation}
 \alpha( \mathbf{a}, \mathbf{b}) ( D)  =  x_1^{ \rho(D)}   y_2^{|E|- \rho(D)} \tilde{R}_D\left(\frac{y_1}{x_1}+1, \frac{x_2}{y_2} +1\right),
 \end{equation}
where $\mathbf{a} = ( x_1,x_2,\sqrt{x_1x_2})$, $\mathbf{b}  =(y_1,y_2,\sqrt{y_1y_2})$, and $E$ is the ground set of the delta-matroid $D$. 
\end{theorem}
The proof of Theorem~\ref{s2.t3} is similar in structure to that of Theorem~\ref{s2.t1}, and will follow from a sequence of lemmas.

It is convenient for us to relate $D|_e$, the restriction of $D$ to $e$, to the element $e$ of $D$. For this we need a little additional notation. Let $D=(E, \mathcal{F})$ be a delta-matroid, and $e\in E$. 
 Then  $e$ is a {\em ribbon loop} if $e$ is a loop in $D_{\min}$.
  A ribbon loop $e$ is  is \emph{orientable} if $e$ is not a loop in $(D\ast e)_{\min}$,
  and is  \emph{non-orientable} if $e$ is a  loop in $(D\ast e)_{\min}$.
 % (We remark that, just as matroid terminology often follows graph terminology, this delta-matroid terminology arises from the connection between delta-matroids and ribbon graphs, as discussed in Section~\ref{s.BRrg} below.) 
  Note that it can be determined if $e$ is a (orientable/non-orientable) ribbon loop by looking for its membership in sets in $\mathcal{F}_{\min}$ and   $\mathcal{F}_{\min+1}$.

\begin{lemma}\label{s2.l9}
 Let $D=(E, \mathcal{F})$ be a delta-matroid, and $e\in E$.   Then   $e$ is not a ribbon loop  (is an orientable ribbon loop,  is a non-orientable ribbon loop, respectively) if and only if  $D|_e$ is isomorphic to $D_c$  ($D_o$, $D_n$, respectively).
\end{lemma}
\begin{proof}
Let   $f\in E$ with  $f\neq e$. 
We start by showing that $e$ is a ribbon loop in $D$ if and only if it is a ribbon loop in $D\backslash f$.
This  is easily seen to be true if $f$ is a coloop, so suppose it is not. Then there is some $Z\in \mathcal{F}$ with $f\notin Z$.
 
 Suppose that $e$ is not a ribbon loop. 
Then there is some $X\in \mathcal{F}(D)_{\min}$ with $e\in X$. We show there is some $Y\in \mathcal{F}(D)_{\min}$ such that  $e\in Y$ but $f\notin Y$. From this it immediately follows that  $e\in Y\in \mathcal{F} (D\backslash f )_{\min}$ and so  $e$ is not a ribbon loop in $D\backslash f$.  To construct $Y$, if $f\notin X$ take $Y=X$, otherwise $f\in X$ and so  $f\in X \triangle Z$, where $Z$ is as above. The Symmetric Exchange Axiom gives that there is some $X\triangle \{f,u\} \in \mathcal{F}$. By the minimality of $X$, $u\neq e$ and $X\triangle \{f,u\} \in \mathcal{F}(D)_{\min}$. Take   $Y=X\triangle \{f,u\}$. It follows that $e$ is not a ribbon loop in $D\backslash f$. 

For the converse suppose that $e$ is a ribbon loop. Then $e$ is not in any feasible set in $\mathcal{F}(D)_{\min}$. To show that $e$ is not in any feasible set in $\mathcal{F}(D\backslash f)_{\min}$ it is enough to show that there is some set in $\mathcal{F}(D)_{\min}$ that does not contain $f$. Choose  $X\in \mathcal{F}(D)_{\min}$. If $f\notin X$ we are done, otherwise we have  $f\in X\in \mathcal{F}(D)_{\min}$. Then 
 $f\in X \triangle Z$,  and the Symmetric Exchange Axiom gives that there is some $X\triangle \{f,u\} \in \mathcal{F}$. The minimality of $X$ gives that  $f\notin X\triangle \{f,u\} \in \mathcal{F}(D)_{\min}$. This is the required set. Thus  $e$ is a ribbon loop in $D\backslash f$. 

We have just shown that $e$ is a ribbon loop in $D$ if and only if it is one in $D\backslash f$. 
Next we show that a ribbon loop $e$ is orientable in $D$ if and only if it is orientable in $D\backslash f$. 
Again this  is easily seen to be true if $f$ is a coloop, so suppose it is not. Then there is some $Z\in \mathcal{F}$ with $f\notin Z$.

If $e$ is a non-orientable ribbon loop, then there is some $X\in \mathcal{F}(D)_{\min+1}$ with $e\in X$. We show there is some $Y\in \mathcal{F}(D)_{\min+1}$ with $e\in Y$ and $f\notin Y$. We have seen above (in the argument that if $e$ is a ribbon loop in $D$ then it is one in $D\backslash f$)  that there is a set in  $Y\in \mathcal{F}(D)_{\min}$ not containing $f$  and no sets in  $\mathcal{F}(D)_{\min}$ contain $e$, it follows that  $e$ is a non-orientable ribbon loop in $\mathcal{F}(D\backslash f)$.   
 To construct $Y$, if $f\notin X$ take $Y=X$, otherwise $f\in X$ and so  $f\in X \triangle Z$, where $Z$ is as above. The Symmetric Exchange Axiom gives that there is some $X\triangle \{f,u\} \in \mathcal{F}$.  The set $X\triangle \{f,u\}$ must be in $ \mathcal{F}(D)_{\min}$ or $ \mathcal{F}(D)_{\min+1}$, but since it contains $e$ it must be in  $ \mathcal{F}(D)_{\min+1}$. Thus taking $Y=X\triangle \{f,u\}$ gives the required set.
 
 Conversely,  if $e$ is an orientable ribbon loop, then no element of   $\mathcal{F}(D)_{\min}$  or  $\mathcal{F}(D)_{\min+1}$ contains  $e$.   We have seen above (in the argument that if $e$ is a ribbon loop in $D$ then it is one in $D\backslash f$)  that there is a set in $\mathcal{F}(D)_{\min}$ that does not contain $f$. It follows that  no element of   $\mathcal{F}(D\backslash f)_{\min}$  or  $\mathcal{F}(D\backslash f)_{\min+1}$ will contain $e$, so  $e$ is an orientable ribbon loop of $D\backslash f$.  
 Thus we have shown that a ribbon loop $e$ is orientable in $D$ if and only if it is orientable in $D\backslash f$.

 Finally, since $e$ is not a ribbon loop  (an orientable ribbon loop,  a non-orientable ribbon loop, respectively) in $D$ if and only if it is one in $D\backslash f$,  the result stated in the lemma follows by deleting the edges in $E\backslash e$ one at a time.
\end{proof}

\begin{lemma}\label{s2.l7}
Let $D=(E, \mathcal{F})$ be a delta-matroid, and $e\in E$. Then 
\[  \rho(D) = \begin{cases}    \rho(D/e) +1 &  \text{if $e$ is not a ribbon loop},    \\  \rho(D/e) &  \text{if $e$ is  an orientable ribbon loop},    \\  \rho(D/e) +\frac{1}{2} &  \text{if $e$ is  a non-orientable ribbon loop.}\end{cases}   \]
\end{lemma}
\begin{proof}
We prove the lemma by computing $r_{\max}(D)$ and $r_{\min}(D)$  which, respectively, equal the maximum and minimum cardinalities of the feasible sets in   $\mathcal{F}(D)$.

First suppose that $e$ is not a ribbon loop, so $e$ is in some feasible set in   $\mathcal{F}(D)_{\min}$. Since $e$ is not a loop, $\mathcal{F}(D/e)=\{ X\backslash e \mid X\in \mathcal{F} \text{ and } e\in X \}$ and it follows $r_{\min}(D/e)=r_{\min}(D)-1$. 
For $r_{\max}$, we first show that $e$ appears in some element of $\mathcal{F}(D)_{\max}$.  Let $X\in  \mathcal{F}(D)_{\max}$. If $e\in X$ we are done, otherwise choose some $Y\in \mathcal{F}(D)$ that contains $e$ (this exists since $e$ is not a ribbon loop). Then $e\in X\triangle Y$ and so the Symmetric Exchange Axiom gives $ X \triangle \{e,u\} \in \mathcal{F}(D)$. By the maximality of $X$, we have $X \triangle \{e,u\}\in    \mathcal{F}(D)_{\max}$  and so $e$ appears in some element of $\mathcal{F}(D)_{\max}$. It then follows from the definition of contraction that $r_{\max}(D/e)=r_{\max}(D)-1$  (observe that this argument holds as long as $e$ is not a loop. We will use this fact below.). Thus  $\rho(D) =\tfrac{1}{2} ( r_{\max}(D) +r_{\min}(D)  )    =   \tfrac{1}{2} ( r_{\max}(D/e) +r_{\min}(D/e)+2) = \rho(D/e)$. 

Next suppose that $e$ is  a non-orientable ribbon loop. In particular, $e$ is not a loop.  We have $e$ is not in any element of $\mathcal{F}(D)_{\min} $ nor any element of  $\mathcal{F}(D\ast e)_{\min} $. It is not hard to see that the latter implies that $e$ is in some element of $ \mathcal{F}(D)_{\min+1}$. Then, since $\mathcal{F}(D/e)=\{ X\backslash e \mid X\in \mathcal{F} \text{ and } e\in X \}$, it follows that $r_{\min}(D/e)=r_{\min}(D)$. 
The identity $r_{\max}(D/e)=r_{\max}(D)-1$ follows as in the case of when $e$ is not a ribbon loop above. Thus we have that  $\rho(D)  =\tfrac{1}{2} ( r_{\max}(D) +r_{\min}(D)  )    =   \tfrac{1}{2} ( r_{\max}(D/e) +r_{\min}(D/e)+1)= \rho(D/e)+\tfrac{1}{2} $.

Finally suppose that $e$ is  an orientable ribbon loop. If $e$ is a loop then $\mathcal{F}(D)=\mathcal{F}(D/e)$ and so 
$\rho(D) =\tfrac{1}{2} ( r_{\max}(D) +r_{\min}(D)  )    =   \tfrac{1}{2} ( r_{\max}(D/e) +r_{\min}(D/e)) = \rho(D/e)$. Now suppose that $e$ is not a loop.  
Since $e$ is  an orientable ribbon loop, $e$ is not in any element of $\mathcal{F}(D)_{\min} $ but is in some element of  $\mathcal{F}(D\ast e)_{\min} $. It is not hard to see that the latter implies that $e$ is not in any  element of $ \mathcal{F}(D)_{\min+1}$. We show that $e$ is  in some  element of $ \mathcal{F}(D)_{\min+2}$, from which it follows immediately from the definition of $\mathcal{F}(D/e)$ that $r_{\min}(D/e)=r_{\min}(D)+1$. Choose $X\in   \mathcal{F}(D)_{\min}$ and a feasible set $Y\in  \mathcal{F}(D)$  that contains $e$. Since  $e\notin X$, the Symmetric Exchange Axiom gives that $ X\triangle \{e,u\} \in \mathcal{F}$. By the minimality of $X$ and since no set in $\mathcal{F}(D)_{\min}$ or $\mathcal{F}(D)_{\min+1}$ contains $e$  it follows that $X\triangle \{e,u\} \in \mathcal{F}(D)_{\min+2}$ and contains $e$.  
The identity $r_{\max}(D/e)=r_{\max}(D)-1$ again follows as in the case of when $e$ is not a ribbon loop above. Thus, by the definition of $\rho$, we have that  $\rho(D) = \tfrac{1}{2} ( r_{\max}(D) +r_{\min}(D)  )    =   \tfrac{1}{2} ( r_{\max}(D/e) +1 +r_{\min}(D/e)-1) = \rho(D/e)$. 
\end{proof}
Observe that this proof gives that $r_{\max}(D/e)=r_{\max}(D)-1$ when $e$ is not a loop, and $r_{\max}(D/e)=r_{\max}(D)$ when it is. We will use this observation in Section~\ref{ss.penrose}.

\begin{proof}[Proof of Theorem~\ref{s2.t3}]
Applying Theorem~\ref{tm} with $r_1(A) := \rho(A)$, and  $r_2(A) :=|A|- \rho(A)$ (so via Lemmas~\ref{s2.l9} and \ref{s2.l7}    $a_1=x_1^1x_2^0$, $a_2=x_1^0x_2^1$, $a_3=x_1^{1/2}x_2^{1/2}$) gives 
\begin{align*}
 \alpha( \mathbf{a}, \mathbf{b}) ( D)  &= y_1^{ \rho(D)}   y_2^{|E|- \rho(D)}  \sum_{A\subseteq E} \left(\frac{x_1}{y_1}\right)^{\rho(A)}\left(\frac{x_2}{y_2} \right)^{|A|-\rho(A)}
 \\&= x_1^{ \rho(D)}   y_2^{|E|- \rho(D)}  \sum_{A\subseteq E} \left(\frac{y_1}{x_1}\right)^{\rho(D)-\rho(A)}\left(\frac{x_2}{y_2} \right)^{|A|-\rho(A)},
\end{align*}
from which the result follows.
\end{proof}

We can use Hopf algebra mappings to show that the 2-variable Bollob\'as-Riordan polynomial, $\widetilde{R}_D(x,y)$, extends the Tutte polynomial from matroids to delta-matroids.
\begin{corollary}\label{s2.t3.c2}
Let $\mathcal{H}^{dm}$ be the Hopf algebra of delta-matroids from Definition~\ref{d.hopfdm}, and $\mathcal{H}^{m}$ be the Hopf algebra of matroids  from Definition~\ref{d.hm}. Then the inclusion $\phi: \mathcal{H}^{m} \rightarrow \mathcal{H}^{dm}$ is a Hopf algebra morphism.  Furthermore it naturally induces the  identity $T_M(x,y) = \tilde{R}_D(x,y)$.
\end{corollary}
\begin{proof}
That the map  is a Hopf algebra morphism follows readily from the fact that delta-matroids  restrict to matroids in a way compatible with  the standard constructions of deletion, contraction, etc. (recall a matroid is a delta-matroid).   

In Theorem~\ref{tgen} let $\mathcal{H}=\mathcal{H}^{m}$,  $\mathcal{H'}=\mathcal{H}^{dm}$,
 $\delta_{\mathcal{H}}$ be the selector used in Theorem~\ref{t.mat}, and 
 $\delta_{\mathcal{H'}}$ be the selector used in Theorem~\ref{s2.t3}. 
Then since matroids are closed under deletion and contraction and  $D_n$ is not a matroid, $D_n$  will never appear as a term in $\delta_{\mathcal{H}}^{(k)} (M)$. Therefore
 $  \delta_{\mathcal{H}}(x_1,x_2) ( M )  =  \delta_{\mathcal{H}'}( x_1,x_2,\sqrt{x_1x_2}) (M) $.
  By Theorem~\ref{tgen}, it follows that 
 \[ \alpha_{\mathcal{H}}(  x_1,x_2, y_1,y_2) ( M)  =   \alpha_{\mathcal{H}'}( x_1,x_2, \sqrt{x_1x_2},y_1,y_2,\sqrt{y_1y_2})  (M) .\]
 Theorems~\ref{t.mat} and~\ref{s2.t3} give 
 \[  
  x_1^{r(M)} y_2^{|E|-r(M)}    T_{M}\left(  \frac{y_1}{x_1}+1, \frac{x_2}{y_2}+1  \right)
=
  x_1^{ \rho(M)}   y_2^{|E|- \rho(M)} \tilde{R}_M\left(\frac{y_1}{x_1}+1, \frac{x_2}{y_2} +1\right),
      \] 
 from which the result follows upon noting that if $M$ is a matroid, $r(M)=\rho(M)$.
\end{proof}

\subsection{Bollob\'as and Riordan's ribbon graph polynomial}\label{s.BRrg}
We use standard ribbon graph terminology following \cite{EMMbook}.
For a ribbon graph $G=(V,E)$   we set $e(G):=|E|$, $v(G):=|V|$, $c(G)$ to be is its number of components,  $f(G)$ its number of boundary components, and $\gamma (G)$ its Euler genus (which is twice its genus if it is orientable and its genus if it is not). The rank of $G$ is $r(G):=v(G)-c(G)$.
\emph{Euler's formula} gives that $v(G)-e(A)+f(A)=2c(A)-\gamma(A)$.

The \emph{Bollob\'as-Riordan polynomial} of \cite{BR02} is defined as
\begin{equation}\label{dhdf}
R_G(x,y,z) := \sum_{A \subseteq E}   (x-1)^{r( G ) - r( A )}   y^{|A|-r(A)} z^{\gamma(A)}. 
\end{equation}
In this section we focus on the  \emph{2-variable Bollob\'as-Riordan polynomial} 
\begin{equation}\label{gjs}
\widetilde{R}_G( x,y):= \sum_{A \subseteq E}   (x-1)^{\rho( G ) - \rho( A )}   (y-1)^{|A|-\rho(A)},
\end{equation}
where  
\begin{equation}\label{e.rgrho}
\rho(A) := \tfrac{1}{2}\left(|A|+ v(A)-f(A)\right),
\end{equation}
 and  $ \rho( G ):=\rho( E ) $.

Euler's formula can be used to relate the two versions of the  Bollob\'as-Riordan polynomial: 
\[
\widetilde{R}_G(x+1,y+1)=x^{\frac{1}{2} \gamma(G)}R_G(x+1,y,1/\sqrt{xy}).
\]

We turn to the Hopf algebra of ribbon graphs. As was the case with graphs, to ensure a single element of graded dimension zero in the Hopf algebra we work with equivalence classes of ribbon graphs. For this,  let $G=(V,E)$ be a ribbon graph, $v\in V$, and $P$ and $Q$ be non-trivial ribbon subgraphs of $G$. Then $G$ is said to be the \emph{join} of $P$ and $Q$, written $P\vee Q$, if $G=P\cup Q$ and $P\cap Q=\{v\}$ and if there exists an arc on $v$ with the property that all edges of $P$ incident to $v$ meet it there, and none of the edges of $Q$ do. Note that the parameters $\gamma$, $r$, $e$, and $\rho$ are invariant under joins and disjoint unions (but $v$ and $c$ are not).  

We omit the straightforward proof of the following lemma.
\begin{lemma}\label{l.rgmin}
The set of equivalence classes of ribbon graphs considered up to joins and isomorphism  forms a minor system where the grading is given by the cardinality of the edge set, deletion and contraction are given by ribbon graph deletion and contraction, and multiplication is given by direct sum.
\end{lemma}
 For convenience we usually identify a ribbon graph with its equivalence class. 
\begin{definition}\label{d.hopfrg}
Let  $\mathcal{H}^{rg}$  denote the  Hopf algebra associated with ribbon graphs via  Lemma~\ref{l.rgmin} and Proposition~\ref{p.1}. 
Its coproduct is given by 
$
 \Delta_{rg}(G)=\sum_{A\subseteq E(G)}D\backslash A^c  \otimes  D/ A
$,
 where the deletion and contraction are ribbon graph deletion and contraction.
\end{definition}

There are exactly three elements of $\mathcal{H}^{rg} $ of graded dimension 1. Accordingly we set
   \begin{equation}\label{e.rgic}
\delta_{\mathrm{b}} (G) =\begin{cases} 1 &\mbox{if } 
G =  \raisebox{-1.6mm}{\includegraphics[scale=0.15]{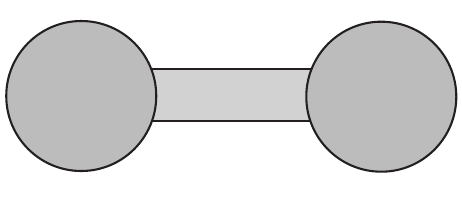}}, \\
0 & \mbox{otherwise} ;\end{cases} 
\quad \quad  
\delta_{\mathrm{o}} (G) =\begin{cases} 1 &\mbox{if } 
G = \raisebox{4.5mm}{\rotatebox{-90}{\includegraphics[scale=0.2]{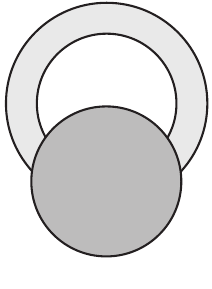}}} ,
\\
0 & \mbox{otherwise}; \end{cases}
\quad \quad  
\delta_{\mathrm{n}} (G) =\begin{cases} 1 &\mbox{if } 
G =\raisebox{4.5mm}{\rotatebox{-90}{\includegraphics[scale=0.2]{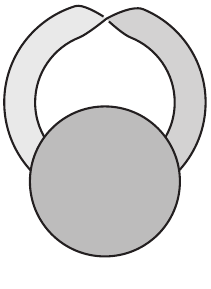}}} ,
\\
0 & \mbox{otherwise}.\end{cases}
  \end{equation}  
Then 
\begin{equation}\label{e.rgbr1}
 \delta_{\mathbf{a}}= \delta(a_1,a_2,a_3) := a_1\delta_{b} + a_2\delta_{o} +a_3\delta_{n}.
  \end{equation}
By considering the ribbon graph that describes a graph with one vertex and two loops on a Klein bottle, it can be seen that  $\delta_{\mathbf{a}}$ is not uniform unless $a_3=\sqrt{a_1a_2}$. The following theorem will show that the converse holds, so    $\delta_{\mathbf{a}}$ is  uniform if and only if $a_3=\sqrt{a_1a_2}$.

\begin{theorem}\label{t.rgbr1}
The 2-variable Bollob\'as-Riordan polynomial arises as the Tutte polynomial of the  Hopf algebra $\mathcal{H}^{rg}$, 
 \[
 \alpha( \mathbf{a}, \mathbf{b}) ( G)  = x_1^{ \rho(G)}   y_2^{|E|- \rho(G)} \tilde{R}_G\left(\frac{y_1}{x_1}+1, \frac{x_2}{y_2} +1\right),
 \]
where $\mathbf{a} = ( x_1,x_2,\sqrt{x_1x_2})$, $\mathbf{b}  =(y_1,y_2,\sqrt{y_1y_2})$, and $E=E(G)$. 
\end{theorem}

We proceed as we did in Section~\ref{s.gr} when we recovered the Tutte polynomial for graphs from that for matroids via a Hopf algebra map.  A \emph{quasi-tree} is a ribbon graph with exactly one boundary component, so $f(G)=1$. 
%Just as the spanning trees in a graph give rise to a graphic matroid, the spanning quasi-trees in a ribbon graph give rise to a delta-matroid.
 If $G$ is a ribbon graph let $D(G):=(E, \mathcal{F}(G))$, where $\mathcal{F}(G)$ consists of the edge sets of the spanning subgraphs of $G$ that restrict to a quasi-tree in each connected component of $G$, i.e.,     $\mathcal{F}(G):= \{ A\mid f(A)=c(G)  \} $.  It was shown in \cite{ab2,CMNR1} that  $D(G)$ is a delta-matroid.

\begin{lemma}\label{t.rgbr.l1}
There is a natural Hopf algebra morphism $\phi: \mathcal{H}^{rg}\rightarrow \mathcal{H}^{dm}$ given by $\phi:G\rightarrow D(G)$. 
\end{lemma}
\begin{proof}
Since $D(G\vee H)=D(G)\oplus D(H)$, $\phi$ is well-defined. 
It is easily seen that $\phi$ is multiplicative, and sends the (co)unit to the (co)unit.
It was shown in \cite{CMNR1} that $D(G)/A=D(G/A)$ and  $D(G)\backslash A=D(G\backslash A)$, giving
$  D(\Delta_{rg}(G))  =    \sum_{A\subseteq E} D(G\backslash  A^c)  \otimes  D(G/A)  =  \sum_{A\subseteq E} D(G)\backslash  A^c \otimes  D(G)/A  = \Delta_{dm}(D(G)) $.
\end{proof}
We will use $\phi$ to identify the Tutte polynomial of $\mathcal{H}^{rg}$.

\begin{proof}[Proof of Theorem~\ref{t.rgbr1}]
Upon verifying that $ \delta_{b}(G)=  \delta'_{c}(D(G))$, $ \delta_{o}(G)=  \delta'_{o}(D(G))$,  and $ \delta_{n}(G)=  \delta'_{n}(C(G))$, where the primed $\delta$'s are those of Equation~\eqref{e.dm.del}, Theorems~\ref{tgen} and~\ref{s2.t3} give
\[ \alpha(\mathbf{a}, \mathbf{b}) ( G) = x_1^{ \rho(D(G))}   y_2^{|E|- \rho(D(G))} 
\sum_{A\subseteq E }\frac{y_1}{x_1}^{\rho_{D(G)}(E)-\rho_{D(G)}(A)}\frac{x_2}{y_2}^{|A|-\rho_{D(G)}(A)},\]
where $E=E(D(G))$.
It remains to show that for any $A\subseteq E$,
\begin{equation}\label{e.t.rgbr1}\rho_{D(G)}(A) =  \rho_{G}(A).\end{equation}  It was shown in \cite{ab2,CMNR1} that 
 $D(G)_{\min}=C(G)$ and  $D(G)_{\max}=B(G^*)=(C(G^*))^*$. Then using the rank functions for cycle matroids and dual matroids we get 
%\begin{multline*} 
$2 \rho_{D(G)}(A)  =     2 \rho (D(G)\backslash A^c)   
=  r_{\max}( D(G)\backslash A^c  )   +  r_{\min}( D(G)\backslash A^c  )  
 =   r_{\max}( D(G\backslash A^c)  )   +  r_{\min}( D(G\backslash A^c ) )   
 =     |A| -     r((G\backslash A^c)^*)  + r(G\backslash A^c)  
 = |A|  -  v((G\backslash A^c)^*) +  c((G\backslash A^c)^*)  + v(G\backslash A^c) - c(G\backslash A^c)
= |A|  + v(G\backslash A^c) -   f(G\backslash A^c) 
=2 \rho (G\backslash A^c)
=2 \rho_G (A),$
%\end{multline*}
and the result follows.
\end{proof}

Note that in the proof of Theorem~\ref{t.rgbr1} we have shown that 
\begin{equation}\label{t.rgbr.e}
\widetilde{R}_G(x,y) = \widetilde{R}_{D(G)}(x,y),
\end{equation}
and that this identity is naturally induced by the Hopf algebra morphism $\phi$ of Lemma~\ref{t.rgbr.l1}.

\begin{corollary}\label{t.rgbr1.c1}
 \begin{equation}
    \tilde{R}_G(x,y) = \alpha( 1, y-1,\sqrt{y-1} ,x-1,1,\sqrt{x-1}) ( G ).
 \end{equation}
\end{corollary}

\begin{corollary}\label{t.rgbr1.c2}
Let $\mathcal{H}^{prg}$ be the Hopf subalgebra of $\mathcal{H}^{rg}$ generated by plane graphs (i.e., ribbon graphs of genus 0), and $\mathcal{H}^{g}$ be the Hopf algebra of graphs from Definition~\ref{d.hg}.  Then 
the  projection $\phi: \mathcal{H}^{prg} \rightarrow \mathcal{H}^{g}$ defined by setting $ \phi (G)$ to be the underlying graph of $G$ is a Hopf algebra morphism.   Furthermore it naturally induces the  identity $T_{G}(x,y) =   \tilde{R}_G(x,y)$.
\end{corollary}
\begin{proof}
Clearly $\phi(G\backslash e) =  \phi(G)\backslash e$, for each $e\in E(G)$. If $e$ is not a loop then it is also clear that $\phi(G/ e) =  \phi(G)/ e$. If $e$ is a loop then, since $G$ is plane,  there can be  no cycle $C$ of $G$ interlaced with $e$ (i.e, at the vertex $v$ which meets $e$, no cycle $C$ appears in the cyclic order $eCeC$ at that vertex). Thus $G/e = G_1 \sqcup G_2= G_1 \vee G_2$, where we have used the fact that elements of $\mathcal{H}^{prg}$ are considered modulo joins. It is then not hard to see that  $\phi(G/ e) = \phi(G_1 \vee G_2) = \phi(G)/e$. From these observations it follows easily that $\phi$  is a Hopf algebra morphism.

To obtain the polynomial identities we apply Theorem~\ref{tgen} to Theorems~\ref{t.rgbr1} and~\ref{t.gr}.
In Theorem~\ref{tgen} let $\mathcal{H}=\mathcal{H}^{prg}$,  $\mathcal{H'}=\mathcal{H}^{g}$,
 $\delta_{\mathcal{H}}$ be the selector used in Theorem~\ref{t.rgbr1}, and 
 $\delta_{\mathcal{H'}}$ be the selector used in Theorem~\ref{t.gr}. 
  By Theorems~\ref{tgen}, \ref{t.rgbr1} and~\ref{t.gr}, 
 \[  
 x_1^{ \rho(G)}   y_2^{|E|- \rho(G)} \tilde{R}_G\left(\frac{y_1}{x_1}+1, \frac{x_2}{y_2} +1\right) =    x_1^{r(\phi(G))} y_2^{|E(\phi(G))|-r(\phi(G))}    T_{\phi(G)}( \frac{y_1}{x_1}+1, \frac{x_2}{y_2}+1).
    \] 
Since $G$ is plane, Euler's formula then gives $v(G)-e(G)+f(G)=2c(G)$. Substituting for $e(G)$ in Equation~\eqref{e.rgrho} gives 
$ \rho(G) = r(\phi(G))$, from which the result follows.
\end{proof}

\subsection{The three-variable Bollob\'as-Riordan polynomial}\label{dasg} 
Here we determine the minor system that gives rise  to the (3-variable) Bollob\'as-Riordan polynomial, \eqref{dhdf}. We will see that $R_G(x,y,z)$ is not  associated with ribbon graphs, but rather ribbon graphs whose vertex set has been partitioned.

A \emph{vertex partitioned ribbon graph}, $(G,\mathcal{P})$ consists of a ribbon graph $G=(V,E)$ and a partition  $\mathcal{P}$ of its vertex set $V$. 
Deletion and contraction for  vertex partitioned ribbon graphs is defined in the natural way.  If $e \in E(G)$, then \emph{deletion} is defined by $(G,\mathcal{P}) \backslash e:=  (G\backslash e,\mathcal{P})$. \emph{Contraction} is defined by $(G,\mathcal{P}) / e:=  (G/ e,\mathcal{P}')$, where the partition $\mathcal{P}'$ is induced by $\mathcal{P}$ as follows. Suppose $e=(u,v)$ and $P_u, P_v\in \mathcal{P}$ are the blocks of the partition containing $u$ and $v$ respectively ($u$ may equal $v$ and the blocks need not be distinct). Then $\mathcal{P}'$ is obtained from $\mathcal{P}$ by removing blocks of the partition $P_u$ and  $P_v$,  and replacing them with the block $(P_u\cup P_v) \backslash \{u,v\} \cup W$ where $W$ is the set of vertices created by the contraction (so $w$ consists of one or two vertices). 

There are two graphs naturally associated with   $(G,\mathcal{P})$. The first is the underlying graph of $G$. The second is obtained by identifying the vertices in the underlying graph of $G$ that belong to each block of $\mathcal{P}$. We denote this graph by $G_{/\mathcal{P}}$. 
As an example, Figure~\ref{f.vprg} shows a  vertex partitioned ribbon graph $(G,\mathcal{P})$, a minor of it, and  $G_{/\mathcal{P}}$.

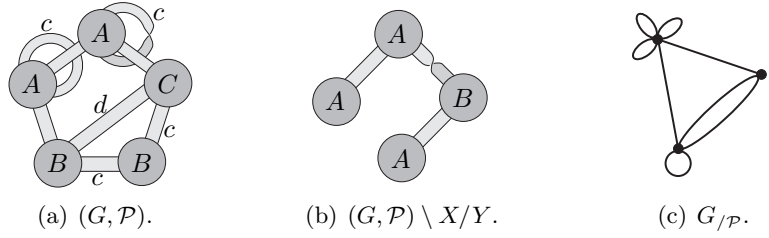
\begin{figure}
\centering
\subfigure[$(G,\mathcal{P})$. ]{
\labellist \small\hair 2pt
\pinlabel $A$  at 37 159
\pinlabel $A$  at 142 237
\pinlabel $B$ at 77 40
\pinlabel $B$   at 204 40
\pinlabel $C$   at 244 159
\pinlabel $c$   at 60 248
\pinlabel $c$   at 230 265
\pinlabel $c$   at 246 88
\pinlabel $c$   at 137 14
\pinlabel $d$   at 142 129
\endlabellist
\includegraphics[height=25mm]{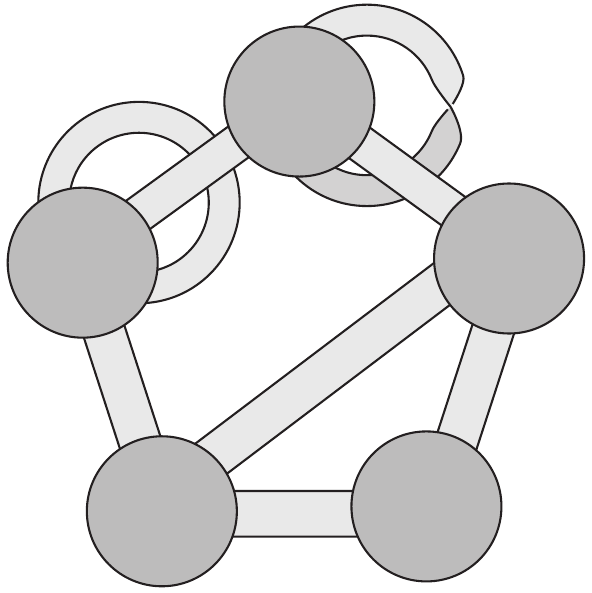}
\label{f.vprg.1}
}
\hspace{1cm}
\subfigure[$(G,\mathcal{P}) \ba X /Y $.]{
\labellist \small\hair 2pt
\pinlabel $A$  at 41 134
\pinlabel $A$  at 134 235
\pinlabel $A$ at 140 46
\pinlabel $B$   at 234 138
\endlabellist
\includegraphics[height=25mm]{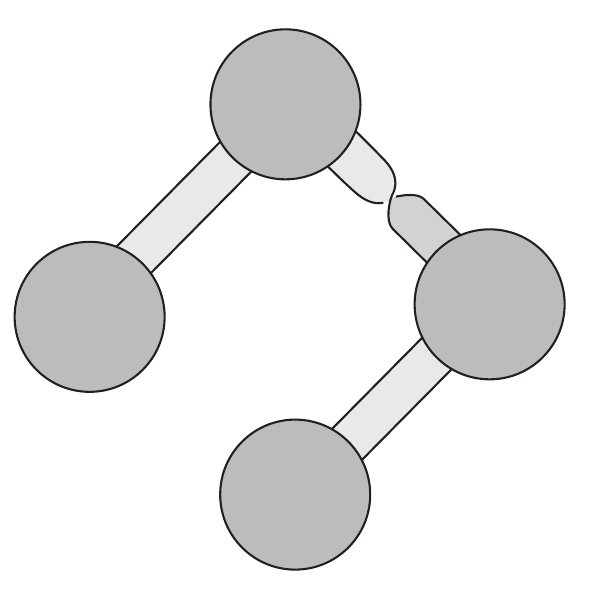}
\label{f.vprg.2}
}
\hspace{1cm}
\subfigure[$G_{/\mathcal{P}}$.]{
\includegraphics[height=25mm]{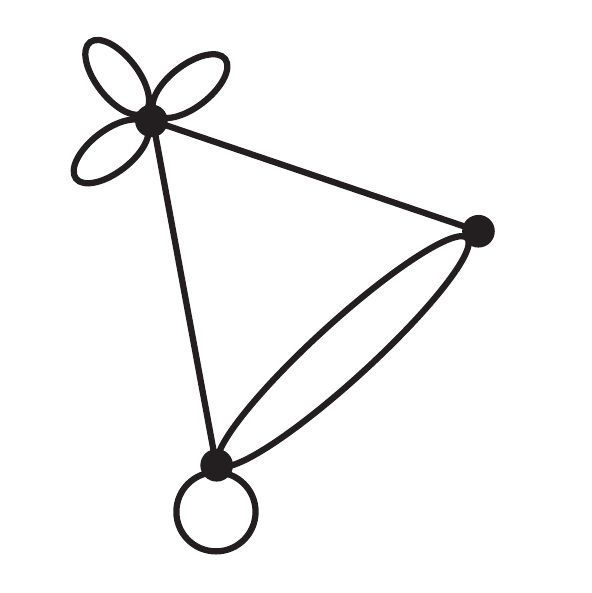}
\label{f.vprg.3}
}

\caption{A vertex partitioned ribbon graph and associated constructions. The letters $A$, $B$ and $C$ indicate the blocks of the partition. $X$ consists of the edge marked $d$, and $Y$ consists of the edges marked $c$.}
\label{f.vprg}
\end{figure}

We say that 
  $(G,\mathcal{P})$ is a the \emph{join} of $(G_1,\mathcal{P}_1)$ and $(G_2,\mathcal{P}_2)$, written $(G_1,\mathcal{P}_1)\vee (G_2,\mathcal{P}_2)$, if $G=G_1\vee G_2$ and, for $i=1,2$, $\mathcal{P}_i$ is the restriction of $\mathcal{P}$ to elements in $V(G_i)$. 
We state the following Lemma without proof.
\begin{lemma}\label{l.vprgms}
The set of equivalence classes of vertex partitioned ribbon graphs considered up to joins and isomorphism  forms a minor system where the grading is given by the cardinality of the edge set, deletion and contraction are given as above, and multiplication is given by  disjoint union. 
\end{lemma} 
 We will now identify a vertex partitioned ribbon graph with its equivalence class. 
\begin{definition}\label{d.hopfvrg}
We let  $\mathcal{H}^{vrg}$  denote the  Hopf algebra associated with vertex partitioned ribbon graphs via  Proposition~\ref{p.1}. 
Its coproduct is given by 
$
 \Delta_{vrg}(G,\mathcal{P})=\sum_{A\subseteq E(G)}(G,\mathcal{P})\backslash A^c  \otimes  (G,\mathcal{P})/ A$.
\end{definition}

While $\mathcal{H}^{rg} $ has three elements of graded dimension 1, $\mathcal{H}^{vrg} $ has four, giving rise to the following maps. 
   \begin{equation}\label{e.prgic}
   \begin{aligned}
\delta_{\mathrm{b}} (G,\mathcal{P}) =\begin{cases} 1 &\mbox{if } 
 (G,\mathcal{P})  =  (\raisebox{-1.6mm}{\includegraphics[scale=0.15]{d3a}}, \{ \{u\},\{v\} \}), \\
0 & \mbox{otherwise} ;\end{cases} 
&\quad 
\delta_{\mathrm{o}} (G,\mathcal{P}) =\begin{cases} 1 &\mbox{if } 
 (G,\mathcal{P}) = (\raisebox{4.5mm}{\rotatebox{-90}{\includegraphics[scale=0.2]{d5}}}  , \{ \{v\} \}),
\\
0 & \mbox{otherwise}; \end{cases} 
\\
\delta_{\mathrm{n}} (G,\mathcal{P}) =\begin{cases} 1 &\mbox{if } 
 (G,\mathcal{P})  =(\raisebox{4.5mm}{\rotatebox{-90}{\includegraphics[scale=0.2]{d2}}} , \{ \{v\} \}),
\\
0 & \mbox{otherwise};\end{cases}
& \quad 
\delta_{\mathrm{l}} (G,\mathcal{P}) =\begin{cases} 1 &\mbox{if } 
 (G,\mathcal{P}) =  (\raisebox{-1.6mm}{\includegraphics[scale=0.15]{d3a}}, \{ \{u,v\} \}) ,\\
0 & \mbox{otherwise} ;\end{cases} 
  \end{aligned}
  \end{equation}
  where the drawings above represent the equivalence classes of the  ribbon graphs, and $u$ and $v$ are the vertices in the relevant figures.
  
We set
\begin{equation}\label{e.vbr1}
 \delta_{\mathbf{a}}= \delta(a_1,a_2,a_3,a_4) := a_1\delta_{b} + a_2\delta_{o} +a_3\delta_{n}+a_4\delta_{l}.  
  \end{equation}
As it is not when restricted to ribbon graphs,  $ \delta_{\mathbf{a}}$ is not uniform unless $a_3=\sqrt{a_2a_4}$. The following theorem will shows that    $ \delta_{\mathbf{a}}$ is  uniform if and only if $a_3=\sqrt{a_2a_4}$. 

\begin{theorem}\label{t.vbr1}
The canonical Tutte polynomial of the  Hopf algebra $\mathcal{H}^{vrg}$ is given by
 \begin{equation}\label{e.t.vbr1}
 \alpha( \mathbf{a}, \mathbf{b}) ( G,\mathcal{P})  =      y_1^{r(G_{/\mathcal{P}})}    y_2^{|E|-\rho(G)}    y_3^{\rho(G)-r(G_{/\mathcal{P}})}   
  \sum_{A\subseteq E}   \left(\frac{x_1}{y_1}\right)^{r(A_{/\mathcal{P}})}  \left(\frac{x_2}{y_2}\right)^{|A|-\rho(A)} \left(\frac{x_3}{y_3}\right)^{\rho(A)-r(A_{/\mathcal{P}})},
 \end{equation}
where $\mathbf{a} = ( x_1,x_2,\sqrt{x_2x_3},x_3)$, $\mathbf{b}  =(y_1,y_2,\sqrt{y_2y_3},y_3)$, and $E=E(G)$. 
\end{theorem}
\begin{proof}
Set 
$r_1(G,\mathcal{P}):=r(G_{/\mathcal{P}})$,   
$r_2(G,\mathcal{P}):=|E|-\rho(G)$, 
$r_3(G,\mathcal{P}):=\rho(G)-r(G_{/\mathcal{P}})$, 
where $\rho(G)$ is as in Equation~\eqref{e.rgrho}, and $E=E(G)$.

Recalling that $e$ is a loop in a graph if and only if it is a loop in its cycle matroid we have
\begin{equation}\label{e.t.vbr2}
 r_1(G,\mathcal{P}) = \begin{cases}  r_1((G,\mathcal{P})/e)&  \text{if $e$ is  a loop in $G_{/\mathcal{P}}$},    \\   r_1((G,\mathcal{P})/e) +1 &  \text{otherwise}.\end{cases} 
   \end{equation}
Similarly, by Lemma~\ref{s2.l7}, and using a result from \cite{CMNR1} that $e$ is  (not a / an orientable /  a nonorientable) loop in $G$ if and only if $e$ is a (not a / an orientable /  a nonorientable) ribbon loop in $D(G)$, we have  
\begin{equation}\label{e.t.vbr3}
 r_2(G,\mathcal{P}) = \begin{cases}  r_2((G,\mathcal{P})/e)  &  \text{if $e$ is not a loop in $G$},    \\   r_2((G,\mathcal{P})/e) +1 &  \text{if $e$ is an orientable loop in $G$}, \\   r_2((G,\mathcal{P})/e)  +\frac{1}{2} &  \text{if $e$ is a non-orientable loop in $G$} .\end{cases}  
   \end{equation}

From the cases for $r_1$ and $r_2$ above we can deduce that 
\[ r_3(G,\mathcal{P}) = \begin{cases}  r_3((G,\mathcal{P})/e)  &  \text{if $e$ is not a loop in $G$ or $G_{/\mathcal{P}}$},    \\  
 r_3((G,\mathcal{P})/e)  &  \text{if $e$ is  an orientable loop in $G$ and a loop in $G_{/\mathcal{P}}$},    \\
  r_3((G,\mathcal{P})/e)  &  \text{if $e$ is  a non-orientable loop in $G$ and a loop in $G_{/\mathcal{P}}$},    \\
   r_3((G,\mathcal{P})/e)  &  \text{if $e$ is not a loop in $G$, but is a loop in $G_{/\mathcal{P}}$}.
 \end{cases}   \]

Then if $\delta_1:=\delta_c$, $\delta_2:=\delta_o$, $\delta_3:=\delta_n$, and $\delta_4:=\delta_l$  in Theorem~\ref{tm} we have
$m_{11}=m_{22}=m_{43}=1$, $m_{32}=m_{33}=\tfrac{1}{2}$, and all other $m_{ij}$ are zero. The theorem then gives
 \[
 \alpha( \mathbf{a},\mathbf{b}) ( G,\mathcal{P})  =    y_1^{r(G_{/\mathcal{P}})}    y_2^{|E|-\rho(G)}    y_3^{\rho(G)-r(G_{/\mathcal{P}})}   
  \sum_{A\subseteq E}   \left(\frac{x_1}{y_1}\right)^{r(A_{/\mathcal{P}})}  \left(\frac{x_2}{y_2}\right)^{|A|-\rho(A)} \left(\frac{x_3}{y_3}\right)^{\rho(A)-r(A_{/\mathcal{P}})},
\]
as required.
\end{proof}

We can recognise the Bollob\'as-Riordan polynomial in Equation~\eqref{e.t.vbr1}:
\begin{corollary}\label{c.vbr1}
Let $G=(V,E)$ be a ribbon graph, $\hat{\mathcal{P}}=\{\{v\} \mid v\in V \}$, and $\alpha$ be as in Theorem~\ref{t.vbr1}. Then 
\[
 \alpha( \mathbf{a}, \mathbf{b}) ( G,\hat{\mathcal{P}}) =  x_1^{r(G)}   y_2^{|E|-r(G)}    (y_3/y_2)^{\frac{1}{2}\gamma(G)}   R_{G}\left(   \frac{y_1}{x_1}+1, \frac{x_2}{y_2},  \sqrt{  \frac{x_3y_2}{x_2y_3}}  \right) .
\] 
In particular, if $\mathbf{a} = ( 1, y,yz,yz^2)$, $\mathbf{b}  =(x-1,1,1,1)$, then 
\[ \alpha( \mathbf{a}, \mathbf{b}) ( G,\hat{\mathcal{P}}) =  R_G(x,y,z). \]
\end{corollary}
\begin{proof}
The result follows from Theorem~\ref{t.vbr1} upon noting that here $r(A)=r(A_{/\mathcal{P}})$  and, via Euler's Formula, $\rho(A)=r(A)+\tfrac{1}{2}\gamma(A)$.
\end{proof}

In light of Corollary~\ref{c.vbr1} it is natural to make the following definition.
\begin{definition}\label{def.vbr}
The \emph{Bollob\'as-Riordan polynomial}, $R_{ ( G,\mathcal{P})}(x,y,z)$, of a vertex partitioned ribbon graph  $( G,\mathcal{P})$ is defined by 
\begin{align*} R_{ ( G,\mathcal{P})}(x,y,z) &:=    \alpha( \mathbf{x}, \mathbf{y}) ( G,\mathcal{P})  
= \sum_{A\subseteq E(G)}  (x-1)^{r(G_{/\mathcal{P}}) -  r(A_{/\mathcal{P}})}  y^{|A|-r(A_{/\mathcal{P}})} z^{2( \rho(A)-r(A_{/\mathcal{P}}) )},
\end{align*}
where $  \alpha$ is as in Theorem~\ref{t.vbr1},   $\mathbf{x} = ( 1, y,yz,yz^2)$, and $\mathbf{y}  =(x-1,1,1,1)$.
\end{definition}

\begin{corollary}\label{c.vbr2} 
 The  projection $\phi: \mathcal{H}^{vrg} \rightarrow \mathcal{H}^{rg}$ defined by $ \phi (G, \mathcal{P})=G$ is a Hopf algebra morphism.   Furthermore it naturally induces the  identity \[\tilde{R}_{G}(x,y) =   (x-1)^{ \rho(G)-r(G/\mathcal{P}) }  R_{(G,\mathcal{P})}(x,y-1,1/\sqrt{( x-1)(y-1)})  .\]\end{corollary}
\begin{proof}
The mapping is easily seen to be a well-defined Hopf algebra morphism. 
In Theorem~\ref{tgen} let $\mathcal{H}=\mathcal{H}^{vrg}$,  $\mathcal{H'}=\mathcal{H}^{rg}$,
 $\delta_{\mathcal{H}}$ be the selector used in Theorem~\ref{t.vbr1}, and 
 $\delta_{\mathcal{H'}}$ be the selector used in Theorem~\ref{t.rgbr1}. 
 We have $  \delta_{\mathcal{H}}(s, x_1,x_2,x_3, x_1) (G, \mathcal{P}) =  \delta_{\mathcal{H'}}(s, x_1,x_2,x_3) ( G) $. By Theorem~\ref{tgen}, and Corollaries~\ref{t.rgbr1} and~\ref{t.rgbr1.c1},  $\tilde{R}_{G}(x,y) =  (x-1)^{ \rho(G)-r(G_{/\mathcal{P}}) }  R_{(G,\mathcal{P})}(x,y-1,1/\sqrt{( x-1)(y-1)}) $.
\end{proof}

\subsection{Krushkal's polynomial}\label{xpdg} 
 The Krushkal polynomial \cite{But,Kr}  is a 4-variable extension of the Tutte polynomial to graphs  embedded (but not necessarily cellularly embedded) in surfaces.  
For a graph $G=(V,E)$ embedded in a surface $\Sigma$, denoted $G\subset \Sigma$, the {\em Krushkal polynomial} is defined by
\[  K_{G\subset \Sigma} ( x,y,a,b ) :=  \sum_{A\subseteq E(G)}  x^{r(G)-r(A)} y^{ \kappa(A)} a^{ \frac{1}{2}s(A)} b^{ \frac{1}{2} s^{\perp}(A)}   ,   \]
where $s(A):= \gamma(N(V \cup A))$ is the Euler genus of a regular neighbourhood  $N(V \cup A)$ of the spanning subgraph $(V,A)$ of $G$ (note $N(V \cup A)$ can be considered as a ribbon graph);  $s^{\perp}(A) := \gamma(\Sigma \backslash N(V \cup A))$; and, as in Equation~\eqref{LVeq3},
\begin{equation}\label{e.vgs1}
\kappa(A):= \#\mathrm{cpts}(\Sigma  \backslash N(V \cup  A))  - \#\mathrm{cpts}(\Sigma  \backslash N(V)).
\end{equation} 
Observe that here $\#\mathrm{cpts}(\Sigma  \backslash N(V))=\#\mathrm{cpts}(\Sigma)$ since we are  considering graphs in surfaces, rather than graphs in pseudo-surfaces.  
Note that we are following \cite{But} and using the form of the exponent of $y$ from the proof of Lemma~4.1 of \cite{ACEMS}  rather than the homological definition given in \cite{Kr}.

 The Krushkal polynomial absorbs both the Bollob\'as-Riordan and Las~Vergnas polynomials. From  \cite{But,Kr}
  \[   R_G(x,y,z) = y^{\frac{1}{2}\gamma(G)}   K_{G\subset \Sigma}  (x-1,y,yz^2,y^{-1}),   \]
  where $R_G$ is computed by considering the ribbon graph arising from a neighbourhood of $G$ in $\Sigma$. 
From \cite{ACEMS,But, EMMlv} 
\begin{equation}\label{e.lvkr} 
 L_{G \subset \Sigma}(x,y,z) =  z^{\frac{1}{2}(s(E) -s^{\perp} (E))}   K_{G\subset \Sigma}  (x-1,y-1,z^{-1},z).
  \end{equation}

Similarly to what we saw with the Bollob\'as-Riordan polynomial, the Krushkal polynomial is not the canonical Tutte polynomial arising from graphs in surfaces and their minors, but rather with vertex partitioned graphs in surfaces.
A \emph{vertex partitioned graph in a surface}, $(G\subset \Sigma ,\mathcal{P})$ consists of a  graph $G=(V,E)$  embedded in a surface $\Sigma$ (although not necessarily cellularly embedded) and a partition  $\mathcal{P}$ of its vertex set $V$. 

Considering only  graphs in surfaces for the moment (i.e., forgetting about the partition), if $e\in E$ then $(G\subset \Sigma )\backslash  e$ is the  graph in a surface obtained by removing the edge $e$ from the drawing of $G\subset \Sigma$ (without removing the  points of $e$ from $\Sigma$, or its incident vertices).    Edge contraction $(G\subset \Sigma )/  e$ is defined by forming a quotient space of the surface $G/e \subset \Sigma'$ then resolving any pinch points by ``splitting'' them into new vertices as in Figure~\ref{cont.a} and~\ref{cont.b}  (formally, delete a small neighbourhood of $e$ in $\Sigma$ and contract any resulting boundary components to points which become vertices).  

Deletion and contraction for  vertex partitioned graphs in surfaces is defined in the natural way, and analogously to the ribbon graph case.  If $e \in E(G)$, then \emph{deletion} is defined by $(G\subset \Sigma ,\mathcal{P}) \backslash e:=  ((G\subset \Sigma )\backslash e,\mathcal{P})$. \emph{Contraction} is defined by $ (G\subset \Sigma ,\mathcal{P}) / e:=  ((G\subset \Sigma )/ e,\mathcal{P}')$, where the partition $\mathcal{P}'$ is induced by $\mathcal{P}$ as follows. Suppose $e=(u,v)$ and $P_u, P_v\in \mathcal{P}$ are the blocks containing $u$ and $v$ respectively ($u$ may equal $v$ and the blocks need not be distinct). Then $\mathcal{P}'$ is obtained from $\mathcal{P}$ by removing blocks $P_u$ and  $P_v$,  and replacing them with the block $(P_u\cup P_v) \backslash \{u,v\} \cup W$ where $W$ is the set of vertices created by the contraction (so $w$ consists of one or two vertices). 

There are three graphs naturally associated with   $(G\subset \Sigma ,\mathcal{P})$. The  underlying graph in a surface $G\subset \Sigma$, the underlying abstract graph $G$, and the abstract graph  $G_{/\mathcal{P}}$  obtained by identifying the vertices in $G$ that belong to each block of $\mathcal{P}$.

It is convenient at this point to extend Equation~\eqref{e.rgrho} to the present setting, defining
\begin{equation}\label{e.vgs2}
\rho(A) := \tfrac{1}{2}\left(|A|+ v(A)-b(A)\right),
\end{equation}
where $b(A)$ denotes the number of boundary components of $ N(V\cup A)$. We set $   \rho( G\subset\Sigma,\mathcal{P}) := \rho(E)$.

We will show that the Krushkal polynomial arises as the canonical Tutte polynomial associated with vertex partitioned graphs in surfaces. 
\begin{definition}\label{d.vgs2}
We will say that two  vertex partitioned graphs in surfaces  $(G_1\subset \Sigma_1, \mathcal{P}_1)$ and  $(G_k\subset \Sigma_k,\mathcal{P}_k)$ are  \emph{Kr-equivalent} if there is a sequence of   vertex partitioned graphs in surfaces  $(G_1\subset \Sigma_1, \mathcal{P}_1), (G_2\subset \Sigma_1, \mathcal{P}_2), \ldots, (G_k\subset \Sigma_k,\mathcal{P}_k)$ such that $(G_i\subset \Sigma_i,\mathcal{P}_i)$ is obtained from  $(G_{i-1}\subset \Sigma_{i-1},\mathcal{P}_{i-1})$, or vice versa, by one of the following moves. 
\begin{enumerate}
\item Deleting a component of the surface that contains no edges of the graph.
\item Deleting an isolated vertex from a graph.
\item Connect summing two surface components (away from any graph components).
\item Replacing a region, with another surface with boundary (so that it forms a region of a new graph in a surface).
\end{enumerate}
\end{definition}
It is clear that Kr-equivalence gives rise to an equivalence relation, and we let $\mathcal{G}^{vgs}$ denote the set of all equivalence classes of graphs in pseudo-surfaces considered up to Kr-equivalence.  We grade $\mathcal{G}^{vgs}$ by the number of edges in any graph that represents its class.

The following lemma is easily verified.
\begin{lemma}\label{l.vgsms}
$\mathcal{G}^{vgs}$ forms a minor system where the grading is given by the cardinality of the edge set, deletion and contraction are given as above, and multiplication is given by disjoint union. 
\end{lemma}

\begin{definition}\label{d.vgs1}
We let  $\mathcal{H}^{vgs}$  denote the  Hopf algebra associated with $\mathcal{G}^{vgs}$ via Lemma~\ref{l.vgsms} and Proposition~\ref{p.1}. 
Its coproduct is 
$
 \Delta_{vgs}( G \subset \Sigma, \mathcal{P})=\sum_{A\subseteq E(G)} (( G \subset \Sigma, \mathcal{P})\backslash  A^c)  \otimes  (( G \subset \Sigma, \mathcal{P})/A )$.
\end{definition}

\begin{lemma}\label{l.vgs1}
$\mathcal{G}^{vgs}_1$  has a basis consisting of exactly five elements represented by 
\begin{enumerate}
\item a 1-path in the sphere with each vertex appearing in its own block of the partition,
\item a loop in the sphere,
\item a 1-path in the sphere with both vertices appearing in the same block of the partition,
\item a loop cellularly embedded in the real projective plane,
\item a loop that forms the meridian of a torus.
\end{enumerate}
\end{lemma}
\begin{proof}
Let $(G\subset \Sigma, \mathcal{P})$ be a vertex partitioned graph in a surface with exactly one edge. If that edge is a bridge then  delete any isolated vertices,  remove any empty surface components, then replace the remaining region with a disc. What remains is a 1-path in the sphere with each vertex appearing in its own block of the partition or  a 1-path in the sphere with both vertices appearing in the same block of the partition.  

Otherwise the single edge in $(G\subset \Sigma, \mathcal{P})$ is a loop. Again delete any isolated vertices then  remove any empty surface components. The resulting  vertex partitioned graph in a surface has one or two regions. If it has two regions replace each with a disc to obtain a loop in the sphere. If it has one region then either a neighbourhood $N(e)$ of the edge is an annulus or M\"obius band. If it is a  M\"{o}bius band replace the unique region with a disc to get a loop cellularly embedded in the real projective plane.  If it is an  annulus  replace the unique region with an annulus to get a loop that forms the meridian of a torus.
\end{proof}

We now describe the selector of $\mathcal{H}^{vgs}$. For this set
\begin{equation}\label{dhgads}
\begin{aligned}
  \delta_{1}(G\subset \Sigma , \mathcal{P} ) &:= \begin{cases} 1 &\mbox{if $(G\subset \Sigma , \mathcal{P} )  $ is a 1-path in the sphere, with a partition of two blocks,} \\   0 & \mbox{otherwise;} \end{cases} 
  \\
   \delta_{2}(G\subset \Sigma , \mathcal{P} )  &:= \begin{cases} 1 &\mbox{if $(G\subset \Sigma , \mathcal{P} )   $ is a loop in the sphere,} \\   0 & \mbox{otherwise;} \end{cases} 
    \\ 
  \delta_{3}(G\subset \Sigma , \mathcal{P} ) &:= \begin{cases} 1 &\mbox{if $(G\subset \Sigma , \mathcal{P} )  $ is a 1-path in the sphere with a partition of one block,} \\   0 & \mbox{otherwise;} \end{cases} 
   \\ 
  \delta_{4}(G\subset \Sigma , \mathcal{P} )  &:= \begin{cases} 1 &\mbox{if $(G\subset \Sigma , \mathcal{P} )   $ is a  loop cellularly embedded in the real projective plane,} \\   0 & \mbox{otherwise;} \end{cases} 
    \\
   \delta_{5}(G\subset \Sigma , \mathcal{P} ) &:= \begin{cases} 1 &\mbox{if $(G\subset \Sigma , \mathcal{P} )  $ is a loop that is a meridian of a torus,}  \\   0 & \mbox{otherwise.} \end{cases} 
\end{aligned}
\end{equation}
Set
\[
 \delta_{\mathbf{a}}= \delta(a_1,a_2,a_3,a_4,a_5) := a_1\delta_{1} + a_2\delta_{2} +a_3\delta_{3}+a_4\delta_{4}+a_5\delta_{5}.  
\]
It can be seen that  $ \delta_{\mathbf{a}}$ is not uniform unless $a_4=\sqrt{a_3a_5}$, in which case the following theorem says that it is.

\begin{theorem}\label{t.vkr1}
The canonical Tutte polynomial of the  Hopf algebra $\mathcal{H}^{vgs}$ is given by
\begin{multline*}
 \alpha( \mathbf{a},\mathbf{b}) ( G\subset\Sigma,\mathcal{P})  =    
   y_1^{r(G_{/\mathcal{P}})} 
    y_2^{\kappa(G\subset\Sigma,\mathcal{P})}  
         y_3^{\rho(G\subset\Sigma,\mathcal{P})-r(G_{/\mathcal{P}})}   
    y_4^{  |E| - \rho(G\subset\Sigma,\mathcal{P})-\kappa(G\subset\Sigma,\mathcal{P})}   
\\
\\\sum_{A\subseteq E}  
 \left(\frac{x_1}{y_1}\right)^{r(A_{/\mathcal{P}})} 
  \left(\frac{x_2}{y_2}\right)^{\kappa(A)} 
        \left(\frac{x_3}{y_3}\right)^{\rho(A)-r(A_{/\mathcal{P}})}     
   \left(\frac{x_4}{y_4}\right)^{  |A| - \rho(A)-\kappa(A)}   
      ,\end{multline*}
where $\mathbf{a} = ( x_1,x_2, x_3 , \sqrt{x_3x_4},x_4 )$, $\mathbf{b}  =(y_1,y_2,y_3,  \sqrt{y_3y_4},y_4 )$, and $E=E(G)$. 
\end{theorem}
\begin{proof}
Set 
$r_1(G\subset\Sigma,\mathcal{P}):=r(G_{/\mathcal{P}})$,   
$r_2(G\subset\Sigma,\mathcal{P}):=\kappa(G\subset\Sigma,\mathcal{P})$, 
$r_3(G\subset\Sigma,\mathcal{P}):=  \rho(G\subset\Sigma,\mathcal{P})-r(G_{/\mathcal{P}})$, 
$r_4(G\subset\Sigma,\mathcal{P}):=  |E| - \rho(G\subset\Sigma,\mathcal{P})-\kappa(G\subset\Sigma,\mathcal{P})$.
Equation~\eqref{e.t.vbr1} relates $r(G_{/\mathcal{P}})$ and $r((G_{/\mathcal{P}})/e)$. 
By observing that  $ N(V\cup A)$ gives rise to a ribbon graph, we see $\rho(A)$ from Equation~\eqref{e.vgs2} corresponds to $\rho(A)$ from Equation~\eqref{e.rgrho}, and then \eqref{e.t.vbr2} gives the relation between $\rho(G_{/\mathcal{P}})$ and $\rho((G_{/\mathcal{P}})/e)$. 
Recall from Section~\ref{ssec.lv} that  $e\in E(G)$ is a {\em quasi-loop} if  $\kappa(e)=1$. It is then not hard to see that 
\[
 \kappa(G,\mathcal{P}) = \begin{cases}  \kappa((G,\mathcal{P})/e)  +1&  \text{if $e$ is  a quasi-loop},    \\   \kappa((G,\mathcal{P})/e)  &  \text{otherwise} \end{cases} 
 .\]
From these we can deduce the relations between $r_i(G,\mathcal{P})$ and $r_i((G,\mathcal{P})/e)$, for each $i$. Applying this to Theorem~\ref{tm} gives  $m_{11}=m_{22}=m_{33}=m_{54}=1$, $m_{43}=m_{44}=\tfrac{1}{2}$, and all other $m_{ij}$ are zero. The result follows.\end{proof}

\begin{corollary}\label{c.vkr1} 
Let $G\subset \Sigma$ be a graph in a surface, $\hat{\mathcal{P}}=\{\{v\} \mid v\in V \}$, and $\alpha$ be as in Theorem~\ref{t.vkr1}. Then  
\[ \alpha( \mathbf{a},\mathbf{b}) ( G\subset\Sigma,\hat{\mathcal{P}})  =        
x_1^{r(G)} y_2^{\kappa(G)} y_3^{\frac{1}{2}s(G)} y_4^{-\frac{1}{2}s^{\perp}(G)}x_4^{\frac{1}{2}\gamma(\Sigma)} K_{G\subset\Sigma} \left( \frac{y_1}{x_1},  \frac{x_2}{y_2}, \frac{x_3}{y_3}, \frac{y_4}{x_4}\right).\]
In particular, 
when $\mathbf{a} = ( 1,y, a  ,\sqrt{a},1 )$, $\mathbf{b}  =(x,1,1 ,\sqrt{b},b)$,
\[ \alpha( \mathbf{x},\mathbf{y}) ( G\subset\Sigma,\hat{\mathcal{P}})  =        
    b^{ - \frac{1}{2} s^{\perp}(G)}  K_{G\subset \Sigma} (x, y,a,b ), \]
 and   when $\mathbf{a} = ( 1,y, a  ,\sqrt{ab},b )$, $\mathbf{b}  =(x,1,1 ,1,1)$,
\[ \alpha( \mathbf{x},\mathbf{y}) ( G\subset\Sigma,\hat{\mathcal{P}})  =        
    b^{  \frac{1}{2} \gamma(\Sigma)}  K_{G\subset \Sigma} (x, y,a, 1/b ) .\]
\end{corollary}
\begin{proof}
First observe that $G=G_{/\mathcal{P}}$, and so $r(G)=r(G_{/\mathcal{P}})$. Then, by Euler's formula,  
$\rho(A)-r(A_{/\mathcal{P}})   = \rho(A)-r(A) =  \tfrac{1}{2}\left(|A|+ v(A)-b(A))\right) -v(A)+c(A) =    \tfrac{1}{2}\gamma(A) = \tfrac{1}{2}s(A)$. 

For obtaining a $\tfrac{1}{2}s^{\perp}(A)$ exponent,  we can write  $\Sigma = N(V)\cup N(E) \cup N(R) $ where  $V$ is the vertex set of $G$, $E$ its edge set and $N(R)$ is the complement of $N(V)\cup N(E)$, and where the neighbourhoods only intersect on their boundaries. Choose a triangulation of $\Sigma$ that restricts to a triangulation of each of  $N(V)$,  $N(E)$, and   $N(R)$. To compute the  Euler characteristic $\chi(\Sigma \backslash N(V \cup A))$, for $A\subseteq E$,  with this triangulation observe that $\Sigma \backslash N(V \cup A) =  [\Sigma \backslash N(V) ] \cup N(A)$ and that each time we add the neighbourhood of an edge to $\Sigma \backslash N(V)$, the Euler characteristic drops by 1. It follows that $\chi(\Sigma \backslash N(V \cup A)) =  \chi(\Sigma \backslash N(V))  - |A| $. 
Using that $\gamma (\Sigma \backslash N(V \cup A) ) = 2\,\#\text{cpts}(\Sigma \backslash N(V \cup A)) -\chi(\Sigma \backslash N(V \cup A)) - b(\Sigma \backslash N(V \cup A))$ we have
\begin{align*}
\gamma(\Sigma) - s^{\perp}(A) &= \gamma(\Sigma \backslash N(V)) -  \gamma(\Sigma \backslash N(V \cup A)) 
  \\&= 2\,\#\text{cpts}(\Sigma \backslash N(V)) -\chi(\Sigma \backslash N(V)) - b(\Sigma \backslash N(V))
 \\&\quad \quad\quad   -2\,\#\text{cpts}(\Sigma \backslash N(V \cup A)) +\chi(\Sigma \backslash N(V \cup A)) + b(\Sigma \backslash N(V \cup A)) 
 \\&=  [  \chi(\Sigma \backslash N(V))  - |A| - \chi(\Sigma \backslash N(V))  ]  + 2[\#\text{cpts}(\Sigma \backslash N(V \cup A)) -  \#\text{cpts}(\Sigma \backslash N(V)) ]
  \\&\quad \quad\quad  +  [  b(\Sigma \backslash N(V \cup A)) - |V|  ]
 \\&=        2 [  |A|- \rho(A)  - \kappa(A) ].
\end{align*}
The results then follows from  Theorem~\ref{t.vkr1} and the definition of $ K_{G\subset \Sigma}$.
\end{proof}

\begin{corollary}\label{c.vkr2}
~
\begin{enumerate}
\item \label{c.vkr2.a} The  natural mapping $\phi_1: \mathcal{H}^{vgs} \rightarrow \mathcal{H}^{vrg}$ defined by  sending $(G\subset \Sigma, \mathcal{P})$ to $(N(G), \mathcal{P})$  is a Hopf algebra morphism.   Furthermore it naturally induces identity \[ R_G(x,y,z) = y^{\frac{1}{2}\gamma(\Sigma)}   K_{G\subset \Sigma}  (x-1,y,yz^2,y^{-1}).\]

\item \label{c.vkr2.b} The  projection $\phi_2: \mathcal{H}^{vgs} \rightarrow \mathcal{H}^{ps}$ defined by identifying all of the vertices in each block of the partition of $(G\subset \Sigma, \mathcal{P})$ to a pinch point is a Hopf algebra morphism.   Furthermore it naturally induces the polynomial identity \[ L_{\phi_2 (G\subset \Sigma, \mathcal{P})}(x,y,z) =  z^{\frac{1}{2}(s(E) -s^{\perp} (E))}   K_{G\subset \Sigma}  (x-1,y-1,z^{-1},z).\]
\end{enumerate}
\end{corollary}
\begin{proof}
It is readily verified that the  maps are Hopf algebra morphisms. 

For $\phi_1$, in Theorem~\ref{tgen} let $\mathcal{H}=\mathcal{H}^{vgs}$,  $\mathcal{H'}=\mathcal{H}^{vrg}$,
 $\delta_{\mathcal{H}}$ be the selector used in Theorem~\ref{t.vbr1}, and 
 $\delta_{\mathcal{H'}}$ be the selector used in Theorem~\ref{t.rgbr1}. 
 It is easily seen  that 
 \[    \delta_{\mathcal{H}}(a_1,a_2,a_4,a_3,a_2) (G\subset \Sigma, \mathcal{P}) =  \delta_{\mathcal{H'}}( a_1,a_2,a_3,a_4) (\phi( G, \mathcal{P}))  .\]
  Theorem~\ref{tgen} gives 
  \[ \alpha_{\mathcal{H}}( 1,y,yz^2,y,yz , x-1,1,1 , 1,1) ( G\subset\Sigma,\mathcal{P})  = 
     \alpha_{\mathcal{H'}}( 1, y  ,yz ,yz^2 ,x-1,1  ,1,1)(\phi_1( G, \mathcal{P})). \]
  Corollaries~\ref{c.vbr1} and \ref{c.vkr1} then give  
  $   y^{  \frac{1}{2} \gamma(\Sigma)}  K_{G\subset \Sigma} (x-1, y,y,yz^2)   =   R_G(x,y,z)  $,   
but since $G$ is cellularly embedded,  $\gamma(\Sigma)=\gamma(G)$.

For $\phi_2$, in Theorem~\ref{tgen} let $\mathcal{H}=\mathcal{H}^{vgs}$,  $\mathcal{H'}=\mathcal{H}^{gs}$,
 $\delta_{\mathcal{H}}$ be the selector used in Theorem~\ref{t.vbr1}, and 
 $\delta_{\mathcal{H'}}$ be the selector used in Theorem~\ref{s2.c1}. 
 We have $  \delta_{\mathcal{H}}( a_1,a_2,a_3,a_3,a_3)  (G\subset \Sigma, \mathcal{P})  =  \delta_{\mathcal{H'}}( a_1,a_2,a_3)  (\phi_2 (G\subset \Sigma, \mathcal{P}) )$,
 so
 % \begin{multline*} 
 $  z^{\frac{1}{2}(s(E) -s^{\perp} (E))}  K_{G\subset \Sigma}  (x-1, y-1, z^{-1} ,z) =
  \alpha_{\mathcal{H}}(  1, y-1 ,1  ,1  ,1,  x-1, 1,z, z,z ) =
 \alpha_{\mathcal{H'}}(  1, y-1 , 1  , x-1, 1,z)= L_{\phi_2 (G\subset \Sigma, \mathcal{P}) }(x,y,z)$.
  % \end{multline*}
\end{proof}

In light of Corollary~\ref{c.vkr1} we make the following definition. 
\begin{definition}\label{def.vkr}
The \emph{Krushkal polynomial},  $\widetilde{K}_{(G\subset\Sigma,\mathcal{P})} (x,y,a,b)$, of a vertex partitioned graph  in a surface     $(G\subset\Sigma,\mathcal{P})$ is 
\begin{align*} \widetilde{K}_{(G\subset\Sigma,\mathcal{P})} (x,y,a,b) &:=    \alpha( \mathbf{a}, \mathbf{b}) ( G\subset\Sigma,\mathcal{P})  
= \sum_{A\subseteq E(G)} x^{r(G_{/\mathcal{P}})-r(A_{/\mathcal{P}})} 
  y^{\kappa(A)} 
       a^{\rho(A)-r(A_{/\mathcal{P}})}     
 b^{  |A| - \rho(A)-\kappa(A)}   ,\end{align*}
where $  \alpha$ is as in Theorem~\ref{t.vbr1},   $\mathbf{a} = ( 1, y,a,\sqrt{ab},b)$, and  $\mathbf{b}  =(x,1,1,1,1)$.
\end{definition}
We note that if $\mathcal{P}=\{\{v\} \mid v\in V \}$, then by Corollary~\ref{c.vkr1}   
\[ \widetilde{K}_{(G\subset\Sigma,\mathcal{P})} (x,y,a,b) = b^{  \frac{1}{2} \gamma(G)}  K_{G\subset \Sigma} (x, y,a, 1/b ).\]  
Following the proof of Item~\eqref{c.vkr2.a} of Corollary~\ref{c.vkr2} gives 
\[ R_{(G,\mathcal{P})}(x,y,z) = \widetilde{K}_{(G\subset\Sigma,\mathcal{P})}(x-1,y,yz^2,y).   \]
While following the proof of Item~\eqref{c.vkr2.b} of Corollary~\ref{c.vkr2} gives 
\[ L_{\phi_2 (G\subset \Sigma, \mathcal{P})}(x,y,z) = z^{|E(G)|-r(G_{/\mathcal{P}})- \kappa(G\subset\Sigma,\mathcal{P})} \widetilde{K}_{(G\subset\Sigma,\mathcal{P})} (x-1,y-1,z^{-1},z^{-1}),\]
where $\phi_2$ is the mapping from the corollary.

\begin{remark}\label{adhj}
In Section~\ref{sec1} we described three problems in the area topological Tutte polynomials. The first problem was why
  three topological Tutte polynomials (the Las~Vergnas, Bollob\'as-Riordan, and Krushkal polynomials) had naturally arisen in the literature, and if any one of these can claim to be \emph{the} Tutte polynomial of an embedded graph. This is answered by the Hopf algebraic framework of canonical Tutte polynomials. Each of these three topological Tutte polynomials is a canonical Tutte polynomial, but each is a canonical Tutte polynomial of a slightly different class of objects with different concepts of deletion and contraction. It is  worth emphasising here that in order to obtain the Hopf algebraic framework for these topological Tutte polynomials, we  had to enlarge the domain of the polynomials. (For example, the Bollob\'as-Riordan polynomial is properly a polynomial of vertex partitioned ribbon graphs, rather than ribbon graphs.) In each case the domain can be found by starting with a cellularly embedded graph, a notion of deletion and contraction and looking for the class closed under these operations.

The second problem was why the three existing topological Tutte polynomials did not have full deletion-contraction definitions terminating in trivial objects. The answer is that the polynomials have previously been considered on what the canonical picture considers  the wrong domains. Upon extending the domains of the polynomials as guided by the Hopf algebras, the resulting canonical Tutte polynomials \emph{do} have full deletion-contraction definitions (by Theorem~\ref{t.1}).

The final problem is about the Bollob\'as-Riordan polynomial $R_G(x,y,z)$. Most of the known results about this polynomial, particularly its combinatorial interpretations,  do not apply to the full 3-variable polynomial $R_G(x,y,z)$, but rather to its 2-variable specialisation $x^{\gamma(G)/2}   R_G(x+1,y,1/\sqrt{xy})$  (see, for example,  
\cite{BR02,CP,Da,EMM12,EMM15,ES,KP}).
 Why is this? Again our Hopf algebraic framework offers an answer: $x^{\gamma(G)/2}   R_G(x+1,y,1/\sqrt{xy})$ is the canonical Tutte polynomial of ribbon graphs, whereas $R_G(x,y,z)$ is the canonical Tutte polynomial of vertex partitioned ribbon graphs. This suggests that one should look for evaluations and results for $R_G(x,y,z)$ in the setting of vertex partitioned ribbon graphs, since restricting to ribbon graphs alone corresponds to the polynomial  $x^{\gamma(G)/2}   R_G(x+1,y,1/\sqrt{xy})$. 
\end{remark}

\subsection{The Penrose polynomial as a Tutte polynomial}\label{ss.penrose}
We will now illustrate that graph polynomials that are not traditionally regarded as being ``Tutte polynomials'' arise as canonical Tutte polynomials of minor systems.  
 In Section~\ref{s.BRdm} we saw that the 2-variable Bollob\'as-Riordan arises as the Tutte polynomial of delta-matroids and the usual minor operations of deletion and contraction. However, delta-matroids have a third minor operation arising from loop complementation (see \cite{BH11}).  We will examine what happens when we change our notions of deletion and contraction to incorporate the additional minor operation. In particular, we will show that the Penrose polynomial arises in this setting.  

The Penrose polynomial $P_G(\lambda)$  was defined implicitly by
Penrose in~\cite{Pe71} for plane graphs (see also \cite{Aig97,Aig00}). 
  In this section, however, we will focus on matroidal definitions of the Penrose polynomial, discussing its graphical form in Section~\ref{s.penrg}. It was defined for binary matroids by  Aigner and Mielke in~\cite{AM00}. For a binary matroid $M$ with rank function $r$, the \emph{Penrose polynomial} is 
\begin{equation}\label{pdefmq} P_M(\lambda) : = \sum_{X\subseteq E}  (-1)^{|X|} \lambda^{\dim(B_M(X))},  \end{equation}
where $B_M(X)$ is the binary vector space formed of the incidence vectors of the sets in the collection
$\{ A \in \mathcal{C}(M) \mid A\cap X \in \mathcal{C}^*(M)\}$.
Brijder and Hoogeboom defined the Penrose polynomial in greater generality for vf-safe delta-matroids in \cite{BH12}.

Following  Brijder and Hoogeboom \cite{BH11}, let  $D=(E,\mathcal{F})$ be a delta-matroid  and  $e\in E$.
Then $D+e$ is defined to be the pair $(E,\mathcal{F}')$ where
$ \mathcal{F}'= \mathcal{F} \triangle \{ F\cup e \mid F\in \mathcal{F} \text{ and } e\notin F    \} $.
If $e_1, e_2 \in E$ then $(D+e_1)+e_2 = (D+e_2)+e_1 $, and so  for $A=\{a_1, \ldots , a_n\}\subseteq E$ we can define the {\em loop complementation} of $D$ on $A$,  by $D+A:= D+a_1+\cdots  + a_n$.

In general $D+A$ need not   be a delta-matroid, thus we restrict our attention to a class of delta-matroids that is closed under loop complementation. A delta-matroid $D=(E,\mathcal{F})$ is said to be \emph{vf-safe} if the application of any sequence of twists and loop complementations results in a delta-matroid. The class of
vf-safe delta-matroids is known to be closed under deletion and contraction, and strictly contains the class of binary delta-matroids (see for example \cite{BH12}). 
In particular,  in~\cite{CMNR2} it was shown that  delta-matroids of ribbon graphs are  vf-safe.

Set $d_D:=r(D_{\min})$.
If $X\subseteq E$, then the {\em dual pivot} on $X$, denoted by $D \bar{\ast } X$, is defined by  $D \bar{\ast} X:= ((D\ast X)+X)\ast X$.
The {\em Penrose polynomial} of  $D$, defined by Brijder and Hoogeboom  in \cite{BH12}, is 
\begin{equation}\label{pdefde} P(D;\lambda)  = \sum_{A\subseteq E}  (-1)^{|A|} \lambda^{d_{D\ast E \bar{\ast} A}}.  \end{equation}
It was shown in \cite{BH12} that when the delta-matroid $D$ is a binary matroid, Equations~\eqref{pdefmq} and \eqref{pdefde} agree.

Here we introduce  a  function on delta-matroids by 
\begin{equation}\label{d.dmxi}
\xi(A)  :=  \tfrac{1}{2}( |A|+r_{\max}(D+A)-r_{\max}(D)) = |A|/2+\rho(D+A)-\rho(D).
\end{equation}
With this we define the \emph{2-variable Penrose polynomial}, $\tilde{P}_D(x,y)\in \mathbb{Z}[x^{1/2}, y^{1/2}]$, by
\begin{equation}\label{s3.e10}
\tilde{P}_D(x,y) :=  \sum_{A\subseteq E}  (x-1)^{\xi(E)-\xi(A)}(y-1)^{|A|-\xi(A)}.
\end{equation} 

\begin{proposition}\label{p.pepe}
The Penrose polynomial can be recovered as a specialisation of the 2-variable Penrose polynomial:
\[\left.\tilde{P}_D(x,y)\right|_{\sqrt{x}=1+i\sqrt{\lambda},\sqrt{y}=1-i\sqrt{\lambda}} =  (-1)^{|E|+r_{\max}(D+E)}  \lambda^{\xi(D) - |E|}  P_D(\lambda).\]  
\end{proposition}
\begin{proof}
For $\sqrt{x}=1+i\sqrt{\lambda}$ and $\sqrt{y}=1-i\sqrt{\lambda}$, the definition of $\xi$ gives
\begin{align*}
\tilde{P}_D(x,y) &=  \sum_{A\subseteq E}  (i\sqrt{\lambda})^{2\xi(E)-2\xi(A)}(-i\sqrt{\lambda})^{2|A|-2\xi(A)}
\\&= (-1)^{|E|+r_{\max}(D+E))}  \lambda^{\xi(D) - |E|}    \sum_{A\subseteq E}  (-1)^{|A|)}(\lambda)^{|E|-r_{\max}(D+A)}.
\end{align*}
Next 
\begin{align*}
d_{D\ast E \bar{\ast} A} &= r_{\min}((D^*)\ast A+A\ast A) 
=  r_{\min}((D+A)^*)
=  |E|- r_{\max}(D+A),
\end{align*}
where the first equality is by definition,  the  second uses the twisted duality identities of \cite{BH11,CMNR2} to write $  (D+A)^* =  D+A\ast E =  D \ast A \ast A  +A \ast A^c \ast A   =  D \ast A  \ast A^c  \ast A  +A \ast A   = D \ast E    \ast A  +A \ast A  =(D^*)\ast A+A\ast A$, the third equality follows by looking how duality changes the size of a maximal feasible set.
\end{proof}

For our minor systems we consider  vf-safe delta-matroids, but rather than usual deletion and contraction for delta-matroids, which results in the Bollob\'as-Riordan polynomial, we  use the minor operations $D/e$ and $(D+e)/e$. With these notions of minors, the following result is easily checked.
 \begin{lemma}\label{l.penmin}
 The set of isomorphism classes of vf-safe delta-matroids  forms a minor system where the grading is given by the cardinality of the ground set, ``deletion'' and ``contraction'' are given by   $D/e$ and  $(D+e)/e$ and multiplication is given by direct sum.  \end{lemma}

\begin{definition}\label{pe.d1}
Let  $\mathcal{H}^{pe}$  denote the  Hopf algebra associated with vf-safe delta-matroids via  Lemma~\ref{l.penmin} and Proposition~\ref{p.1}. 
Its coproduct is given by 
$ \Delta_{pe}(D)=\sum_{A\subseteq E}D/A^c \otimes  (D+A)/A$.
\end{definition}

We  need to be able to  recognise when an element $e$ of $D$  has $D/e^c$  isomorphic to $D_c$, $D_o$, or $D_n$. 
Let $D=(E, \mathcal{F})$ be a delta-matroid, and $e\in E$. 
 Then  we say that $e$ is a {\em ribbon dual-loop} if $e$ is a coloop in $D_{\max}$.
  A ribbon dual-loop $e$ is  \emph{orientable} if $e$ is not a coloop in $(D\ast e)_{\max}$,
  and is  \emph{non-orientable} if $e$ is a  coloop in $(D\ast e)_{\max}$.
Observe that  $e$ is an (orientable/non-orientable) ribbon dual-loop in $D$ if and only if $e$ is an (orientable/non-orientable) ribbon loop in $D^*$. Also observe that it can be determined if $e$ is a (orientable/non-orientable) ribbon dual-loop  by looking for its membership  in sets in $\mathcal{F}_{\max}$ and   $\mathcal{F}_{\max-1}$.

\begin{lemma}\label{s2.l10}
 Let $D=(E, \mathcal{F})$ be a delta-matroid, and $e\in E$.   Then   $e$ is an orientable ribbon dual-loop   (is not a ribbon dual-loop,  is a non-orientable ribbon dual-loop, respectively) if and only if  $D/e^c$ is isomorphic to $D_c$  ($D_o$, $D_n$, respectively).
\end{lemma}
\begin{proof}
We start by observing that 
\[ D/e^c = (D \ast e^c) \backslash e^c  =   ([(D \ast e^c) \backslash e^c]\ast e)\ast e =  [(D \ast E) \backslash e^c]\ast e = (D^*|_{e} )^*  . \]
So $D/e^c = D_c$  if and only if $(D^*|_{e} )^* =D_c$ if and only if $D^*|_{e} =D_o$. By Lemma~\ref{s2.l9} this happens if and only if $e$ is an orientable ribbon loop of $D^*$ which happens if and only if $e$ is an  orientable ribbon dual-loop in $D$.
Arguing similarly gives that  $D/e^c = D_o$  if and only if $e$ is  not a ribbon dual-loop in $D$, and  $D/e^c = D_o$  if and only if $e$ is a  non-orientable ribbon dual-loop in $D$.
\end{proof}

\begin{lemma}\label{pe.l2}
Let $D=(E, \mathcal{F})$ be a vf-safe delta-matroid, and $e\in E$. Then 
\[  r_{\max}(D) = \begin{cases}   
 r_{\max}(D+e)    &  \text{if $e$ is an orientable ribbon dual-loop},  
  \\  r_{\max}(D+e)   -1&  \text{if $e$ is   not a ribbon dual-loop},  
    \\  r_{\max}(D+e) +1  &  \text{if $e$ is a non-orientable ribbon dual-loop.}\end{cases}   \]
\end{lemma}
\begin{proof}
If $e$ is  an orientable ribbon dual-loop then it is  a coloop of $D_{\max}$ but not a coloop of $(D\ast e)_{\max}$. If there were any sets  $F\in \mathcal{F}(D)_{\max-1}$ with $e\notin F$  then $\mathcal{F}(D\ast e)_{\max}$ would consist exactly of sets of the  form $F\cup\{e\}$ for these $F$, so $e$ would be a coloop   $(D\ast e)_{\max}$. Thus there is an element $X\in \mathcal{F}(D)_{\max}$ such that  $X\backslash \{e\} \notin \mathcal{F}(D)$ and it follows that $X$ is a maximal  set in $\mathcal{F}(D+e)$. Thus $r_{\max}(D+e) = r_{\max}(D)$. 

If $e$ is   not a ribbon dual-loop then it is not a coloop of $D_{\max}$. Thus there is $F\in \mathcal{F}(D)_{\max}$ such that $e\notin F$ and $F\cup \{e\} \notin \mathcal{F}(D)$. It follows that $F\cup \{e\} \in \mathcal{F}(D+e)$, and that  $r_{\max}(D+e) = r_{\max}(D) +1$.

If $e$ is a non-orientable ribbon dual-loop then, by  Lemma~\ref{s2.l10}, $D/e^c=D_n$. If there was a set  $F\in \mathcal{F}(D)_{\max}$ such that $F\backslash \{e\}\notin \mathcal{F}(D)$ then by contracting  the elements of $F\backslash \{e\}$ first we would get that  $D/e^c=D_c$. Similarly,  if there was a set  $F\in \mathcal{F}(D)_{\max-1}$ such that $F\cup \{e\}\notin \mathcal{F}(D)$ then by contracting  the elements of $F\backslash \{e\}$ first we would get that  $D/e^c=D_o$. Thus  $F\in \mathcal{F}(D)_{\max}$ if and only if $F\backslash \{e\}\in \mathcal{F}(D)_{\max-1}$, and it follows that  $r_{\max}(D+e) = r_{\max}(D) -1$.
\end{proof}

\begin{lemma}\label{pe.l1}
Let $D=(E, \mathcal{F})$ be a vf-safe delta-matroid, and $e\in E$. Then 
\[  \xi(D) = \begin{cases}
   \xi((D+e)/e) +\frac{1}{2} &  \text{if $e$ is an orientable ribbon dual-loop},  
  \\  \xi((D+e)/e)+  1&   \text{if $e$ is   not a ribbon dual-loop},  
    \\  \xi((D+e)/e) &  \text{if $e$ is a non-orientable ribbon dual-loop.}
    \end{cases}   \]
\end{lemma}
\begin{proof} 
We have
\begin{align}\label{pe.l1.e2}
\begin{split}
2(\xi(D) -  \xi((D+e)/e)) =&  |E| + r_{\max}(D+E)- r_{\max}(D) -|E| +1 - r_{\max}([(D+e)/e] +(E\backslash\{e\})) 
\\&  +r_{\max}((D+e)/e)
\\=&    r_{\max}(D+E)- r_{\max}(D) - r_{\max}((D+E)/e)  +r_{\max}((D+e)/e) +1,
\end{split}
\end{align}
where the first equality is by definition, the second uses that  $[(D+e)/e] +(E\backslash\{e\}) = (D+E)/e$. We also have  
\begin{equation}\label{pe.l1.e1}
r_{\max}(D)=\begin{cases}   r_{\max}(D/e)  & \text{if $e$ is a loop}, \\  r_{\max}(D/e)+1 &  \text{otherwise}.    
 \end{cases}
\end{equation}
(This identity follows easily from the definitions, but was also shown in the proof of Lemma~\ref{s2.l7}.)
Furthermore, we claim that $e$ is a loop of $D+e$ if and only if it is also a loop of $D+E$. Assuming this claim for the moment, we can use  
  Equation~\eqref{pe.l1.e1} to eliminate all of the contractions in  \eqref{pe.l1.e2}, giving
\[  2(\xi(D) -  \xi((D+e)/e))  =  r_{\max}(D+e)-    r_{\max}(D) +1. \] 
The result then follows by an application of Lemma~\ref{pe.l2}.

It remains to verify the claim that  $e$ is a loop of $D+e$ if and only if it is  a loop of $D+E$. For this suppose $f\in E$ with $e\neq f$ (if there is no such $f$ then the result is trivially true). Then if $e$ is a loop of $D+e$ it appears in no feasible sets of $D+e$, and so it cannot appear in a feasible set of $(D+e)+f$. By induction it follows that if $e$ is a loop of $D+e$ then it is a loop of $(D+e)+(E\backslash\{e\}) = D+E$. Applying this result to the delta-matroid $D+(E\backslash\{e\})$ gives that if $e$ is a loop of $(D+(E\backslash\{e\})+e)$ then it is a loop of $(D+(E\backslash\{e\})+e)+(E\backslash\{e\}) = D+e$ (since loop complementation is involuntary and commutes on disjoint elements). This completes the proof of the claim and the proof of the lemma.
\end{proof}

For constructing the Tutte polynomial of $\mathcal{H}^{pe}$  we use the same $\delta_{\mathbf{a}}$ as for the Bollob\'as-Riordan polynomial, see Equation~\eqref{e.BRdelta1}.  For uniformity we see that, by applying $\delta\otimes \delta$ to  $\Delta(D)$, where $D$ is over $E=\{e,f\}$ and has feasible sets $\emptyset$, $\{e\}$, and   $\{e,f \}$,  $\delta_{\mathbf{a}}$ is uniform only if $a_1=\sqrt{a_2a_3}$.

\begin{theorem}\label{s3.t1}
The canonical Tutte polynomial of the  Hopf algebra $\mathcal{H}^{pe}$ is the 2-variable Penrose polynomial
\begin{equation}\label{s3.t1.e1}
 \alpha( \mathbf{a}, \mathbf{b}) ( D)  =  x_1^{ \xi(D)}   y_2^{|E|- \xi(D)} \tilde{P}_D\left(\frac{y_1}{x_1}+1, \frac{x_2}{y_2} +1\right),
 \end{equation}
where $\mathbf{a} = (\sqrt{x_1x_2} ,x_1,x_2)$, $\mathbf{b}  =(\sqrt{y_1y_2},y_1,y_2)$, and $E$ is the ground set of the delta-matroid $D$. 
\end{theorem}
\begin{proof}
Applying Theorem~\ref{tm} with $D\del e := D/e$ and $D\con e:=  (D+e)/e$, $r_1(A) := \xi(A)$, and  $r_2(A) :=|A|- \xi(A)$, and using  
 Lemmas~\ref{s2.l10} and \ref{pe.l1} gives $m_{11}=1/2$,   $m_{21}=1$, $m_{31}=0$, $m_{12}=1/2$,   $m_{22}=0$, $m_{32}=1$  (so  $a_1=\sqrt{x_1x_2}$, $a_2=x_1$, and $a_3=x_2$). Thus
\begin{align*}
 \alpha( \mathbf{a}, \mathbf{b}) ( D)  &= y_1^{ \xi(D)}   y_2^{|E|- \xi(D)}  \sum_{A\subseteq E} \left(\frac{x_1}{y_1}\right)^{\xi(A)}\left(\frac{x_2}{y_2} \right)^{|A|-\xi(A)}
 \\&=  x_1^{ \xi(D)}   y_2^{|E|- \xi(D)}  \sum_{A\subseteq E} \left(\frac{y_1}{x_1}\right)^{\xi(D)-\xi(A)}\left(\frac{x_2}{y_2} \right)^{|A|-\xi(A)},
\end{align*}
from which the result follows.
\end{proof}

At this point the reader might ask  what happens if we choose the minor operations   $D/e$ and $(D+e)/e $, rather than $D\ba e$ and $(D+e)/e $. If we do we obtain equivalent polynomials.  
Let  $\mathcal{H}^{pe}$ be the Hopf algebra of Definition~\ref{pe.d1} with coproduct $\Delta_{pe}=\sum_{A\subseteq E}D/A^c \otimes  (D+A)/A$.  Using the ``deletion and contraction'' $D\ba e$ and $(D+e)/e$ instead and proceeding as for $\mathcal{H}^{pe}$ results in a second Hopf algebra of delta-matroids, that we denote $\hat{\mathcal{H}}^{pe}$, with coproduct  $\hat{\Delta}_{pe}=\sum_{A\subseteq E}D\ba A^c \otimes  (D+A)/A$.

\begin{lemma}\label{dual5}
The function  $\ast:  \hat{\mathcal{H}}^{pe} \rightarrow \mathcal{H}^{pe}$ defined by $\ast: D\mapsto D^*$, where $D^*$ is the dual of the delta-matroid $D$ is a Hopf algebra morphism.
\end{lemma}
\begin{proof}
To show that  $ \Delta_{pe}(D^*) = (\hat{\Delta}_{pe} (D))^* $ we use the twisted duality properties of delta-matroids from Brijder and Hoogeboom \cite{BH11} (see also \cite{CMNR2}) and that $D/A=(D\ast A) \ba A$. First
\[ D^*\ba A = (D\ast E) \ba A = ((D\ast A )\ast A^c ) \ba A  =   ((D\ast A ) \ba A)\ast A^c  = (D/A)^*.   \]  
Also,
\begin{multline*} (D^*+A)/A = (((D\ast E)+A)\ast A) \ba A  = ((((D\ast A )\ast A^c )  + A) \ast A) \ba A =     ((((D\ast A ))  + A) \ast A) \ba A )\ast A^c   \\ =  ((((D+ A ))  \ast  A) + A) \ba A )\ast A^c  =  (((D+ A ))  \ast  A) \ba A )\ast A^c = ((D+ A )) / A )\ast A^c  = (D+ A )) / A )^*. \end{multline*} 
Then 
\[ \Delta_{pe}(D^*)  =  \sum_{A\subseteq E}D^*/A^c \otimes  (D^*+A)/A =  \sum_{A\subseteq E}(D\ba A^c)^* \otimes  ((D+A)/A)^*=    (\hat{\Delta}_{pe} (D))^* .\]
The remaining properties are easily verified.
\end{proof}

We use the functions of of Equation~\eqref{e.dm.del} to define selectors.  For $ \hat{\mathcal{H}}^{pe}$ we take  $\hat{\delta}_{\hat{\mathbf{a}}}=a_1\delta_{b} + a_2\delta_{o} +a_3\delta_{n}$, and for $\mathcal{H}^{pe}$ take $\delta_{\mathbf{a}}=a_2\delta_{b} + a_1\delta_{o} +a_3\delta_{n}$. 
 Let $\hat{\alpha}(\hat{\mathbf{a}}, \hat{\mathbf{b}}  )$ denote the canonical Tutte polynomial  associated with   $\hat{\mathcal{H}}^{pe}$, and  $\alpha(\mathbf{a}, \mathbf{b}  )$ denote the canonical Tutte polynomial  associated with   $\mathcal{H}^{pe}$. Then, since $\ast$ is a Hopf algebra morphism, by Lemma~\ref{dual5} and  that $\hat{\delta}_{\hat{\mathbf{a}}} = \delta_{\mathbf{a}} \circ \ast$, we can apply Theorem~\ref{tgen} to get
 $   \hat{\alpha}(\hat{\mathbf{a}}, \hat{\mathbf{b}}  ) (D) =  \alpha(\mathbf{a}, \mathbf{b}  )(D^*)  $.  
Theorem~\ref{tgen} also gives that $\hat{\delta}_{\hat{\mathbf{a}}}$ is uniform if and only if $ \delta_{\mathbf{a}}$ is. Thus we have shown the following. 
\begin{theorem}\label{dual6}
The Tutte polynomial of the  Hopf algebra $\hat{\mathcal{H}}^{pe}$ is the 2-variable Penrose polynomial 
\begin{equation}
 \hat{\alpha}(\hat{\mathbf{a}}, \hat{\mathbf{b}}  ) (D)=  x_1^{ \xi(D^*)}   y_2^{|E|- \xi(D^*)} \tilde{P}_{D^*}\left(\frac{y_1}{x_1}+1, \frac{x_2}{y_2} +1\right),
 \end{equation}
where $\hat{\mathbf{a}} = (x_1,\sqrt{x_1x_2} ,x_2)$, $\hat{\mathbf{b}}  =(y_1,\sqrt{y_1y_2},y_2)$, and $E$ is the ground set of the delta-matroid $D$. 
\end{theorem}
Thus the three minor operations $D\ba e$, $D/e$ and $(D+e)/e$ of delta-matroids only generate two Tutte polynomials: the 2-variable Bollob\'as-Riordan polynomial and the 2-variable Penrose polynomial.

\subsection{The Penrose polynomial for ribbon graphs}\label{s.penrg}
 Let $G$ be a ribbon graph and $A\subseteq E(G)$. The  \emph{partial Petrial}, $G^{\tau(A)}$, of $G$ is the ribbon graph obtained from $G$ by for each edge $e\in A$,  choosing one of the arcs $[a,b]$ where $e$ meets a vertex, detaching $e$ from the vertex along that arc giving two copies of the arc $[a,b]$, then reattaching it but by gluing $[a,b]$ to the arc $[b,a]$ (the directions are reversed). This is shown in  Figure~\ref{c1.edgetwist}.
From   \cite{EMM13}, the \emph{Penrose polynomial} of a ribbon graph (or cellularly embedded graph) $G$ is defined by 
\[  P_G(\lambda)   := \sum_{A\subseteq E(G)}  (-1)^{|A|}\lambda^{f(G^{\tau(A)})}.   \] 
If  $D(G)$ is the  delta-matroid of $G$ then, from \cite{CMNR2}, 
 \begin{equation}\label{e.2p} P_G(\lambda) =  \lambda^{c(G)} P_{D(G)}(\lambda).\end{equation}

\begin{figure}
\centering
\includegraphics[height=10mm]{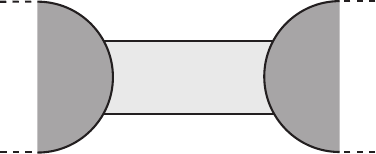}  
\raisebox{4mm}{\includegraphics[width=15mm]{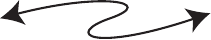}} 
\includegraphics[height=10mm]{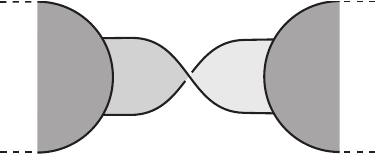}  
\caption{Giving a ``half-twist'' to an edge of a ribbon graph $G$ to form $G^{\tau(e)}$.}
\label{c1.edgetwist}
\end{figure}

Proceeding as in Section~\ref{ss.penrose}, we  define the \emph{2-variable Penrose polynomial} by 

\begin{equation}\label{yisnf}
\tilde{P}_G(x,y) :=  \sum_{A\subseteq E}  (x-1)^{\xi(E)-\xi(A)}(y-1)^{|A|-\xi(A)},
\end{equation} 
where
\[\xi(A)  =  \tfrac{1}{2}\left( |A|+f(G^{\tau(A)})-f(G)\right).
\]
We state the following without proof.
\begin{lemma}\label{l.penrgms}
The set of equivalence classes of ribbon graphs considered up to joins and isomorphism  forms a minor system where the grading is given by the cardinality of the edge set, deletion is given by $G/e$, contraction by $G^{\tau(e)}/e$, and multiplication is given by direct sum.
\end{lemma}
We  identify a ribbon graph with its equivalence class. 
\begin{definition}\label{d.hopperg}
Let  $\mathcal{H}^{per}$  denote the  Hopf algebra associated with ribbon graphs via Lemma~\ref{l.penrgms} and  Proposition~\ref{p.1}. 
Its coproduct is given by 
$
 \Delta_{per}(G)=\sum_{A\subseteq E(G)}G/ A^c  \otimes  G^{\tau(A)}/A$.

\end{definition}

Using $\delta_{b}$, $\delta_{o}$, and  $\delta_{n}$ from    Equation~\eqref{e.rgic}, gives
\begin{equation}\label{e.rgpe1}
 \delta_{\mathbf{a}}= \delta(a_1,a_2,a_3) :=a_1\delta_{b} + a_2\delta_{o} +a_3\delta_{n} .
  \end{equation}
  By considering the ribbon graph that describes a graph with one vertex and two loops on a Klein bottle, it can be seen that 
 $\delta_{\mathbf{a}}$ is not uniform unless $a_1=\sqrt{a_2a_3}$. The following theorem shows that it is uniform if this holds, and identifies the corresponding canonical Tutte polynomial.

\begin{theorem}\label{t.rgpe1}
The 2-variable Penrose polynomial polynomial arises as the canonical Tutte polynomial of the  Hopf algebra $\mathcal{H}^{per}$: 
 \begin{equation}
 \alpha( \mathbf{a}, \mathbf{b}) ( G)  =  x_1^{ \xi(D)}   y_2^{|E|- \xi(D)} \tilde{P}_D\left(\frac{y_1}{x_1}+1, \frac{x_2}{y_2} +1\right),\end{equation}
where $\mathbf{a} = ( \sqrt{x_1x_2}, x_1,x_2)$, $\mathbf{b}  =(\sqrt{y_1y_2} ,y_1,y_2)$, and $E=E(G)$. 
\end{theorem}

To prove the theorem we use the following lemma.

\begin{lemma}\label{t.rgpe.l1}
There is a natural Hopf algebra morphism $\phi: \mathcal{H}^{per}\rightarrow \mathcal{H}^{pe}$ given by $\phi:G\rightarrow D(G)$. 
\end{lemma}
\begin{proof}
Since $D(G\vee H)=D(G)\oplus D(H)$, $\phi$ is well-defined. 
It is easily seen that $\phi$ is multiplicative, and sends the (co)unit to the (co)unit.
It was shown in \cite{CMNR1,CMNR2} that $D(G)/A=D(G/A)$ and  $D(G)+ A=D(G^{\tau(A)})$, giving
$  D(\Delta_{per}(G))  =    \sum_{A\subseteq E} D(G/ A^c)  \otimes  D(G^{\tau(A)}/A)  =  \sum_{A\subseteq E} D(G)/ A^c  \otimes  (D(G)+A)/A  = \Delta_{pe}(D(G)) $.
\end{proof}

\begin{proof}[Proof of Theorem~\ref{t.rgpe1}]
Upon verifying that $ \delta_{b}(G)=  \delta'_{c}(D(G))$, $ \delta_{o}(G)=  \delta'_{o}(D(G))$  and $ \delta_{n}(G)=  \delta'_{n}(C(G))$, where the primed $\delta$'s are those of Equation~\eqref{e.dm.del}, Theorems~\ref{tgen} and~\ref{s2.t3} give
\[ \alpha( \mathbf{a}, \mathbf{b}) ( G) =  x_1^{ \xi(D(G))}   y_2^{|E|- \xi(D(G))} \tilde{P}_{D(G)}\left(\frac{y_1}{x_1}+1, \frac{x_2}{y_2} +1\right),\]
where $E:=E(D(G))$.
It remains to show that for any $\xi_{D(G)}(A) =  \xi_{G}(A)$, but this follows  by  
 Equation~\eqref{d.dmxi}, which gives, 
$\xi(A)  =  |A|/2+\rho(D+A)-\rho(D)$, and by Equation~\eqref{e.t.rgbr1}, which gives $\rho_{D(G)}(A) =  \rho_{G}(A)=   \tfrac{1}{2}\left(|A|-  v(A)+f(A)\right)$.
\end{proof}

\begin{corollary}\label{t.rgpc1}
Let $G$ be a ribbon graph. Then 
\[\left.\tilde{P}_G(x,y)\right|_{\sqrt{x}=1+i\sqrt{\lambda},\sqrt{y}=1-i\sqrt{\lambda}} =  (-1)^{f(G^{\tau(E)})-c(G)}  \lambda^{\xi(G) - |E|-c(G)}  P_{G}(\lambda).\]  
\end{corollary}
\begin{proof}
The result follows from Theorem~\ref{t.rgpe1}, Proposition~\ref{p.pepe} and Equation~\eqref{e.2p} upon noting that $D(G^{\tau(E)})_{\max} =  (C(G^{\tau(E)})^*))^*$, and so $r_{\max}(D(G)+E)=|E|-r((G^{\tau(E)})^*)= |E|-f(G^{\tau(E)})+c(G)$.
\end{proof}

\end{document}